\documentclass[12pt,twoside]{amsart}
\usepackage{pstricks,pst-plot,}
\usepackage{amssymb}
\usepackage{tikz}
\usepackage{float}
\usepackage{verbatim}
\usepackage{amsmath}
\usepackage{bm}
\usepackage[margin=2.2cm, top=3.0cm, bottom=3.5cm]{geometry}
\usepackage[T1]{fontenc}
\usepackage{times}
\usepackage{amssymb,latexsym}
\usepackage{enumerate}
\usepackage[stable]{footmisc}
\usepackage{color}
\usepackage[colorlinks,linkcolor=black,citecolor=black,urlcolor=black]{hyperref}
\usepackage{amsthm}
\usepackage{mathrsfs}
\usepackage{geometry}
\usepackage{enumitem}
\usepackage{cases}
\usepackage{xcolor} 
\usepackage{appendix} 
\usepackage[utf8]{inputenc}

 \usepackage{fancyhdr}
\usepackage{titlesec}
\usepackage{titletoc}
\titleformat{\section}
    {\normalfont\large\bfseries}  
    {\thesection}                 
    {0.5em}                       
    {}
    
\titleformat{\subsection}
    {\normalfont\bfseries}        
    {\thesubsection}              
    {0.5em}                       
    {}
\titleformat{\subsubsection}
    {\normalfont\bfseries}        
    {\thesubsubsection}              
    {0.5em}                       
    {}

\titlespacing*{\section}
    {0pt}                         
    {4ex plus 1ex minus 0.2ex}    
    {2.5ex plus 0.2ex}            

\titlespacing*{\subsection}
    {0pt}                         
    {3.5ex plus 1ex minus 0.2ex}  
    {2ex plus 0.2ex}              

\titlecontents{section}
    [0pt]                          
    {\vspace{0.8ex}\normalfont}    
    {\contentslabel{2.0em}}        
    {\hspace*{-2.0em}}             
    {\titlerule*[0.5pc]{.}\contentspage}  
    
\titlecontents{subsection}
    [3.5em]                        
    {\vspace{0.5ex}\normalfont}    
    {\contentslabel{2.8em}}        
    {\hspace*{-2.8em}}             
    {\titlerule*[0.5pc]{.}\contentspage}  

\makeatletter
\def\@afterheading{%
  \@nobreaktrue
  \everypar{%
    \if@nobreak
      \@nobreakfalse
      \clubpenalty\@M
    \else
      \clubpenalty\@clubpenalty
      \everypar{}%
    \fi}}
\makeatother

\newcommand{\R}{\ensuremath{\mathbb{R}}}

\newcommand{\C}{\ensuremath{\mathbb{C}}}

\newcommand{\bea}{\begin{eqnarray*}}
\newcommand{\eea}{\end{eqnarray*}}

\makeatletter
\newcommand{\sumprime}{\if@display\sideset{}{'}\sum%
            \else\sum'\fi}
\makeatother

\numberwithin{equation}{section}

\newtheorem{theorem}{Theorem}[section]
\newtheorem{proposition}[theorem]{Proposition}

\newtheorem{example}[theorem]{Example}

\newtheorem{corollary}[theorem]{Corollary}
\newtheorem{lemma}[theorem]{Lemma}

\numberwithin{definition}{section} 
\newtheorem{remark}[theorem]{Remark}

\title[Energy estimates and counterexamples to Calder\'on-Zygmund theory]{Energy estimates for level sets of holomorphic functions and  universal counterexamples to Calder\'on-Zygmund theory}

 \author{Yifei  Pan}
\address{Department of Mathematical Sciences, Purdue University Fort Wayne, Fort Wayne, IN 46805-1499, USA}
 \email{pan1@pfw.edu}

 \author{Guokuan  Shao}
\address{School of Mathematics (Zhuhai), Sun Yat-Sen University, Zhuhai, Guangdong 519082, China}
 \email{shaogk@mail.sysu.edu.cn}

\author{ Jianfei Wang }
\address{School of Mathematical Sciences, Huaqiao University, Quanzhou, Fujian 362021, China}
\email{wangjf@mail.ustc.edu.cn}

 \author{Jujie Wu}
\address{School of Mathematics (Zhuhai), Sun Yat-Sen University, Zhuhai, Guangdong 519082, China}
\email{wujj86@mail.sysu.edu.cn}

\thanks{Jujie Wu is the corresponding author.}

\begin{document}

\begin{abstract}
We demonstrate that the failure of $L^1$ regularity in Calder\'on-Zygmund theory is a universal phenomenon: 
every non-constant holomorphic function in $\C^n$ generates a counterexample to the Poisson equation. 
 In order to achieve this goal, we shall establish sharp level-set estimates that link harmonic analysis to the geometry of complex structure through  Hironaka's resolution of singularities and the \L{}ojasiewicz gradient inequality.

\smallskip
  

\end{abstract}

\maketitle

\tableofcontents

\section{Introduction}
 
The Calder\'on-Zygmund theory \cite{GT} is a cornerstone of modern harmonic analysis and the theory of partial differential equations. Central to this theory is   the systematic study of singular integral operators, which are fundamental to understanding the regularity of solutions to elliptic equations. A classic application is the Poisson equation
\[
\Delta u = g
\]
on a domain \(\Omega \subset \mathbb{R}^{n}\). In the classical smooth setting, if \(g\in C^k (\Omega)\), one expects \(u\in C^{k + 2}(\Omega)\).  More generally, within the Sobolev framework, the theory guarantees that if $g\in L_{\mathrm{loc}}^{p} (\Omega)$, for $1<p<\infty$, the second weak derivative of $u$ lies in $ L^p_{\mathrm{loc}} (\Omega)$,  i.e., $u\in W^{2,p}_{\mathrm{loc}}(\Omega)$. At $p=1$, however, this regularity breaks down: there exist $g\in L_{\mathrm{loc}}^{1} (\Omega) $ such that $u\notin W^{2,1}_{\mathrm{loc}}(\Omega)$. 


A standard textbook counterexample is the radial function
\[
u(x,y) = \log (-\log( x^2+y^2))
\]
defined near the origin in \(\mathbb{R}^{2}\). While \(\Delta u\) is integrable in the sense of distributions, its second derivatives fail to be integrable near the origin. Although instructive, such radial examples have long fostered the impression that the \(L^{1}\) failure is a peculiarity of isolated point singularities.

The central thesis of this paper is that the failure of \(L^{1}\) regularity is a far more widespread phenomenon, deeply rooted in the geometry of complex varieties. We demonstrate that every nonconstant holomorphic function \(f(z)\) on a domain \(\Omega \subset \mathbb{C}^{n} \cong \mathbb{R}^{2n}\) yields, via a specific iterated logarithmic transformation, a counterexample to the Calder\'on-Zygmund framework. Concretely, given \(z_{0} \in \Omega\), we define near \(z_{0}\)
\[
u(z) = \log (-\log |f(z) - f(z_{0})|^{2}).
\]
We prove that \(u\) is a distributional solution to \(\Delta u = g\) with \(g\in L_{\mathrm{loc}}^{1}(\Omega)\), yet \(u\notin W_{\mathrm{loc}}^{2,1}(\Omega)\). In these examples, the singular locus is not merely a point, but a complex subvariety of codimension one, the zero set of $f(z)-f(z_0)$.  This greatly expands the catalogue of known counterexamples, progressing from simple radial functions to entire families of holomorphic functions, including those with cusps, nodes, and higher-order vanishing singularities.  Thus, the breakdown of Calder\'on-Zygmund theory at $p=1$ is not an isolated phenomenon but a geometric one, intimately tied to the underlying complex analytic structure. 

To establish the distributional validity of these counterexamples, a detailed analysis of the behavior of \(u\) near the zero set \(Z (f)= \{f(z) = f(z_0)\}\) is required. Without loss of generality, we assume \(f(z_{0}) = 0\). This analysis hinges on precise geometric estimates for the level sets \(\{|f| = \varepsilon\}\) and the sublevel sets \(\{|f| < \varepsilon\}\). The \L{}ojasiewicz gradient inequality provides control on $|\partial f| $ near the singular set, the coarea formula connects volume and surface integral, and Hironaka's resolution of singularities reduces the general one to explicit monomial calculations. 

A substantial portion of the paper is devoted to deriving sharp energy estimates alongside  volume estimates for these sets, the results that are of independent interest in algebraic geometry and complex analysis.

In this paper, we will denote 
$
\partial f = \left( \frac{\partial f}{\partial z_1}, \frac{\partial f}{\partial z_2}, \cdots, \frac{\partial f}{\partial z_n} \right),$
$\nabla = \left( \frac{\partial}{\partial x_1}, \frac{\partial}{\partial y_1}, \frac{\partial}{\partial x_2}, \cdots, \frac{\partial}{\partial x_n}, \frac{\partial}{\partial y_n} \right),
$
where \( z_j = x_j + iy_j \). The Laplacian of \( u \) in \( \mathbb{C}^n \) is given by
$
\Delta u := 4 \sum_{i=1}^n \frac{\partial^2 u}{\partial z_i \partial \bar{z}_i}.
$
Recall that if $f$ is holomorphic then we have the identity 
$
|\nabla|f|| = |\partial f|,
$
which will be used throughout the paper.

Our first main result consists of the following sharp energy estimates:
\begin{theorem}\label{t02}
Let \(U\) be the unit polydisc in \(\mathbb{C}^{n}\) ($n\geq 2$) and \(f(z)\) a holomorphic function in a neighborhood of \(\overline{U}\) with a non empty zero set \(Z(f)\). Then
\begin{enumerate}
    \item
    \[
    \int_{\{z\in U:|f|=\varepsilon \}}|\partial f|\,dS = O(\varepsilon),\qquad \varepsilon \to 0^{+},
    \]
    \item
    \[
    \int_{\{z\in U:|f|< \varepsilon \}}|\partial f|^{2}\,dV = O(\varepsilon^{2}),\qquad \varepsilon \to 0^{+},
    \]
\end{enumerate}
where \(dS\) is the surface measure on the level set and \(dV\) is the Lebesgue measure.
\end{theorem}



We emphasize that the order $O(\varepsilon)$ and $O(\varepsilon^2)$  above are sharp. Furthermore, we obtain estimates below for the measures of the level and sublevel sets, where the orders are sharp as well:
\begin{theorem}\label{tC}
Under the same hypotheses as in Theorem \ref{t02} , there exist constants \(\varepsilon_0 > 0\), \(\gamma \in (0,1]\), and \(\tau \in (0,2]\) such that
\[
\mathcal{H}^{2n - 1}\big(\{z\in U:|f| = \varepsilon \} \big) = O(\varepsilon^{\gamma}),\quad 
\mathrm{Vol}\big(\{z\in U:|f| < \varepsilon \} \big) = O(\varepsilon^{\tau}),\qquad 0<\varepsilon < \varepsilon_0.
\]
Moreover, the condition $\gamma=1$ or $\tau=2$, i.e., that the exponent attains its maximal value, holds if and only if the zero set of the holomorphic function is a  complex submanifold of codimension one. Namely, the following conditions are equivalent:

 1.  \(\gamma = 1\);

2. \(Z(f)\) is a smooth complex submanifold of codimension one;

3. \(\tau = 2\);

 4. \(f\) is locally biholomorphic to a linear function at each point of \(Z(f)\).
\end{theorem}

\begin{remark} We  note that 
\begin{enumerate}
    \item $\gamma$ and $\tau$ are geometric invariants which can be computed via semi-global  \L{}ojasiewicz exponent as defined in Lemma \ref{lm:actuallyused}.
    \item $\gamma<1$ or $\tau<2$ if and only if the singular locus of the variety set $Z(f)$ is non-empty.
\end{enumerate}

\end{remark}

Equipped with these estimates, we construct universal counterexamples to the Calder\'on-Zygmund theory:

\begin{theorem}\label{t01}
Let \(f\) be a nonconstant holomorphic function on a domain \(\Omega \subset \mathbb{C}^n\) ($n  \geq 1$), and let \(z_0 \in \Omega\). Define, in a small neighborhood $U$ of $z_0$,
\[
u(z) = \log\left(-\log |f(z) - f(z_0)|^{2}\right),
\]
and 
\begin{numcases}{g(z)=}
\Delta u (z), &  $z\notin Z(f),$  \\  \nonumber 
  0, & $z\in Z(f)$,
\end{numcases}
where \(Z(f) = \{z\in U:f(z) = f(z_0)\}\). Then

 1.  \(u\in W_{\mathrm{loc}}^{1,2}(U)\) but \(u\notin W_{\mathrm{loc}}^{1,p}(U)\) for any \(p > 2\) near \(Z(f)\);
 
2. \(g\in L_{\mathrm{loc}}^{1}(U)\) but \(g\notin L_{\mathrm{loc}}^{p}(U)\) for any \(p > 1\) near \(Z(f)\);

3.  \(\Delta u = g\) holds in the sense of distributions in $U$;

4.  \(u\notin W_{\mathrm{loc}}^{2,1}(U)\) near \(Z(f)\).\\
Thus, every nonconstant holomorphic function generates, via an iterated logarithm, a counterexample to the \(L^1\)-regularity of the Poisson equation. The singular set is a complex hypersurface, illustrating that the failure of the Calder\'on-Zygmund theory at \(p = 1\) is inherently linked to complex-analytic geometry.
\end{theorem}

As a complementary result, we also show that a modified construction yields functions that do belong to \(W^{2,1}(\Omega)\), demonstrating that the failure at \(p = 1\) is not universal for all singular data.

\begin{theorem}\label{th:maintheorem2}
Under the same setup, define
\begin{numcases}{v(z)=}
\frac{1} {\log ( - \log |f(z)-f(z_0)|^2)}, & $z \notin Z(f);$  \\  \nonumber 
  0, & $z \in Z(f)$.
\end{numcases}
Then

 1. The function $v$ is continuous on $U$ and  \(v\in W_{\mathrm{loc}}^{1,2}(U)\) but \(v\notin W_{\mathrm{loc}}^{1,p}(U)\) for any \(p > 2\) near \(Z(f)\);
 
2. \(t\in L_{\mathrm{loc}}^{1}(U)\) but \(t\notin L_{\mathrm{loc}}^{p}(U)\) for any \(p > 1\) near \(Z(f)\), where 
\[
t(z) = 
\begin{cases}
    \Delta v, & z\notin Z(f),\\
    0,  & z\in Z(f);
\end{cases}
\]

3. The Poisson equation  \(\Delta v = t\) holds in the sense of distributions in $U$;

4. More importantly,  \(v\in W_{\mathrm{loc}}^{2,1}(U)\), but \(v\notin W_{\mathrm{loc}}^{2,p}(U)\) for $p>1$,  near \(Z(f)\).\\

\end{theorem}

\begin{remark}
We   emphasize that Theorem \ref{th:maintheorem2}, in particular, allows one to construct ample examples of $W_{\mathrm{loc}}^{2,1}$ functions whose singular sets have complex codimension $1$. For example, in $\mathbb{C}^2$, the function
\[
    v(z) = \frac{1}{ \log (-\log |z_1^2-z_2^3|^2 )}
\]
belongs to $W_{\mathrm{loc}}^{2,1}$ near the origin. A direct verification of this fact would be highly challenging, if not impossible.
\end{remark}

Our approach synthesizes tools from real algebraic geometry (the \L{}ojasiewicz inequality), geometric measure theory (the coarea formula), and complex geometry (Hironaka's resolution). The explicit calculations for monomials and families such as \(z_{1}^{p} - z_{2}^{q}\) provide concrete examples where the exponents \(\gamma\) and \(\tau\) attain a full range of values, reflecting the geometry of the singularities.


The paper is organized as follows: Section 2 develops the energy and volume estimates for level sets. Section 3 constructs the counterexamples and proves Theorems 1.3 and 1.4. Section 4 discusses the related weighted singular exponents and their sharpness. Appendices provide a volume formula for graphs of holomorphic functions, review Hironaka's theorem and illustrate the resolution space.

\section{Energy and volume estimates for (sub)level sets of holomorphic functions}
 In this section, we present several technical lemmas and establish refined energy estimates for the special case of monomials, which will serve as the foundation for proving Theorems \ref{t02} and  \ref{tC}.
\subsection{Regularity of small level sets} 
In this subsection, we establish a regularity result for the level sets of holomorphic functions. We begin by recalling the \L{}ojasiewicz inequality \cite{Fee}.
\begin{lemma}\label{l2203}
If $\phi:U\rightarrow\mathbb{R}$ is a real analytic function with $\phi(0)=0$, $\nabla \phi(0)=0$,  where $0\in U$ is a domain in $\mathbb{R}^n$, then there exist a constant $C>0$ and an exponent $\beta\in [\frac{1}{2},1)$ such that
\begin{equation*}
|\nabla \phi|\geq C|\phi|^\beta,
\end{equation*}
in a neighborhood of $0$.
\end{lemma}

A brief remark is in order regarding the \L{}ojasiewicz inequality.  In the classical literature, the exponent was frequently cited in the range   $\beta \in (0,1)$, until Feehan \cite{Fee} rigorously proved that $\beta$ must actually satisfy $\beta\geq \frac 1 2 $. The value $\beta = \frac 1 2 $ is   sharp and corresponds to the case of a Morse function. As will be seen later, the lower bound $\beta \geq \frac 12$ turns out to be crucial in our holomorphic setting, without which our work would not proceed as of now.

As a direct consequence of Lemma \ref{l2203}, we obtain a version of the \L{}ojasiewicz inequality specifically tailored for holomorphic functions. We will use $\partial f$ instead of $\nabla f$ to reflect the complex analytic context.
\begin{lemma}\label{l2204}
Let $f:U\rightarrow\mathbb{C}$ be a nonconstant holomorphic function with $f(0)=0$, where $0\in U$ is a domain in $\mathbb{C}^n$. Then, in a neighborhood of $0$, we have
\begin{equation*}
|\partial f|\geq C|f|^{\alpha},
\end{equation*}
for some constant $C>0$, where $0\leq \alpha<1$.
\end{lemma}
Note that, unlike in Lemma \ref{l2203},   we do not assume $\partial f(0)=0$ here.
\begin{proof}
    Consider the function defined by  
\[
\phi(z) = |f(z)|^2.
\]  
Applying Lemma \ref{l2203} to \(\phi(z)\), since $\phi(0)=0$ and $\nabla \phi(0) =0$, we obtain the following gradient estimate in a neighborhood of $0$ 
\[
|\nabla \phi(z)| \geq C|\phi(z)|^{\beta}, 
\]  
where $\frac{1}{2} \leq \beta < 1.$
Consequently,  
\[
|f(z)|   |\nabla |f(z)|| \geq C|f(z)|^{2\beta}.
\]  
If \(|f(z)| \neq  0\), then dividing by  $|f(z)|$ yields
\[
\bigl| \nabla |f(z)| \bigr|  \geq C|f(z)|^{2\beta-1}.
\]  
Since $\bigl| \nabla |f(z)| \bigr| = |\partial f(z)|$ when $f(z)\neq 0$, it follows that  

\[
|\partial f(z)| \geq C|f(z)|^{\alpha}, \quad \text{with } 0 \leq \alpha=2\beta-1 < 1.
\]
For $f(z)=0$, the proof is complete by continuity.
\end{proof}
Note that in the one dimensional case,   Lemma \ref{l2204} can be proven directly without invoking Lemma \ref{l2203}.
Indeed,  if $k$ is the vanishing order of $f$ at $z=0$, then  $|f'| \geq C|f|^{\alpha}$ near $0$ with $\alpha=\frac{k-1}{k}$.

The following lemma is fundamental for  studying   the  level sets defined by the modulus of a holomorphic function. It provides a   stronger conclusion  than Sard's theorem, as the latter only guarantees    for almost every $\varepsilon$.

\begin{lemma}[Regularity of small level sets]\label{l2201}
Let $ K \subset \mathbb{C}^n $ be a compact subset and $ U $ be an open neighborhood of $K$. Let $ f: U \to \mathbb{C} $ be a nonconstant holomorphic function. Let $ Z (f)= \{z \in K: f(z) = 0\} \neq \emptyset$.  Then, there exists a constant $\varepsilon_0 > 0$ such that for all $ z \in K$, if  
$$
0 < |f(z)| < \varepsilon_0,
$$  
then $z$ is not a critical point of the function $|f|^2$.  
Specifically, $ \nabla |f|^2 \neq 0 $ on $\{z \in K: 0 < |f(z)| < \varepsilon_0\}$, and its level set $\{z \in K: |f(z)| = \varepsilon\}$ is a smooth real analytic submanifold for all $ 0 < \varepsilon < \varepsilon_0 $.  
\end{lemma}
\begin{proof}

Define the function $\phi: U \to \mathbb{R}$ by  
$$\phi(z) = |f(z)|^2 = f(z)\overline{f(z)},$$  
since $ f$ is holomorphic, $\phi$ is real analytic. We compute the gradient of $\phi$, identifying the vanishing of the real gradient with the vanishing of the complex derivative 
$$
\nabla \phi(z) = 0 \iff \frac{\partial \phi}{\partial z_j}(z) = 0, \ \   \forall j = 1 \cdots n.
$$  
Using the chain rule, we have  
$$
\frac{\partial}{\partial z_j}(f \bar{f}) = \overline{f(z)} \frac{\partial f}{\partial z_j}(z).
$$  
Thus, $ z $ is a critical point of $\phi$ if and only if  
$$
\overline{f(z)}  \partial f (z) =0.
$$
This occurs in two cases:

1.  $ f(z) = 0 $ (i.e.,  $ z \in Z (f) $);

2.  $ \partial f = 0 $ (i.e., $ z $ is a critical point of the holomorphic map $ f $ with a non-zero value).\\
We must show that  case $2$ cannot occur arbitrarily close to the zero set $ Z (f) $. To this end, we need to  apply the \L{}ojasiewicz inequality locally.

Let $ p \in Z (f) $. Since $\phi$ is real analytic and $ \phi(p) = 0 $, By Lemma \ref{l2203} there exists an open neighborhood of $ W_p \subset U$ of $p$, a constant $ C_p > 0 $,  and $ Q_p \in [\frac{1}{2}, 1)$ such that for all $ z \in W_p $ 
$$|\nabla\phi(z)| \geq C_p |\phi(z) - \phi(p)|^{Q_p} = C_p |\phi(z)|^{Q_p}.$$
Substituting $ \phi = |f(z)|^2 $ we get
$$|f(z)|   |\nabla |f(z)||   \geq C_p |f(z)|^{2Q_p}.$$
For any $ z \in W_p $, if $ f(z) \neq 0 $, the right hand side is strictly positive. Consequently, $ |\nabla\phi(z)| > 0 $. Namely, there are no critical points of $ |f|^2 $ in the punctured subset  $ W_p\backslash Z (f) $. Since the zero set $ Z (f) $ is closed in $ K $, it is compact. The family of open sets $ \{W_p\}_{p \in Z (f)} $ covers $ Z(f)$. By compactness, extracting a finite subcover, there exists an open set $ W = \bigcup_{j=1}^m W_j $ such that $Z (f)\subset W\subset U$.
In particular, we have established  
$$ \nabla |f|^2 \neq 0,  \quad \text{for any}  \quad   z \in W \backslash Z (f).$$  
In order to avoid possible enthalpies outside $W$,  consider the set $K \backslash W$. Since $W$ is open, $K \backslash W$ is compact. Furthermore, because  $Z (f)\subset W$,  $f$ has no zero in $K \backslash W$. Therefore, $|f|$ attains a strictly positive minimum on this set:  
$$\delta = \min_{z \in K \backslash W} |f(z)| > 0.$$  
Choose $\varepsilon_0 = \delta$. Let $z \in K$ be any point such that 
$$0 < |f(z)| < \varepsilon_0.$$  
Suppose $z \in K \backslash W$. Then $|f(z)| \geq \delta = \varepsilon_0$. This contradicts the assumption $|f(z)| < \varepsilon_0$. Therefore, it must be that $z \in W$. Since $|f(z)| > 0$, we conclude that  
$$z \in W \backslash Z (f).$$ 
Based on our previous analysis, 
$$
\nabla |f|^2 \neq 0 \quad \text{on} \quad  W \backslash Z (f).
$$  
Finally, for any $\varepsilon\in (0, \varepsilon_0)$, the level set $\{z \in K : |f| = \varepsilon\}$ contains no critical points. By the implicit function theorem, the level set is a smooth submanifold.

\end{proof}
It follows from Lemma \ref{l2201} that the first integral in Theorem \ref{t02} is well defined, at least for $0<\varepsilon<\varepsilon_0$.

We stress that the constant $\varepsilon_0$ is absolutely necessary in Lemma \ref{l2201}. For example, consider a function $f(z_1,z_2)=z_1^2-z_1+z_2^2$, then $f_{z_1}(z_1,z_2)=f_{z_2}(z_1,z_2)=0$ at $p=(\frac{1}{2},0)$ and $f(p)=-\frac{1}{4}$. The level set $\{|f(z)|=\frac{1}{4}\}$ contains the critical point $p$, which implies the level set is not smooth; hence, for this function,   $\varepsilon_0<\frac{1}{4}$.

Based on the proof of Lemma \ref{l2201}, setting $\alpha= \max\{ \alpha_1, \alpha_2, \cdots, \alpha_m\}$, the following semi-global \L{}ojasiewicz inequality can be obtained directly. This specific inequality will be frequently utilized in the subsequent results of this paper. 

\begin{lemma}\label{lm:actuallyused}
    Let $f$ be a nonconstant holomorphic function in a neighborhood of the unit polydisc $\overline{U}$ in $\C^n$ with $f(0)=0$. Then there exist an exponent $0\leq \alpha <1$, a constant $C>0$ and a radius $\varepsilon_0>0$ such that 
\begin{equation*}
|\partial f|\geq C|f|^{\alpha},  \ \ \text{whenever}  \  z\in U  \ \text{with} \  \ |f(z)| <\varepsilon_0.
\end{equation*}
\end{lemma}

\subsection{The coarea formula and Dirac delta function}
In this subsection, we establish the tools necessary to analyze the energy integrals of holomorphic functions over sublevel sets.

The standard coarea formula states that for a Lipschitz function $ u: \mathbb{R}^n \to \mathbb{R} $ and an integrable function $ f: \mathbb{R}^n \to \mathbb{R} $, the following identity holds
\begin{equation}\label{e241}
\int_{\mathbb{R}^n} f(x)  dx = \int_{-\infty}^{\infty} \left( \int_{u^{-1}(t)} \frac{f(x)}{|\nabla u(x)|}dS_x\right) dt,
\end{equation}
where $ u^{-1}(t) = \{ x \in \mathbb{R}^n : u(x) = t \} $ is the level set of $ u $, and $ dS$ denotes the $(n-1)$-dimensional Hausdorff measure.

We can reinterpret this formula using the Dirac delta function $\delta$, which is not a function but a distribution characterized by the sifting property
$$\int_{-\infty}^{\infty} f(t) \delta(t-a) dt = f(a).$$

A heuristic, yet highly intuitive and computationally effective formulation of the coarea formula \cite{Ho90} is given by
\begin{align}\label{eq:norcoarea}
\int_{\mathbb{R}^n} f(x) \delta(u(x)-t) dx = \int_{u^{-1}(t)} \frac{f(x)}{|\nabla u(x)|} dS_x.
\end{align}
We refer to this as the "Dirac delta  version" of the coarea formula. To unpack this: on the left hand side,  the distribution $\delta(u(x)-t)$  essentially "picks out" the level set $ u(x) = t $,  evaluating to zero everywhere else. Naturally, this integral over $\mathbb{R}^n$ must be understood in the sense of distributions. The right hand side represents the classical integral of $f$ over the level set $u^{-1}(t)$, weighted by the crucial Jacobian factor $\frac{1}{|\nabla u(x)|}$.


To facilitate the practical application of this delta version,  consider a simple example: computing the surface area of a level set $u(x)=t$. One would compute 
$$\text{Surface Area } (u^{-1} (t) )= \int_{\mathbb{R}^n} \delta(u(x)-t) |\nabla u(x)| \, dx,\qquad (\text{area }(u^{-1}(t))).$$
Indeed, returning to the coarea interpretation \ref{eq:norcoarea},  
If we choose $ f(x) = |\nabla u(x)| $, we find 
\begin{align*}
	\int_{\mathbb{R}^n} |\nabla u(x)| \delta (u(x)-t) \, dx& = \int_{u^{-1}(t)} \frac{|\nabla u(x)|}{|\nabla u(x)|} \, dS_x\\
&= \int_{u^{-1}(t)} dS_x.
\end{align*}
The right hand side is exactly the $(n-1)$-dimensional measure of the level set $u^{-1}(t)$, which perfectly corresponds to its surface area. 


\subsection{Energy estimates for the case of monomials} 
We now consider $ U = \{ z \in \C^n : |z_i|<1, i=1,2,\cdots, n \} $ and 
$f(z)= z^A = z_1^{k_1} \cdots z_n^{k_n}$, where $A = (k_1, \cdots, k_n),\ k_j \geq 0$ being non-negative integers. We will prove stronger results for the case of monomials than the general  $O(\varepsilon)$ or $O(\varepsilon^2)$ bounds stated  in Theorem \ref{t02}. 
\begin{proposition}\label{p01}
When $f$ is a monomial,   we obtain the following asymptotic formula as $\varepsilon \to 0$,
\begin{equation*}
\begin{split}
I(\varepsilon) &= \int_{\{z\in U: |f| = \varepsilon\} } |\partial f| dS \sim c \varepsilon,\\
J(\varepsilon) &= \int_{\{z\in U:|f| < \varepsilon\} } |\partial f|^2 dV \sim \frac{c}{2} \varepsilon^2.
\end{split}
\end{equation*}

Here we use $ f \sim g $ as $ x \to 0 $ if $\lim_{x \to 0} \frac{f(x)}{g(x)} = 1$, and $c>0$ is a constant.
\end{proposition}

Before present the full proof of this proposition, we carry out specific computations for low dimensional cases  to build intuition.

\subsubsection{Case $ n = 2$  for $I(\varepsilon) $ }
  By applying the Dirac delta version of the coarea formula \eqref{eq:norcoarea} over bounded domains, we can rewrite $ I(\varepsilon) $ as
$$
I(\varepsilon) = \int_{  \{ z\in U:|f| = \varepsilon \}} |\partial f| dS= \int_{U} \big |\nabla |f|\big |^2 \delta (|f| - \varepsilon) dV.
$$
Since $ f $ is holomorphic, we have $ \big|\nabla |f| \big| =|\partial f| $. Then the formula becomes
$$
I(\varepsilon) = \int_{U} |\partial f|^2 \delta (|f| - \varepsilon) dV.
$$
Using polar coordinates $z_1 = r_1 e^{i \theta_1},   z_2 = r_2 e^{i \theta_2}$,
if $f(z)=z_1^kz_2^l$, then $ |f| = r_1^{k} r_2^{l} $, and
$$
\frac{\partial f}{\partial z_1} = k z_1^{k-1} z_2^l, \quad \frac{\partial f}{\partial z_2} = l z_1^{k} z_2^{l-1}.
$$
Consequently, we obtain
\begin{align*}
	|\partial f|^2 &= \left| \frac{\partial f}{\partial z_1} \right|^2 + \left| \frac{\partial f}{\partial z_2} \right|^2 = k^2r_1^{2k-2}r_2^{2l}+l^2r_1^{2k}r_2^{2l-2}\\
&= |f|^2 \left( \frac{k^2}{r_1^2} + \frac{l^2}{r_2^2} \right).
\end{align*}
Substituting these expressions into $I(\varepsilon)$, we have
\begin{align*}
	I(\varepsilon) &= (2\pi)^2 \int_0^1 \int_0^1 r_1 r_2 |\partial f|^2 \delta(r_1^k r_2^l - \varepsilon) dr_1 dr_2\\
&= (2\pi)^2 \int_0^1 \int_0^1 [k^2 r_1^{2k-1} r_2^{2l+1} + l^2 r_1^{2k+1} r_2^{2l-1}] \delta(r_1^k r_2^l - \varepsilon) dr_1 dr_2.
\end{align*}
We perform a further change of variables: Let $u = r_1^k$ and $ v = r_2^l.$ Then $r_1 = u^{1/k}, r_2 = v^{1/l},$ and the associated Jacobian is
$$
dr_1 dr_2 = \frac{1}{kl}u^{1/k-1} v^{1/l-1} dudv.
$$
The integral $I(\varepsilon)$ then becomes
\begin{align} \label{eq:ivar}
I(\varepsilon) = (2\pi)^2 \int_0^1 \int_0^1 \left[ \frac{k}{l} uv^{1+\frac{2}{l}} + \frac{l}{k} u^{1+\frac{2}{k}} v \right] \delta(uv - \varepsilon) dv du.
\end{align}

To rigorously evaluate this delta integral, we rely on a general property of the Dirac delta distribution.
\begin{lemma}\label{l1}
If $g : \mathbb{R} \to \mathbb{R}$ is continuously differentiable  function with a unique root $v_0$, and $g'(v_0) \neq 0$, then
$$
\delta(g(v)) = \frac{1}{|g'(v_0)|} \delta(v-v_0)
$$
in the sense of distributions.
\end{lemma}

This lemma means that its delta function "picks out" the value $ v=v_0 $ , the factor $\frac{1}{|g'(v_0)|}$ accounts for the change in scale due to the transformation $ u=g(v) $.

\begin{proof}[Proof of Lemma \ref{l1}]
Let $ \varphi(v) $ be a test function compactly supported on $\mathbb{R}$ ( i.e., $ \varphi \in C_0^\infty(\mathbb{R}) $ ) . We evaluate the integral 
\begin{align}\label{eq:shfte1}
I = \int_{-\infty}^{\infty} \delta(g(v)) \varphi(v) dv 
\end{align}
Since $ g $ has a unique root $ v_0 $ with $ g'(v_0) \neq 0 $, $ g $ is invertible in a neighborhood of $v_0$. Let $ v=h(u) $ be the local inverse of $ u=g(v) $.   Then
$$
dv = \frac{du}{|g'(h(u))|},
$$
and $ g(v_0)=0$ corresponds exactly to $u=0$. Substituting this into the integral, we obtain
\begin{align}\label{eq:shfte2}
I = \int_{-\infty}^{\infty} \delta(u) \varphi(h(u)) \frac{1}{|g'(h(u))|} du.
\end{align}
Now, combining $(\ref{eq:shfte1})$ and (\ref{eq:shfte2}), and using the sifting property of the delta function, we find
$$
I = \frac{\varphi(h(0))}{|g'(h(0))|} = \frac{\varphi(v_0)}{|g'(v_0)|}.
$$
Conversely, evaluating the right-hand side distribution against the same test function yields 
\begin{align}\label{eq:shift3}
\int_{-\infty}^{\infty} \frac{1}{|g'(v_0)|} \delta(v-v_0) \varphi(v) dv = \frac{\varphi(v_0)}{|g'(v_0)|},
\end{align}
Because both (\ref{eq:shfte1}) and (\ref{eq:shift3}) are equal to $\frac{\varphi(v_0)}{|g'(v_0)|}$ for an arbitrary test functions $\varphi$, we conclude that the distributions are identical
$$
\delta(g(v)) = \frac{1}{|g'(v_0)|} \delta(v-v_0).
$$
\end{proof}

Now we come back to the evaluation using Lemma \ref{l1}. For a fixed $u$, let $g(v) = uv - \varepsilon$, so $v_0 = \frac{\varepsilon}{u}$ is the unique root and $g'(v) = u$. Thus,
$$
\delta(uv - \varepsilon) = \frac{1}{u} \delta\left(v - \frac{\varepsilon}{u}\right).
$$  
Substituting this back into formula (\ref{eq:ivar}), we get 
$$
I(\varepsilon) = (2\pi)^2 \int_0^1 \int_0^1 \left[ \frac{k}{l} u v^{1+ \frac{2}{l}} + \frac{l}{k} u^{1+ \frac{2}{k}} v \right] \frac{1}{u} \delta\left(v - \frac{\varepsilon}{u}\right) dvdu.
$$  
By the sifting property of $\delta(v - \frac{\varepsilon}{u})$, the delta function forces $v = \frac{\varepsilon}{u}$, and the limit  $v\leq 1$, forces $u \geq \varepsilon$. Therefore, we arrive at 
\begin{align*}
I(\varepsilon)& = (2\pi)^2 \int_\varepsilon^1 \left[\frac{k}{l} u \left(\frac{\varepsilon}{u}\right ) ^{1+ \frac{2}{l}} + \frac{l}{k} u^{1+ \frac{2}{k}}  \left(\frac{\varepsilon}{u}\right) \right]\frac{1}{u} du\\
&= (2\pi)^2 \int_\varepsilon^1 \left[ \frac{k}{l} \varepsilon^{1+ \frac{2}{l}} u^{-1- \frac{2}{l}} + \frac{l}{k} \varepsilon u^{\frac{2}{k}-1} \right] du\\
&= 2\pi^2 \varepsilon \left[ k\left (1 - \varepsilon^\frac{2}{l}\right) + l\left (1 - \varepsilon^\frac{2}{k}\right ) \right].\\
&\sim 2\pi^2 (k+l)\varepsilon, \quad \text{as }\varepsilon\to0.
\end{align*}

\begin{remark}
    
We observe, in particular, $I(\varepsilon)$ is differentiable at $\varepsilon = 0$, and its derivative $I'(\varepsilon)$ is H\"{o}lder continuous over $[0,1]$. Namely,
$I(\varepsilon)$ is not $C^2$ at $\varepsilon = 0$ if $k,l \ge 3$.
Surprisingly, when $n=3$, there is a possible $\log \varepsilon$ in $I(\varepsilon)$ as shown in the following case.

\end{remark}

\subsubsection{Case $n=3$  for $I(\varepsilon)$}
Let 
$f (z)= z_1^{k_1} z_2^{k_2} z_3^{k_3},  k_i \geq 1,\ i=1,2,3.$
Using polar coordinates
$$	z_1 = r_1 e^{i\theta_1},
	z_2 = r_2 e^{i\theta_2},
  	z_3 = r_3 e^{i\theta_3},$$
we get
\begin{align*}
	|\partial f|^2
	&= \left| \frac{\partial f}{\partial z_1} \right|^2
	+ \left| \frac{\partial f}{\partial z_2} \right|^2
	+ \left| \frac{\partial f}{\partial z_3} \right|^2\\
	&= k_1^2 r_1^{2k_1-2} r_2^{2k_2} r_3^{2k_3}
	+ k_2^2 r_1^{2k_1} r_2^{2k_2-2} r_3^{2k_3}
	+ k_3^2 r_1^{2k_1} r_2^{2k_2} r_3^{2k_3-2},
\end{align*}
and by substituting into formula (\ref{eq:norcoarea}) and simplifying, we have
\begin{align*}
	I(\varepsilon)
	&= \int_{\{z\in U: |f|=\varepsilon\} } |\partial f|\, dS 
	= \int_U |\partial f|^2\, \delta(|f(z)| - \varepsilon)\, dV\\
	&= (2\pi)^3 \int_0^1\int_0^1\int_0^1 r_1 r_2 r_3 |\partial f|^2\,
	\delta(r_1^{k_1} r_2^{k_2} r_3^{k_3} - \varepsilon)\,
	dr_1 dr_2 dr_3\\
	&= (2\pi)^3 \int_0^1\int_0^1\int_0^1 \Bigl[
	k_1^2 r_1^{2k_1-1} r_2^{2k_2+1} r_3^{2k_3+1}
	+ k_2^2 r_1^{2k_1+1} r_2^{2k_2-1} r_3^{2k_3+1} \\
	&\qquad\qquad + k_3^2 r_1^{2k_1+1} r_2^{2k_2+1} r_3^{2k_3-1}\Bigr]\cdot \delta(r_1^{k_1} r_2^{k_2} r_3^{k_3} - \varepsilon)\,
	dr_1 dr_2 dr_3.
\end{align*}
Using a change of variables $u = r_1^{k_1},\ v = r_2^{k_2},\ w = r_3^{k_3},$ we deduce that
\begin{align*}
	I(\varepsilon)
	&= (2\pi)^3 \frac{1}{k_1 k_2 k_3}
	\int_0^1\int_0^1\int_0^1 \Bigl[
	k_1^2 u v^{1 + \frac{2}{k_2}} w^{1 + \frac{2}{k_3}}
	+ k_2^2 u^{1 + \frac{2}{k_1}} v w^{1 + \frac{2}{k_3}} \\
	&\qquad\qquad\qquad\qquad
	+ k_3^2 u^{1 + \frac{2}{k_1}} v^{1 + \frac{2}{k_2}} w
	\Bigr]\,
	\delta(uvw - \varepsilon)\, du\, dv\, dw.
\end{align*}

Now we are applying Lemma \ref{l1} to compute the above integral. Since 
$$ uvw = \varepsilon \Leftrightarrow w = \frac{\varepsilon}{uv} ,\ w \in [0,1] \Leftrightarrow 0 \leq \frac{\varepsilon}{uv} \leq 1\Leftrightarrow uv \geq \varepsilon,$$ given $ u, v $ fixed,  
we have
$$ g(w) = uvw - \varepsilon,\  g'(w) = uv, \ w_0 = \frac{\varepsilon}{uv},\  g(w_0) = 0. $$  
By Lemma \ref{l1},  
$$
\delta (g(w)) = \frac{1}{|g'(w_0)|} \delta (w-w_0) = \frac{1}{uv} \delta \left (w - \frac{\varepsilon}{uv}\right).
$$
Substituting into $ I(\varepsilon) ,$  we get
\begin{align*}
	I(\varepsilon)
	&= (2\pi)^3 \frac{1}{k_1 k_2 k_3}
	\iint_{uv \geq \varepsilon,\ 0 \le u,v \le 1}
	\Bigl[
	k_1^2\, u v^{1+\frac{2}{k_2}} \left(\frac{\varepsilon}{uv}\right)^{1+\frac{2}{k_3}}
	\\
	&\qquad+ k_2^2\, u^{1+\frac{2}{k_1}} v \left(\frac{\varepsilon}{uv}\right)^{1+\frac{2}{k_3}}+ k_3^2\, u^{1+\frac{2}{k_1}} v^{1+\frac{2}{k_2}} \frac{\varepsilon}{uv}
	\Bigr]
	\frac{1}{uv}\, du\, dv
	\\[0.5em]
	&= \frac{(2\pi)^3}{k_1k_2k_3} \, \Bigl[
	 k_1^2 \varepsilon^{1+\frac{2}{k_3}}\iint_{uv \ge \varepsilon}
	u^{-1-\frac{2}{k_3}}\, v^{\frac{2}{k_2}-\frac{2}{k_3}-1}
	\, du\, dv
	\\
	&\qquad
	+ k_2^2\, \varepsilon^{1+\frac{2}{k_3}}
	\iint_{uv \ge \varepsilon}
	u^{\frac{2}{k_1}-\frac{2}{k_3}-1}\, v^{-1-\frac{2}{k_3}}
	\, du\, dv
	\\
	&\qquad
	+ k_3^2\, \varepsilon
	\iint_{uv \ge \varepsilon}
	u^{\frac{2}{k_1}-1}\, v^{\frac{2}{k_2}-1}
	\, du\, dv \Bigr].
\end{align*}
A direct computation gives
\begin{align*}
	\int_{uv \ge \varepsilon}&
u^{-1-\frac{2}{k_3}}\, v^{\frac{2}{k_2}-\frac{2}{k_3}-1}\, du\, dv\\
&=
\begin{cases}
	\dfrac{k_2 k_3}{4}
	\left[
	\varepsilon^{-\frac{2}{k_3}}
	-\varepsilon^{-\frac{2}{k_2}}
	+\dfrac{k_3}{k_2 k_3 - k_2}\,
	\varepsilon^{\frac{2}{k_2}-\frac{2}{k_3}}
	\right], & \text{if } k_2 \ne k_3, \\[1em]
	\dfrac{k_2^{2}}{4}
	\left( \varepsilon^{-\frac{2}{k_3}} - 1 \right)
	+ \dfrac{k_3}{2} \ln \varepsilon, & \text{if } k_2 = k_3 .
\end{cases}
\end{align*}
By symmetry,
\begin{align*}
\int_{uv \ge \varepsilon}&
u^{-1-\frac{2}{k_3}}\, v^{\frac{2}{k_1}-\frac{2}{k_3}-1}\, du\, dv\\
&=
\begin{cases}
	\dfrac{k_1 k_3}{4}
	\left[
	\varepsilon^{-\frac{2}{k_3}}
	-\varepsilon^{-\frac{2}{k_1}}
	+\dfrac{k_1}{k_3 k_1 - k_1}\,
	\varepsilon^{\frac{2}{k_1}-\frac{2}{k_3}}
	\right], & \text{if } k_1 \ne k_3, \\[1em]
	\dfrac{k_3^{2}}{4}
	\left( \varepsilon^{-\frac{2}{k_3}} - 1 \right)
	+ \dfrac{k_3}{2} \ln \varepsilon, & \text{if } k_1 = k_3 .
\end{cases}
\end{align*}
Notice the presence of $\log \varepsilon$ here.
\begin{align*}
\int_{uv \ge \varepsilon}&
u^{\frac{2}{k_1}-1}\, v^{\frac{2}{k_2}-1}\, du\, dv\\
&=
\begin{cases}
	\dfrac{k_1 k_2}{4(k_2 - k_1)}
	\left( k_2 - k_1 + k_1 \varepsilon^{\frac{2}{k_1}} - k_2 \varepsilon^{\frac{2}{k_2}} \right),
	& \text{if } k_1 \ne k_2, \\[1em]
	\dfrac{k_1^{2}}{4}
	\left( 1 - \varepsilon^{\frac{2}{k_1}} \right)
	+ \dfrac{k_1}{2}\left( \varepsilon^{\frac{2}{k_1}} \ln \varepsilon \right),
	& \text{if } k_1 = k_2 .
\end{cases}
\end{align*}
Putting all together
$$
I(\varepsilon) \sim C\,\varepsilon \quad \text{as } \varepsilon \to 0.
$$
The significance of the above is that $I(\varepsilon)$ has a limited differentiability at $\varepsilon = 0$ again, at best $C^{1,\alpha}$ only. 

\subsubsection{Case $n=2$ for $J(\varepsilon)$}
To get a feel of the volume case, let's compute  
out $ J(\varepsilon) $ directly when $ n=2 $. Let $ f(z) = z_1^k z_2^l $ in $ \mathbb{C}^2 $, $ k,l \geq 1 $.  
$$
|\partial f|^2 = k^2 |z_1|^{2(k-1)} |z_2|^{2l} + l^2 |z_1|^{2k} |z_2|^{2(l-1)}.
$$
Therefore ,  
$$
J(\varepsilon) = \int_{|z_1|^k|z_2|^l<\varepsilon} k^2 |z_1|^{2(k-1)} |z_2|^{2l} + l^2 |z_1|^{2k} |z_2|^{2(l-1)} \, dV.
$$
By symmetry, it suffices to compute  
$$
J_1(\varepsilon) = \int_{|z_1|^k|z_2|^l<\varepsilon} k^2 |z_1|^{2(k-1)} |z_2|^{2l} \, dV.
$$
Using polar coordinates, we have
\begin{align*}
J_1(\varepsilon) &= (2\pi)^2 \int_{ r_1^k r_2^l < \varepsilon} \left( k^2r_1^{2(k-1)}r_2^{2l}\right)r_1r_2dr_1dr_2\\
&=(2\pi)^2k^2\int_{ r_1^k r_2^l < \varepsilon}r_1^{2k-1}r_2^{2l+1} dr_1dr_2.
\end{align*}
After a change of variables $ u_1 = r_1^k, \, u_2 = r_2^l $, we get
\begin{align*}
J_1(\varepsilon) &= (2\pi)^2\frac{k}{l}\int_{u_1u_2<\varepsilon}u_1u_2^{1+\frac{2}{l}}du_1du_2\\
&= (2\pi)^2 \frac{k}{l} \left( \int_0^\varepsilon\left( \int_0^1u_2^{1+\frac{2}{l}}du_2\right) u_1du_1+\int_\varepsilon^1\left( \int_0^{\frac{\varepsilon}{u_1}}u_2^{1+\frac{2}{l}}du_2\right) u_1du_1 \right)\\
&=(2\pi)^2\frac{k}{l}\left( \frac{l}{4}\varepsilon^2-\frac{l^2}{4(l+1)}\varepsilon^{2+\frac{2}{l}}\right) \\
&=(2\pi)^2\left( \frac{k}{4}\varepsilon^2-\frac{kl}{4(l+1)}\varepsilon^{2+\frac{2}{l}}\right).
\end{align*}
Therefore for
$$
J(\varepsilon) = (2\pi)^2 \left( \frac{k + l}{4} \varepsilon^2 - \frac{k l}{4(l+1)} \varepsilon^{2+\frac{2}{l}} - \frac{k l}{4(k+1)} \varepsilon^{2+\frac{2}{k}} \right)\sim\pi^2(k+l)\varepsilon^2.
$$
We notice again $ J(\varepsilon) $ is only $ C^2 $ at $0 $, but not smoother.


\begin{proof}[Proof of Proposition \ref{p01}]
We first assume that $k_1,\cdots,k_n\geq 1$.
	Using polar coordinates $z_j = r_j e^{i\theta_j}$, $|f|=r_1^{k_1}r_2^{k_2}\cdots r_n^{k_n}$, we get 
	$$\left| \frac{\partial f}{\partial z_j} \right|^2
		= k_j^2\, r_1^{2k_1} \cdots r_j^{2k_j-2} \cdots r_n^{2k_n}
		= |f|^2 \frac{k_j^2}{r_j^2}, $$
and
	$$|\partial f|^2
		= \sum_{j=1}^n \left| \frac{\partial f}{\partial z_j} \right|^2
		= |f|^2 \sum_{j=1}^n \frac{k_j^2}{r_j^2}.$$
	It follows that
	\begin{align*}
		I(\varepsilon)
		&= (2\pi)^n \int_0^1 \cdots \int_0^1 
		\prod_{j=1}^n r_j\, |\partial f|^2\,
		\delta\!\left( \prod_{j=1}^n r_j^{k_j} - \varepsilon \right)
		dr_1\cdots dr_n \\
		&= (2\pi)^n \int_{[0,1]^n}
		\left( \sum_{j=1}^n k_j^2 r_j^{2k_j-1} \prod_{i\neq j,i=1}^n r_i^{2k_i+1} \right)
		\delta\!\left( \prod_{i=1}^n r_i^{k_i} - \varepsilon \right)
		dr_1\cdots dr_n.
	\end{align*}
	A change of variables leads
    \begin{equation*}
        u_j = r_j^{k_j},\qquad 
	r_j = u_j^{1/k_j},\qquad 
	dr_j = \frac1{k_j} u_j^{1/k_j - 1}\, du_j.
    \end{equation*}
	Then the Jacobian is
	$$\prod_{j=1}^n dr_j
	=
	\left( \prod_{j=1}^n \frac{1}{k_j} \right)
	\left( \prod_{j=1}^n u_j^{1/k_j - 1} \right)
	du_1\cdots du_n.$$
	Let
	\begin{align*}
		I_j 
		&= k_j^2 \left ( \prod_{j=1}^n\frac{1}{k_j} \right) \int_{[0,1]^n}u_j^{2 - 1/k_j}\left( \prod_{j=1}^n u_j^{1/k_j - 1} \right)
		\left( \prod_{i\neq j,i=1} u_i^{2 + \frac{1}{k_i}} \right)
		\delta\!\left( \prod_{i=1}^n u_i - \varepsilon \right)
		du_1\cdots du_n \\
		&= C_j \int_{[0,1]^n}u_j\prod_{i\neq j,i=1}^n u_i^{1 + \frac{2}{k_i}}
		\delta\!\left( \prod_{i=1}^n u_i - \varepsilon \right)
		du_1\cdots du_n,
	\end{align*}
	where $C_j=k_j^2\prod_{i=1}^n\frac{1}{k_i}$.
	Let's compute $I_j$ using $\delta$ function in variable $u_j$. Indeed, letting $$g(u_j)=\prod_{i=1}^n u_i - \varepsilon,$$ then
	\[
	g'(u_j)=\prod_{i\neq j,i=1}^n u_i,
	\qquad
	g(u_j)=0 \iff u_j=\frac{\varepsilon}{\prod_{i\neq j} u_i}.
	\]
	By Lemma \ref{l1},
	\[
	\delta(g(u_j))
	= \frac{1}{\prod_{i\neq j} u_i}
	\delta\!\left( u_j - \frac{\varepsilon}{\prod_{i\neq j} u_i} \right).
	\]
	Therefore
	\begin{align*}
		I_j
		&= C_j \int_{\prod_{i\neq j} u_i \ge \varepsilon}
		\frac{\varepsilon}{\prod_{i\neq j} u_i}
		\prod_{i\neq j} u_i^{2+ \frac{2}{k_i}}
		\frac{1}{\prod_{i\neq j} u_i} 
		du_1\cdots \widehat{du_j}\cdots du_n \\
		&= C_j\, \varepsilon
		\int_{\prod_{i\neq j} u_i \ge \varepsilon}
		\prod_{i\neq j} u_i^{\frac{2}{k_i} - 1}
		du_1\cdots \widehat{du_j}\cdots du_n,
	\end{align*}
	here $ \widehat{du_j}$ means that $du_j$ is removed. We conclude that $$I(\varepsilon) = (2\pi)^n \sum_{j=1}^n I_j.$$
    The proof is completed by applying Lemma \ref{l02} below, which establishes the asymptotics of $I_j$.
    The general case where some $k_j=0$ follows from the similar arguments, while dealing with lower dimensional cases.
    \end{proof}

\begin{lemma}\label{l02}
Consider $$\mathscr{I} = \int_{0 \le u_i \le 1, u_1\cdots u_m\geq\varepsilon} u_1^{\alpha_1} u_2^{\alpha_2} \dotsm u_m^{\alpha_m} \, du_1 du_2 \dotsm du_m ,$$
where $\alpha_j>-1$ for all $j=1,\cdots,m$.
Then $$\mathscr{I}  \sim \frac{1}{(\alpha_1 + 1) \dotsm (\alpha_m + 1)},   \ \ \ \  \text{ as }\varepsilon \to 0.$$
\end{lemma}

\begin{proof}
We first make the substitution $u_i = e^{-x_i}$.
Then $du_i = -e^{-x_i} dx_i$,
so 
$$
du_1 \cdots du_m = e^{-(x_1 + \dotsb + x_m)} dx_1 \cdots dx_m, \ \  u_i^{\alpha_i} = e^{-\alpha_ix_i}.
$$
The condition $u_1 \dotsm u_m \ge \varepsilon$ becomes
$$
x_1 + \dotsb + x_m \le L, \quad \text{where } L = -\log \varepsilon .
$$
Since $u_i \in [0,1]$, we have $x_i \in [0,\infty)$, and 
$$
\mathscr{I} = \int_{x_1,\cdots,x_m\geq0,\ x_1 + \dotsb + x_m \le L} e^{-(\beta_1 x_1 + \dotsb + \beta_m x_m)} dx_1 \dotsm dx_m,
$$
where $\beta_i = \alpha_i + 1>0$. Then 
\begin{align*}
\lim_{L\to\infty} \mathscr{I}&= \int_{x_1,\dotsc,x_m \ge 0} e^{-(\beta_1 x_1 + \dotsb + \beta_m x_m)} dx_1 \dotsm dx_m\\
&=   \int_0^\infty e^{-\beta_1 x_1} dx_1 \dotsm \int_0^\infty e^{-\beta_m x_m} dx_m  \\
&=    \frac{1}{\beta_1 \beta_2 \dotsm \beta_m} = \frac{1}{(\alpha_1 + 1) \dotsm (\alpha_m + 1)}.
\end{align*}
The proof is complete.
\end{proof}


\subsection {Continuity of $I(\varepsilon)$}
In this subsection, we demonstrate the continuity of  $I(\varepsilon)$ for $\varepsilon>0$, whose importance will be seen later. Note that we do not claim $I(\varepsilon)$ is continuous at $\varepsilon=0$, even though one could define $I(0)=0$. This is because the set $\{f=0\}$ has dimension $2n-2$ and hence surface measure zero despite possible singularity in the set.

\begin{lemma}\label{l03} 
Let 
$$I(\varepsilon) = \int_{\{z\in U : |f|=\varepsilon\}} |\partial f| \, dS$$
for $0<\varepsilon<\varepsilon_0$.  Then $I(\varepsilon)$ is continuous for $0<\varepsilon<\varepsilon_0$, 
where $U$ is the unit ball, $f$ is holomorphic in a neighborhood of $\overline U$ with $f(0)=0$. 
\end{lemma}

\noindent
 \begin{proof}
	 For every $\varepsilon\in(0,\varepsilon_0)$,  the level set
$S_\varepsilon = \{z\in U : |f|=\varepsilon\}$ is a smooth, compact real analytic submanifold of codimension $1$.
We can construct a smooth diffeomorphism between nearby level sets. This allows us to prove that $I(\varepsilon)$ is a continuous function for $\varepsilon>0$.

First we construct the normal flow. We define a vector field $V$ in the shell region
$U_0=\{z\in \overline{U} : 0<|f|<\varepsilon_0\}$ that flows from one level set to the next at unit rate. Let
$$
V(z) = \frac{\nabla |f|}{|\nabla |f||^2}.
$$
Since $|\nabla |f|| = |\partial f|$ is non-zero in this region, this vector field is well-defined and smooth (real analytic).

Let $\gamma(\tau)$ be an integral curve of $V$, i.e.,
$\dot\gamma(\tau)=V(\gamma(\tau))$.  Compute the rate of change of $|f|$ along the curve 
$$
\frac{d}{d\tau}|f(\gamma(\tau))| = \nabla|f|\cdot\dot\gamma
= \nabla|f|\cdot\frac{\nabla|f|}{|\nabla|f||^2}=1 .
$$
Let $\Phi(z, \tau)$ be the flow generated by $V$.
This is the unique solution to the ordinary differential equation 
\[
\frac{\partial}{\partial \tau} \Phi(z, \tau) = V(\Phi(z, \tau))
\]
with initial condition $\Phi(z, 0) = z$.
$\Phi(z, \tau)$ is called the flow map of $V$.

We claim that the flow moves a point from the level set $S_t$ to $S_{t + \tau}$ in exactly time $\tau$.
Consider the scalar function $h(\tau) = u(\Phi(z, \tau)) = |f(\Phi(z, \tau))|$, where $u(z) = |f(z)|$ for short.
Differentiate with respect to $\tau$:
\[
h'(\tau) = \nabla u(\Phi(z, \tau)) \cdot \frac{\partial \Phi}{\partial \tau}.
\]
Substituting the ordinary differential equation we get that 
\[
h'(\tau) = \nabla u \cdot V = 1.
\]
Integrating this simply yields 
\[
h(\tau) = h(0) + \tau \quad \Longrightarrow \quad |f(\Phi(z, \tau))| = |f(z)| + \tau.
\]
Therefore, if $z \in S_{t_0}$,  then $|f(z)| = t_0$,  and further for any $\tau$, the point $\Phi(z, \tau)$ lies on the level set $S_{t_0 + \tau}$.
This implies the flow maps level sets to level sets linearly with time.
If $z \in S_{t_0}$, then the flow for time $\tau$ moves the point to that in $S_{t_0 + \tau}$.

In order to make change of variables for $I(t)$, we introduce a diffeomorphism according to the flow map $\Phi(z, \tau)$.
Indeed, for a fixed small $\tau$, define the map $\Psi_\tau : S_{t_0} \to S_{t_0 + \tau}$ by
\[
\Psi_\tau(z) = \Phi(z, \tau).
\]
This map has the following properties:

1.  Inverse: the inverse map is simply flowing backwards  $$\Psi_{-\tau}(w) = \Phi(w, -\tau),$$
    by using the semi-group property of flow
    $$
    \Psi _{\tau+s} = \Psi_\tau \circ  \Psi_s.
    $$
    
2. Smoothness: since the vector field $V$ is smooth, the dependence of the solution $\Phi$ on the initial data $z$ is smooth by the standard ordinary differential equation theory.

3.  Bijective: uniqueness of the solutions of ordinary differential equations guarantees this is a bijection.\\
Thus, $\Psi_\tau$ is a diffeomorphism between neighbourhoods of $S_{t_0}$ and $S_{t_0 + \tau}$, viewed as subsets of the open shell in $\mathbb{C}^n$.

Now we are prepared to prove the continuity of $I(t)$.
We want to compute $\lim_{t \to t_0} I(t)$. Let $\tau=t-t_0$,
\[
I(t) = \int_{S_t\cap U} |\partial f(w)| \, dS_t(w).
\]
We parametrize $S_t$ using the map $\Psi_\tau$ from the fixed domain $S_{t_0}$.
Let $w = \Psi_\tau(z)$ for $z \in S_{t_0}$.
The surface measure transforms as $dS_t(w) = J_\tau(z) \, dS_{t_0}(z)$, where $J_\tau(z)$ is the Jacobian determinant of the restriction of $\Psi_\tau$ to the tangent space of $S_{t_0}$.
\[
I(t) = \int_{z \in \Psi_\tau^{-1}(S_t \cap U)} |\partial f(\Psi_\tau(z))| \cdot J_\tau(z) \, dS_{t_0}(z).
\]
We analyze the terms as $\tau \to 0$:

1.  Integrand: $|\partial f(\Psi_\tau(z))| \to |\partial f(z)|$ uniformly as $f$ and $\Psi_\tau$ are smooth.  $J_\tau(z) \to 1$ uniformly (as $\tau = 0$, $\Psi_\tau$ is the identity map, so the determinant of the Jacobian is 1).

2.  Domain of integration: it becomes $D_\tau=\{z\in S_{t_0}:\Psi_\tau(z) \in U\cap S_t \}$. We claim that $D_\tau=S_{t_0}$. It suffices to show that $\Psi_\tau:S_{t_0}\rightarrow S_{t_0+\tau}$ is onto. Given any $y\in S_{t_0+\tau}$, $|f(y)|=t_0+\tau$. Set $z=\Psi_{-\tau}(y)$, then $|f(z)|=|f(\Psi_{-\tau}(y))|=|f(y)|-\tau=t_0$. So $z\in S_{t_0}$ and $\Psi_\tau(z)=\Psi_\tau\bigl(\Psi_{-\tau}(y)\bigr)=y$. Thus $y$ is in the image of $\Psi_\tau$.\\

So
	\[
	I(t_0 + \tau) = \int_{S_{t_0}}  |\partial f(\Psi_\tau(z))| \, J_\tau(z) \, dS_{t_0}(z).
	\]
    

Now we look at the full integrand 
\[
G_\tau(z) =|\partial f(\Psi_\tau(z))| \cdot J_\tau(z).
\]

1.  Pointwise convergence: for a.e. $z$, $G_\tau(z) \to |\partial f(z)|$.

2.  Boundedness: since $f$ is smooth and the domain is compact, $|\partial f|$ and $J_\tau$ are uniformly bounded. Thus $|G_\tau(z)| \le C$.\\
By the Lebesgue dominated convergence theorem,
\[
\lim_{\tau \to 0} \int_{S_{t_0}} G_\tau(z) \, dS_{t_0}(z) = \int_{S_{t_0}} \lim_{\tau\to0}G_\tau(z)dS_{t_0},
\]
then
\[
\lim_{t\to t_0}I(t)=\int_{S_{t_0}\cap U}|\partial f(z)|dS_{t_0}=I (t_0).
\]
Thus,  $I(\varepsilon)$ is continuous for $\varepsilon > 0$.
\end{proof}






\subsection{Proof of Theorem \ref{t02}}
Based on Lemma \ref{l03},  we will combine uniform boundedness on fiber volumes and a modified fundamental theorem of calculus (see Lemma \ref{l271}) to prove Theorem \ref{t02}. First we need the following (cf. Appendix \ref{C}).
\begin{lemma}\label{le:graphvol}
For a holomorphic function \(\psi(z')\) defined in a domain \(D \subset \mathbb{C}^{n-1}\), the graph \(\{(z', \psi(z')) : z' \in D\}\) has \((2n-2)\)-dimensional volume, induced from \(\mathbb{C}^n\),  given by
\[
\operatorname{Vol}(\operatorname{graph} \psi) = \int_D \left(1 + |\nabla_{z'} \psi(z')|^2 \right) dV_{\mathbb{C}^{n-1}}(z').
\]

\end{lemma}

The following is considered of independent interest where Hironaka's theorem is invoked.
\begin{theorem}[Uniform boundedness on fiber volumes]\label{l07}	
Let \(f\) be a nonconstant holomorphic function on a neighborhood of the closed unit ball or  polydisc  $\overline{U}$  in $ \mathbb{C}^n$. Then there exists \(\varepsilon_0 > 0\) such that the $(2n-2)$-dimensional volumes of the fibers \(f^{-1}(w) \cap U\) are uniformly bounded for all \(|w| < \varepsilon_0\), i.e., there exists a constant \(M > 0\) such that
\[
\operatorname{Vol}_{2n-2} (f^{-1}(w) \cap U) \leq M \quad \text{for all}  \ \  |w| < \varepsilon_0.
\]
\end{theorem}

\begin{proof} 
The proof consists of two steps. The first step, the monomial case, and the second step applies Hironaka's resolution of singularity to the general one. We first consider the case for the fiber of a monomial function.
Let 
\[
f (z)= z_1^{a_1} z_2^{a_2} \cdots z_n^{a_n}, 
\]
where $U = \Delta^n$ is a polydisc.We want to estimate the volume of the fiber 
\[
V_w = \{z \in \Delta^n : f(z) = w\}.
\]
Here we assume \(w \neq 0\), whose modulus will be small such that $f^{-1} (w) \cap U\neq \emptyset$.

First we see the intuitive case: \(w = 0\). The set \(f^{-1}(0)\) is exactly the union of the coordinate hyperplanes
\[
\{z_1 = 0\} \cup \{z_2 = 0\} \cup \cdots \cup \{z_{n} = 0\}.
\]
The volume of a coordinate hyperplane restricted to a polydisc is finite. Specifically, it is the sum of the volumes of \((n-1)\)-dimensional polydiscs.

Since volume is continuous with respect to the parameter \(w\) in the sense of geometric measure theory, if the volume is finite at \(w = 0\), then it will remain bounded for small perturbations \(w \neq 0\).
To show this, we start with the simplest case \(n = 2\), $f(z) = z_1 z_2 = w$
inside the bidisc \(\{ |z_1|<1, |z_2|<1 \}\) to illustrate the computation.

We can parametrize the curve \(z_1 z_2 = w\) by the variable \(z_1\).
Since \(z_2 = \frac{w}{z_1}\), and we require \(|z_2|<1\), we must have 
$\left|\frac{w}{z_1}\right| < 1$, which implies $|z_1| > |w|$.
So the fiber is the graph of the function \(g(z_1) = \frac{w}{z_1}\) over the annulus \(|w| < |z_1| < 1\).
By Lemma \ref{le:graphvol}, the volume is given by the integral of 
\begin{equation*}
\begin{split}
\text{Volume} &= \int_{|w| < |z_1| < 1} \left(1 + \left|\frac{dg}{dz_1}\right|^2 \right) dx\,dy \\
&=\int_{|w| < |z_1| < 1} \left(1 +\frac{|w|^2}{|z_1|^4} \right) dx\,dy\\
&=2\pi \int_{|w|}^1 \left(r + \frac{|w|^2}{r^3}\right) dr\\
&=2\pi (1-|w|^2),
\end{split}
\end{equation*}
which is obviously bounded as \(w \to 0\).

Now we are ready to do the general case computation. Without loss of generality, assume \(a_n \geq 1\). We will view the fiber as a graph where \(z_n\) is a function of the other variables \(z' = (z_1, \ldots, z_{n-1})\).

The equation 
\[
z_n^{a_n} = \frac{w}{z_1^{a_1} \cdots z_{n-1}^{a_{n-1}}}
\]
defines a multi-valued function. Let \(P(z') = z_1^{a_1} \cdots z_{n-1}^{a_{n-1}}\). Then
\[
|z_n| = \left| \frac{w}{P(z')} \right|^{1/a_n}.
\]
The fiber consists of \(a_n\) sheets. We can calculate the volume of one sheet and multiply by \(a_n\). The condition that \(z \in \Delta^n\) imposes two constraints on the base variables \(z'\):
\[
\begin{aligned}
& |z_j| < 1 \quad \text{for all } j = 1, \ldots, n-1, \\
& |z_n| < 1 \implies \left| \frac{w}{P(z')} \right| < 1 \implies |P(z')| > |w|.
\end{aligned}
\]
Let
\[
\Omega_w = \{ z' \in \Delta^{n-1} : |z_1|^{a_1} \cdots |z_{n-1}|^{a_{n-1}} > |w| \}.
\]
For a holomorphic graph given by \(z_n = g(z_1, \ldots, z_{n-1})\), Lemma \ref{le:graphvol} states that its volume is  
\[
\text{Vol}(V_w) = \int_{\Omega_w} \left( 1 + \sum_{j=1}^{n-1} \left| \frac{\partial g}{\partial z_j} \right|^2 \right) dV_{z'}.
\]
Differentiate the relation $z_1^{a_1} \cdots z_n^{a_n} = w$
implicitly with respect to \(z_j\)
\[
z_1^{a_1}\cdots a_j z_j^{a_j-1} \cdots z_n^{a_n} + a_n z_n^{a_n-1} \frac{\partial z_n}{\partial z_j} \cdot z_1^{a_1} \cdots z_{n-1}^{a_{n-1}} = 0.
\]
Dividing by the original function \(f(z)=w\), this simplifies neatly to logarithmic derivatives
\[
\frac{a_j}{z_j} + \frac{a_n}{z_n} \frac{\partial z_n}{\partial z_j} = 0,
\]
so
\[
\frac{\partial z_n}{\partial z_j} = -\frac{a_j}{a_n} \frac{z_n}{z_j}.
\]
Thus, the total volume is
\[
\text{Vol}(V_w) = a_n \int_{\Omega_w} \left( 1 + \sum_{j=1}^{n-1} \frac{a_j^2}{a_n^2} \frac{|z_n|^2}{|z_j|^2} \right) dV_{z'}.
\]
Substituting
\[
|z_n|^2 = |w|^{2/a_n} \prod_{k=1}^{n-1} |z_k|^{-2a_k/a_n},
\]
we obtain
\begin{align*}
\text{Vol}(V_w) &= a_n \int_{\Omega_w} dV_{z'} + \sum_{j=1}^{n-1} \frac{a_j^2}{a_n} |w|^{2/a_n} \int_{\Omega_w} \frac{1}{|z_j|^2} \left( \prod_{k=1}^{n-1} |z_k|^{-2a_k/a_n} \right) dV_{z'}\\ \nonumber 
& := I + II
\end{align*}

As \(w \to 0\), the domain \(\Omega_w\) approaches the full polydisc \(\Delta^{n-1}\). Since term \(I\) is increasing in \(|w|\) when $|w|\rightarrow 0$, the monotone convergence theorem implies that \(I\) converges to \(a_n\int_{\Delta^{n-1}} dv_z\), a constant times the volume of \(\Delta^{n-1}\). Consequently, \(I\) is bounded by this volume. For the singular part term II,
we switch to polar coordinates \(r_k = |z_k|\). The angular integrals contribute a factor \((2\pi)^{n-1}\). The condition \(|z_1|^{a_1} \cdots |z_{n-1}|^{a_{n-1}} > |w|\) becomes
$\prod_{k=1}^{n-1} r_k^{a_k} > |w|.$

We consider integrals of the form
\[
I_j = |w|^{2/a_n} \int \cdots \int \frac{1}{r_j^2} \left( \prod_{k=1}^{n-1} r_k^{-2a_k/a_n} \right) r_1 \cdots r_{n-1} \, dr_1 \cdots dr_{n-1},
\]
where the integration domain is defined by $0 < r_k \le 1$ ($k=1,\dots,n-1$) and $\prod_{k=1}^{n-1} r_k^{a_k} > |w|$. We assume $a_j > 0$, $a_n > 0$, and $a_k < a_n$ for all $k$ (to ensure convergence near $r_k = 0$ as $|w| \to 0$).

First, simplify the integrand:
\begin{align*}
\frac{1}{r_j^2} \left( \prod_{k=1}^{n-1} r_k^{-2a_k/a_n} \right) r_1 \cdots r_{n-1}
&= r_j^{-2} \cdot r_j^{-2a_j/a_n} \prod_{k \neq j} r_k^{-2a_k/a_n} \cdot r_j \prod_{k \neq j} r_k \\
&= r_j^{-1 - \frac{2a_j}{a_n}} \prod_{k \neq j} r_k^{1 - \frac{2a_k}{a_n}}.
\end{align*}
Thus,
\[
I_j = |w|^{2/a_n} \int r_j^{-1 - \frac{2a_j}{a_n}} \prod_{k \neq j} r_k^{1 - \frac{2a_k}{a_n}} \, dr_1 \cdots dr_{n-1}.
\]

We integrate first with respect to $r_j$, keeping the other $r_k$ fixed. From the constraint $\prod_{k=1}^{n-1} r_k^{a_k} > |w|$, we have
\[
r_j^{a_j} > \frac{|w|}{\prod_{k \neq j} r_k^{a_k}} \quad \Longrightarrow \quad r_j > \left( \frac{|w|}{\prod_{k \neq j} r_k^{a_k}} \right)^{1/a_j} : =  R_{\text{min}}.
\]
Since also $r_j \le 1$, the inner integral is
\[
\int_{R_{\text{min}}}^1 r_j^{-1 - \frac{2a_j}{a_n}} \, dr_j.
\]
Compute:
\[
\int r_j^{-1 - \frac{2a_j}{a_n}} \, dr_j = -\frac{a_n}{2a_j} r_j^{-\frac{2a_j}{a_n}},
\]
so
\[
\int_{R_{\text{min}}}^1 r_j^{-1 - \frac{2a_j}{a_n}} \, dr_j = \left[ -\frac{a_n}{2a_j} r_j^{-\frac{2a_j}{a_n}} \right]_{R_{\text{min}}}^1 = \frac{a_n}{2a_j} \left( R_{\text{min}}^{-\frac{2a_j}{a_n}} - 1 \right).
\]

Now,
\[
R_{\text{min}}^{-\frac{2a_j}{a_n}} = \left( |w|^{1/a_j} \prod_{k \neq j} r_k^{-a_k/a_j} \right)^{-\frac{2a_j}{a_n}} = |w|^{-2/a_n} \prod_{k \neq j} r_k^{2a_k/a_n}.
\]
Substituting back into $I_j$ gives
\begin{align*}
I_j &= |w|^{2/a_n} \int \frac{a_n}{2a_j} \left( |w|^{-2/a_n} \prod_{k \neq j} r_k^{2a_k/a_n} - 1 \right) \prod_{k \neq j} r_k^{1 - \frac{2a_k}{a_n}} \, dr_k \\
&= \frac{a_n}{2a_j} \int \left( \prod_{k \neq j} r_k - |w|^{2/a_n} \prod_{k \neq j} r_k^{1 - \frac{2a_k}{a_n}} \right) dr_k,
\end{align*}
where the integration over $r_k$ ($k \neq j$) is over the region $0 < r_k \le 1$ and $\prod_{k \neq j} r_k^{a_k} \ge |w|$ (since $R_{\text{min}} \le 1$).

As $|w| \to 0$, the second term inside the integral is of order $|w|^{2/a_n}$ and tends to zero. The first term, $\prod_{k \neq j} r_k$, integrated over the region $\prod_{k \neq j} r_k^{a_k} \ge |w|$, approaches the full cube $(0,1]^{n-2}$ (there are $n-2$ variables $r_k$ with $k \neq j$). Thus,
\[
\lim_{|w| \to 0} I_j = \frac{a_n}{2a_j} \int_0^1 \cdots \int_0^1 \prod_{k \neq j} r_k \, dr_k = \frac{a_n}{2a_j} \prod_{k \neq j} \int_0^1 r_k \, dr_k = \frac{a_n}{2a_j} \left( \frac{1}{2} \right)^{n-2}.
\]
Hence $I_j$ is bounded, and for small $|w|$ it is approximately constant.

We need to bound integrals of the form
\[
I_j = |w|^{2/a_n} \int \cdots \int \frac{1}{r_j^2} \left( \prod_{k=1}^{n-1} r_k^{-2a_k/a_n} \right) r_1 \cdots r_{n-1} \, dr_1 \cdots dr_{n-1}.
\]
Consider the exponents for each \(r_k\): for the specific variable \(r_j\), the exponent is \(-1 - \frac{2a_j}{a_n}\); for other variables \(r_k\) (\(k \neq j\)), the exponent is \(1 - \frac{2a_k}{a_n}\).
So we are integrating
\[
I_j = |w|^{2/a_n} \int \left( r_j^{-1-\frac{2a_j}{a_n}} \right) \prod_{k \neq j} \left( r_k^{1-\frac{2a_k}{a_n}} \right) \, dr.
\]
We first integrate \(r_j\) , the ``bad variable'', while keeping all other \(r_k\) fixed. From the domain constraint $\prod r_k^{a_k}>|w|$, we isolate $r_j$
\[
r_j^{a_j} > \frac{|w|}{\prod_{k \neq j} r_k^{a_k}} \implies r_j > \left( \frac{|w|}{\prod_{k \neq j} r_k^{a_k}} \right)^{1/a_j}.
\]
Let 
\[
R_{\text{min}} = |w|^{1/a_j} \prod_{k \neq j} r_k^{-a_k/a_j}.
\]
The inner integral is
\[
\int_{R_{\text{min}}}^1 r_j^{-1-\frac{2a_j}{a_n}} \, dr_j=\frac{a_n}{2a_j} \left( R_{\text{min}}^{-2a_j/a_n} - 1 \right),
\]
if $a_j\neq 0$. The dominant term is \(R_{\text{min}}^{-2a_j/a_n}\). If $a_j=0$, it reduces to one dimension lower case.

Let us substitute \(R_{\text{min}}\) back in
\[
R_{\text{min}}^{-2a_j/a_n} = \left( |w|^{1/a_j} \prod_{k \neq j} r_k^{-a_k/a_j} \right)^{-2a_j/a_n} = |w|^{-2/a_n} \prod_{k \neq j} r_k^{2a_k/a_n}.
\]
We will see the miracle cancellation. Now we substitute this result back into the full integral \(I_j\) as follows
\begin{equation*}
\begin{split}
I_j &= |w|^{2/a_n} \int \cdots\int \left( |w|^{-2/a_n} \prod_{k \neq j} r_k^{2a_k/a_n} \right) \prod_{k \neq j} r_k^{1-\frac{2a_k}{a_n}} \, dr_{k \neq j}\\
&=\int_0^1 \cdots\int_0^1 \prod_{k \neq j} r_k \, dr_{k \neq j}\\
&=C\prod_{k \neq j} \left[ \frac{r_k^2}{2} \right]_0^1,
\end{split}
\end{equation*}
where $C$ is a constant.

The conclusion is that each \(I_j\) is bounded by a constant independent of \(w\). Since the total volume is a sum of such integrals (plus the non-singular part), the total volume of the fiber \(f^{-1}(w) \cap U\) is uniformly bounded.

Finally,  we apply Hironaka’s resolution of singularities to reduce the general case to the monomial case. We begin by  lifting to the resolved manifold.  To estimate the volume of the fiber $V_w= \{ z\in U: f(z)=w\}$, we consider the standard K$\rm{\ddot{a}}$hler form on $U$, $w_0= \frac{i}{2} \sum_{j=1}^n dz_j \wedge d \overline{z}_j$. The $(2n-2)$-dimensional volume is given by 
\[
\text{Vol}_{2n-2}(V_w) = \int_{V_w} \frac{\omega_{0}^{n-1}} {(n-1)!} .
\]
Let $\widetilde{U}$ be a neighborhood of $\overline{U}$ when $f$ is holomorphic.  By Hironaka’s theorem (see Appendix \ref{Hi}), there exists a complex manifold \(X\) and a proper holomorphic map \(\pi: X \to \widetilde{U}\) such that \(\pi\) is a biholomorphism from $X \setminus \pi^{-1} (Z(f))$ to $\widetilde{U}\setminus Z(f)$, where $Z(f) = \{z\in \widetilde{U}:f(z ) =0\}$. For any $w\neq 0$, the fiber $V_w$ does not intersect $Z(f)$. Thus, $\pi$ induces a biholomorphism between the fiber in the resolution $\widetilde{V}_w = \pi^{-1}(V_w) $ and the fiber in the  base $\overline{V}_w$.
By the change of variables formula, 
\[
\text{Vol}_{2n-2}(V_w) =   \int_{\widetilde{V}_w} \pi^* \left(   \frac{\omega_{0}^{n-1}} {(n-1)!}   \right).
\]

For any point $p$ in the exceptional divisor $E=\pi^{-1}(0)$, Hironaka's theorem provides a coordinate chart $(V,u)$ centered at $p$ such that 
\[
(f \circ \pi)(u) = E(u)\cdot  u_1^{a_1} \cdots u_n^{a_n},
\]
where $E(u)$ is a non-vanishing holomorphic unit (so $E(p) \neq 0$).
To apply the pure monomial results from Section 2.3, we perform a local holomorphic change of coordinate to absorb this unit. Specifically, if $a_1>0$, we define $v_1= u_1 [E(u)]^{\frac{1}{a_1}}$ and $v_j = u_j$ for $j>1$. Because $E$ is non-vanishing, this changes is well-defined in a sufficiently small neighborhood $V_p \subset V$. In these new coordinates $v$, the function becomes a pure monomial 
\begin{eqnarray*}
    (f \circ \pi) (v)  = v_1 ^{a_1}v_2 ^{a_2}\cdots v_n ^{a_n} =: v^A.
\end{eqnarray*}
Since the monomial volume estimates in Section 2.3  were derived  specifically for the unit polydisc $\Delta^n = \{ z\in \C^n: |z_i|<1, i =1,2, \cdots,n \}$. A general coordinate chart $V_p$ on the manifold $X$ is typically not a polydisc. However, we can use an induction and scaling argument to bridge this gap.\\

 \textbf{Containment}.  Since $V_p$ is an open neighborhood of the origin in the $v$-coordinates, there exists a radius $R>0$  such that the polydisc $\Delta_{R}^n = \{z\in \C^n: |z_i|<R, i =1,2,\cdots, n  \}$ is contained in $V_p$. Moreover, because the exceptional divisor $E$ is compact, due to the properness of $\pi$, we can cover $E$ with a finite number of such charts $\{V_\alpha\} _\alpha ^N$. In each chart, the image $\phi_\alpha (V_\alpha)$ in $\C^n$ is bounded and thus contained some polydisc $\Delta_{R_\alpha}^n $ of (possibly  large) radius $R_\alpha$.

\textbf{Comparison of forms}. Since $\pi$ is a smooth holomorphic map, the pull back $\pi^*\omega_0$ is smooth on $\overline{V} _\alpha$. On the compact closure of each local chart $\overline{V} _\alpha$, its coefficients are bounded. Thus, there exists a constant $C_\alpha$ such that $\pi^*\omega_0 \leq C_\alpha \omega_{\rm{Eucl}}$ on $\overline{V} _\alpha$, where $\omega_{\rm{Eucl}}$ is the standard Euclidean metric in the $v$-coordinates.

 \textbf{Monotonicity of volume}. We then estimate 
 \begin{eqnarray}
\int _{\widetilde{V}_w \cap V_\alpha} \pi^*\omega_0^{n-1} &\leq &C_\alpha ^{n-1}  \int _ {\phi_\alpha (V_\alpha) \cap \{ v^A=w\}}\omega_{\rm{Eucl}}^{n-1} \\ \nonumber
& \leq & C_\alpha ^{n-1}    \int _ {\Delta_{R_\alpha}^n \cap \{ v^A=w\}}\omega_{\rm{Eucl}}^{n-1} \\ \nonumber
\end{eqnarray}

Now the volume of a monomial fiber in a polydisc of radius $R$ is equivalent to a rescaled fiber in the unit polydisc.  Letting $v_i=R\widetilde{v}_i$, the equation $v^A=w$ becomes $\widetilde{v}^A = \frac{w}{R^{|A|}}$. As $w\rightarrow 0$, the rescaled target value $\frac{w}{R^{|A|}}$ also tends to zero. The monomial proof earlier shows that the volume remains uniformly bounded by a constant $M(R, A)$ as the target value approaches zero. Hence, for each chart $\alpha$, there exists a constant $M_\alpha$ such that the local volume contribution is bounded for all sufficiently small $w$. \\

Finally, we can assemble the pieces. Because $\pi$ is proper map, the preimage of the closed unit ball $\pi^{-1}(\overline{U})$ is compact in $X$, thanks to $f$ being holomorphic in a neighborhood of $\overline{U}$. For sufficiently small $w$, the fiber $\widetilde{V}_w$ is entirely  contained in  the union of the finitely many  coordinate charts $\cup _{\alpha =1}^N \overline{V}_\alpha$. The total volume is therefore bounded by a finite sum of uniform bounds:
\begin{eqnarray}
\text{Vol}_{2n-2}(V_w) &\leq& \sum _{\alpha =1} ^N \int  _{\widetilde{V}_w \cap V_\alpha}  \pi^*\omega_0^{n-1} \\ \nonumber
& \leq & \sum _{\alpha =1} ^N C_\alpha ^{n-1}   M( R_\alpha, A_\alpha) \leq M.
\end{eqnarray}
The proof is complete.
\end{proof}

Next we need to apply the coarea formula from the image side of this form \cite{EG}.
\begin{theorem}[The general coarea formula]\label{t05}
Let $\Phi : \mathbb{R}^m \to \mathbb{R}^k$ (where $m\ge k$) be a smooth map.  
Let $J_k \Phi(x)$ denote the $k$-dimensional Jacobian of $\Phi$ at $x$.  
For any integrable function $g$ on $\mathbb{R}^m$,  
the coarea formula states:
\[
\int_{\mathbb{R}^m} g(x)\, J_k \Phi(x)\, dV_m(x)
= \int_{\mathbb{R}^k} 
\left( \int_{\Phi^{-1}(y)} g(x)\, dS_x \right) dV_k(y),
\]
where $dV_m$ is the volume measure on the domain,  
$dV_k$ is the volume measure on the image,  
and $dS_x$ is the Hausdorff measure (volume) on the fibers (level sets).
Furthermore, the $k$-dimensional Jacobian is defined as:
\[
J_k \Phi(x)
= \sqrt{ \det \left( D\Phi(x)(D\Phi(x))^T \right) },
\]
where $D\Phi$ is the derivative (or Jacobian matrix) of $\Phi$  
at the point $x$, which is a $k\times m$ matrix.
\end{theorem}


To apply the formula, we must compute the Jacobian factor $J_z f$. View $f$ as a map from $\mathbb{R}^{2n} \to \mathbb{R}^2$, and let $f(z) = u(z) + i v(z)$.
The derivative matrix $Df$ is a $2 \times 2n$ real matrix with rows given by the gradients of the real and imaginary parts:
\[
Df = \begin{pmatrix}
	\nabla u \\
	\nabla v
\end{pmatrix}.
\]
The Jacobian determinant $J_2 f$ is defined as
\[
J_2 f = \sqrt{\det(Df Df^T)}.
\]
Compute the product $Df Df^T$:
\[
Df  Df^T = \begin{pmatrix}
	\nabla u \\
	\nabla v
\end{pmatrix}
\begin{pmatrix}
	\nabla u^T & \nabla v^T
\end{pmatrix}
= \begin{pmatrix}
	\nabla u \cdot \nabla u & \nabla u \cdot \nabla v \\
	\nabla v \cdot \nabla u & \nabla v \cdot \nabla v
\end{pmatrix},
\]
where $\nabla u \nabla u^T = \nabla u \cdot \nabla u$. 

Because $f$ is holomorphic, the Cauchy-Riemann equations imply two crucial geometric facts about the gradients of $u$ and $v$: 

1.  Orthogonality $\nabla u \cdot \nabla v = 0$;
    
2. Equal length $|\nabla u|^2 = |\nabla v|^2 = \frac{1}{2} |\nabla f|_{\mathbb{R}^{2n}}^2 = |\partial f|_{\mathbb{C}^n}^2$.

Substituting these into the matrix:
\[
Df Df^T = \begin{pmatrix}
	|\partial f|^2 & 0 \\
	0 & |\partial f|^2
\end{pmatrix}.
\]
The determinant is $(|\partial f|^2)^2 = |\partial f|^4$. Taking the square root, we have
\[
J_2 f = |\partial f|^2.
\]

We are now in a position to prove Theorem \ref{t02}.
\begin{proof}[Proof of Theorem \ref{t02} \, 2]
The quantity we want to compute is the energy integral
\[
J(\varepsilon) = \int_{V_\varepsilon} |\partial f|^2 \, dV,
\]
where $V_\varepsilon = \{z\in U: |f|<\varepsilon\}$.
Notice that the term $|\partial f|^2$ is precisely the Jacobian $J_2 f$. So we are integrating
\[
\int_U \mathbf{1}_{V_\varepsilon}(z) \cdot J_2 f \, dV(z).
\]
Applying the coarea formula, Theorem \ref{t05} for bounded domain, we transform  into an integral over the image space $\mathbb{C}$ with variable  $w$,
\[
= \int_{\mathbb{C}} \left( \int_{f^{-1}(w) \cap U} \mathbf{1}_{V_\varepsilon}(z) \, dS_w\right) dA(w).
\]

To simplify,  note that the characteristic function $\mathbf{1}_{V_\varepsilon}(z)$ equals $1$ if $|f(z)| < \varepsilon$ and $0$ otherwise. 
On the fiber $f(z)=w$, we have $|f(z)|=|w|$. Consequently, the inner integral equals $1$ if $|w|<\varepsilon$ and $0$ otherwise. The outer integral therefore reduces to an integral over the disc $\{w\in\mathbb{C}:|w|<\varepsilon\}$ 
\begin{equation*}
\begin{split}
J(\varepsilon) &= \int_{ \{ w \in \C: |w| < \varepsilon \} } \left( \int_{f^{-1}(w) \cap U} 1 \, dS_w \right) dA(w)\\
&= \int_{\{w \in \C: |w| < \varepsilon\} } \operatorname{Vol}_{2n-2}(f^{-1}(w) \cap U) \, dA(w).
\end{split}
\end{equation*}


Denote $M(w) = \operatorname{Vol}_{2n-2}(f^{-1}(w)\cap U)$, the $(2n-2)$-dimensional volume of the fiber over $w$.  
By Theorem \ref{l07}, the function $M(w)$ is uniformly bounded for sufficiently small $|w|$; that is, there exists a constant $C>0$ such that $M(w)\le C$.

Now evaluate the integral over the disc $|w|<\varepsilon$ using polar coordinates $w=re^{i\theta}$, so that $dA(w)=r\,dr\,d\theta$:
\begin{align} \label{eq:J_polar}
J(\varepsilon) 
= \int_0^\varepsilon \int_0^{2\pi} M(re^{i\theta}) \, r\, d\theta\,dr 
= \int_0^\varepsilon r \left( \int_0^{2\pi} M(re^{i\theta}) \, d\theta \right) dr.
\end{align}
Immediately we obtain the bound
\[
J(\varepsilon) \le C \int_{|w|<\varepsilon} dA(w) = C\pi\varepsilon^2 = O(\varepsilon^2),
\]
which completes the energy estimate for the volume.
\end{proof}
\begin{remark}
      We stress that despite $J(\varepsilon)=O(\varepsilon^2)$, we can not conclude $J'(\varepsilon)=O(\varepsilon)$ in general.
\end{remark} 



We first apply the coarea formula to the function $u(z)=|f(z)|$.  
Since $f$ is holomorphic, we have $|\nabla u| = |\partial f|$. Hence
\begin{align*}
J(\varepsilon)
&= \int_0^\varepsilon \left( \int_{\{z\in U : |f|=t\}} 
\frac{|\partial f|^2}{|\nabla u|} \, dS \right) dt \\
&= \int_0^\varepsilon \left( \int_{\{z\in U : |f|=t\}} 
|\partial f| \, dS \right) dt,
\end{align*}
where $dS$ denotes the Euclidean surface measure on the level set $S_t = \{z\in U : |f|=t\}$. Writing
\[
I(t) = \int_{\{z\in U : |f|=t\}} |\partial f| \, dS,
\]
we obtain the relation
\begin{align} \label{eq:J_I}
J(\varepsilon) = \int_0^\varepsilon I(t)\, dt.
\end{align}

For the proof of part (1) of Theorem \ref{t02}, we also require the following elementary lemma, which is simply a modified version of the fundamental theorem of calculus.

\begin{lemma}\label{l271}
Let $f \in L^1(0,1)$ and suppose $f$ is continuous on $(0,1]$.  
Define
\[
F(x) = \int_0^x f(t) \, dt, \qquad 0 < x \le 1.
\]
Then $F$ is differentiable on $(0,1]$ and
\[
F'(x) = f(x) \quad \text{for all } x\in(0,1].
\]
(No statement is made about differentiability at $x=0$.)
\end{lemma}

\begin{proof}

Fix $x_0\in(0,1]$ and choose $\delta$ with $0<\delta<x_0$.  
Using the additivity of the Lebesgue integral, we write
\[
F(x) = \int_0^\delta f(t)\,dt + \int_\delta^{x} f(t)\,dt, \qquad x>\delta.
\]
The first term is constant with respect to $x$, while the second term is an integral over a compact interval on which $f$ is continuous. By the classical fundamental theorem of calculus,
\[
\frac{d}{dx}\int_\delta^{x} f(t)\,dt = f(x) \quad \text{for } x>\delta.
\]
Since $\delta$ can be taken arbitrarily small, the equality $F'(x)=f(x)$ holds for every $x\in(0,1]$.
\end{proof}



\begin{proof}[Proof of Theorem \ref{t02} \, (1)]
We now have two representations of $J(\varepsilon)$:
\begin{equation*}
\begin{split}
J(\varepsilon) &= \int_0^\varepsilon I(t)\, dt, \\[1mm]
J(\varepsilon) &= \int_0^\varepsilon r \left( \int_0^{2\pi} M(re^{i\theta})\, d\theta \right) dr .
\end{split}
\end{equation*}

From Theorem \ref{l07}, $M(re^{i\theta})$ is bounded, hence $J(\varepsilon)<\infty$ and consequently $I(\varepsilon)\in L^1(0,1)$.  
Since $I(\varepsilon)$ is continuous on $(0,\varepsilon_0)$ for some $\varepsilon_0>0$, Lemma \ref{l271} gives
\begin{equation} \label{eq:J_prime_I}
J'(\varepsilon) = I(\varepsilon) \qquad \text{for every } 0<\varepsilon<\varepsilon_0.
\end{equation}
This is the first relation. 

On the other hand, the function
\[
g(r) = r\int_0^{2\pi} M(re^{i\theta})\,d\theta
\]
belongs to $L^1(0,1)$ because $M$ is bounded. By the fundamental theorem of calculus for absolutely integrable functions,
\begin{equation} \label{eq:J_prime_M}
J'(\varepsilon) = \varepsilon \int_0^{2\pi} M(\varepsilon e^{i\theta})\, d\theta \qquad \text{for almost every } \varepsilon>0.
\end{equation}
Combining \eqref{eq:J_prime_I} and \eqref{eq:J_prime_M}, we obtain
\begin{eqnarray}\label{eq:ivarespilon}
I(\varepsilon) = \varepsilon \int_0^{2\pi} M(\varepsilon e^{i\theta})\, d\theta \quad \text{for almost every } \varepsilon>0.
\end{eqnarray}
On the other hand, the boundedness of $M$ implies that 
\[
I(\varepsilon) \le C \cdot 2\pi \varepsilon = O(\varepsilon) \quad \text{for almost every} \ \  \varepsilon>0.
\]
Thus, we get that
\[
I(\varepsilon) \le C \cdot 2\pi \varepsilon = O(\varepsilon) \quad \text{for } 0<\varepsilon<\varepsilon_0
\]
by the continuity of $I(\varepsilon)$, which completes the proof of part (1).


\end{proof}

\begin{remark}
 We emphasize that \eqref{eq:J_prime_I} is valid for all $\varepsilon>0$. At $\varepsilon=0$ the statement may not hold, as Lemma \ref{l271} does not provide differentiability at the left endpoint, and indeed $I(\varepsilon)$ could be unbounded as $\varepsilon\to0^+$. Moreover, The estimates $O(\varepsilon)$ and $O(\varepsilon^2)$ in Theorem \ref{t02} remain valid when polydiscs are replaced by balls. We will apply this result later.
\end{remark}

\subsection{Proof of Theorem \ref{tC}}
Since we know by Theorem \ref{t02} that if $f$ is holomorphic in a neighborhood of the polydisc $\overline{U}$, then 
\[
\int_{\{  z\in U: |f|=\varepsilon \}} |\partial f|\, dS = O(\varepsilon)
\qquad 0<\varepsilon<\varepsilon_0.
\]
We will apply this to study the volume estimates of sublevel sets, which gives the proof of Theorem \ref{tC} based on the \L{}ojasiewicz inequality.
 
\begin{theorem}[Surface area estimate of level sets]\label{th:levelsurface}
Let $f$ be a holomorphic function defined in a neighborhood of the polydisc $U \subset \mathbb{C}^n$ centered at the origin, with $f(0) = 0$. Then there exists an exponent $0 < \gamma \leq 1$ such that the $(2n-1)$-dimensional Hausdorff measure of its level sets satisfies
\[
\mathcal{H}^{2n-1}\bigl(\{ z \in U : |f(z)| = \varepsilon \}\bigr) = O(\varepsilon^{\gamma}) \quad \text{as } \varepsilon \to 0^+.
\]
\end{theorem}
\begin{proof}
 Using Lemma \ref{lm:actuallyused}, there exist  \(0 \leq \alpha < 1\)  and $\varepsilon_0$ such that  
\begin{equation}\label{eq:gradient_lower_bound}
|\partial f(z)| \geq C |f(z)|^{\alpha}.
\end{equation}
for \(z \in U\) with \(|f(z)| < \varepsilon_0\).



Applying statement (1) of Theorem \ref{t02}, which provides an $L^1$-norm estimate of $\partial f$ on level sets, we obtain for $0 < \varepsilon < \varepsilon_0$ 
\begin{equation}\label{eq:L1_estimate}
\int_{\{z \in U : |f| = \varepsilon\}} |\partial f(z)| \, d S = O(\varepsilon).
\end{equation}

Combining \eqref{eq:gradient_lower_bound} and \eqref{eq:L1_estimate} yields the chain of inequalities valid for all $0 < \varepsilon < \varepsilon_0$ and $C>0$,  
\begin{align*}
O(\varepsilon) 
&= \int_{\{ z \in U : |f| = \varepsilon \}} |\partial f| \, dS\\
&\gtrsim \int_{\{ z \in U : |f| = \varepsilon \}}   |f|^{\alpha} \, dS \\
&\gtrsim \varepsilon^{\alpha} \int_{\{ z \in U : |f| = \varepsilon \}} dS \\
&\gtrsim \varepsilon^{\alpha} \cdot \mathcal{H}^{2n-1}\bigl( \{ z \in U : |f| = \varepsilon \}\bigr).
\end{align*}

Dividing by $  \varepsilon^{\alpha}$ (for $0 < \varepsilon < \varepsilon_0$) gives
\[
\mathcal{H}^{2n-1}\bigl(\{z \in U: |f| = \varepsilon \}  \bigr) \leq C'_1 \cdot \varepsilon^{1-\alpha}
\]
for some constant $C'_1 > 0$. Setting
\[
\gamma := 1 - \alpha,
\]
we observe that $0 < \gamma \leq 1$ since $0 \leq \alpha < 1$. Therefore, for all sufficiently small $\varepsilon > 0$,
\[
\mathcal{H}^{2n-1}\bigl(\{ z \in U : |f| = \varepsilon \}\bigr) = O(\varepsilon^{\gamma}),
\]
which completes the proof.
\end{proof}

With application to level subsets and level sets of any holomorphic function, we get the following

\begin{theorem}[Volume estimate of sublevel sets]\label{th:volume_sublevel}
Let $f$ be a holomorphic function defined in a neighborhood of the polydisc $\overline{U} \subset \mathbb{C}^n$ centered at the origin, with $f(0) = 0$. Then there exists an exponent $0 < \tau \leq 2$ such that the Euclidean volume of its sublevel sets satisfies
\[
\operatorname{Vol}\bigl(\{ z \in U : |f(z)| < \varepsilon \}\bigr) = O(\varepsilon^{\tau}) \quad \text{as } \varepsilon \to 0^+.
\]
\end{theorem}

\begin{proof}

Again with Lemma \ref{lm:actuallyused}, we choose some  \(0 \leq \alpha < 1\) and $\varepsilon_0$ such that  
\begin{equation}\label{eq:gradient_lower_bound2}
|\partial f(z)| \geq C |f(z)|^{\alpha}.
\end{equation}
for \(z \in U\) with \(|f(z)| < \varepsilon_0\).



We now apply the coarea formula to the real-valued function $|f|$. Since $|f|$ is real analytic away from the zero set of $f$ and its gradient satisfies $|\nabla |f|| = |\partial f|$ almost everywhere, we obtain for $0 < \varepsilon < \varepsilon_0$, 
\begin{align*}
\operatorname{Vol}\bigl(\{ z \in U : |f| < \varepsilon \}\bigr) 
&= \int_{\{ |f| < \varepsilon \} \cap U} dV \\
&= \int_0^\varepsilon \left( \int_{\{ z \in U : |f| = t \}} \frac{1}{|\nabla |f||} \, dS\right) dt \\
&= \int_0^\varepsilon \left( \int_{\{ z \in U : |f| = t \}} \frac{1}{|\partial f|} \, dS \right) dt.
\end{align*}

Using the gradient estimate \eqref{eq:gradient_lower_bound2}, which is valid on each level set $\{|f| = t\}$ for $0 < t < \varepsilon_0$, we have the pointwise bound $1/|\partial f| \lesssim 1/|f|^{\alpha} = t^{-\alpha}$ on these level sets. Therefore,
\begin{align*}
\operatorname{Vol}\bigl(\{ z \in U : |f| < \varepsilon \}\bigr) 
&\lesssim \int_0^\varepsilon \left( \int_{\{ z \in U : |f| = t \}} \frac{1}{t^{\alpha}} \, dS \right) dt \\
&\lesssim   \int_0^\varepsilon \frac{1}{t^{\alpha}} \left( \int_{\{ z \in U : |f| = t \}} dS \right) dt.
\end{align*}
Now apply Theorem~\ref{th:levelsurface}, which gives the area estimate
\[
\mathcal{H}^{2n-1}(\{z \in U : |f| = t\}) = O(t^{1-\alpha}), \qquad t \to 0^+.
\]
Inserting this estimate into the co-area formula yields
\begin{align*}
\operatorname{Vol}\bigl(\{ z \in U : |f| < \varepsilon \}\bigr) 
&\lesssim \int_0^\varepsilon \frac{1}{t^{\alpha}} \; O(t^{1-\alpha}) \, dt \\
&\lesssim O\!\left( \int_0^\varepsilon t^{1-2\alpha} \, dt \right).
\end{align*}
Because $0 \le \alpha < 1$, the exponent satisfies $1-2\alpha > -1$; hence the integral converges.  Evaluating it gives
\[
\int_0^\varepsilon t^{1-2\alpha} \, dt = \frac{\varepsilon^{2-2\alpha}}{2-2\alpha},
\]
and therefore
\begin{equation}\label{eq:volume-est}
\operatorname{Vol}\bigl(\{ z \in U : |f| < \varepsilon \}\bigr) 
= O\bigl(\varepsilon^{2-2\alpha}\bigr) \qquad (\varepsilon \to 0^+).
\end{equation}
 \end{proof}


\begin{proof}[Proof of the equivalent part of Theorem \ref{tC} ]
\noindent
\text{1 $\Rightarrow$ 2 and 4.}  
If $\gamma = 1$, then $\alpha = 0$ and the Łojasiewicz inequality becomes
$|\partial f| \geq C$ on a neighborhood of $Z(f)$.  Since $f$ is holomorphic,
$|\partial f| = |\nabla f|$ (up to a multiplicative constant); therefore
$\nabla f$ never vanishes on $Z(f)$.  Consequently $Z(f)$ is a smooth complex
hypersurface, i.e., a complex submanifold of codimension one.  
Moreover, because $\nabla f \neq 0$, the holomorphic inverse function theorem
provides, near any $p \in Z(f)$, local biholomorphic coordinates 
$(w_{1},\dots ,w_{n})$ such that $f = w_{1}$.  This gives 4.

\text{2 $\Rightarrow$ 1.}  
Assume $Z(f)$ is a complex submanifold of codimension one.  As $f$ defines $Z(f)$
and $Z(f)$ is smooth, the differential $df$ does not vanish at any point of $Z(f)$.
By continuity there exist a neighborhood $V$ of $Z(f)$ and a constant $C>0$
such that $|\partial f| \geq C$ on $V$.  Hence $\alpha = 0$, i.e., $\gamma = 1$.

\text{4 $\Rightarrow$ 1.}  
If locally $f = w_{1}$, then $|\partial f| = 1$ identically, whence $\alpha = 0$
and $\gamma = 1$.

Thus conditions $1-4$ are equivalent.
\end{proof}

\begin{remark}\label{r01}
The condition $f$ being holomorphic in a neighborhood of the unit polydisc (or the ball) $U$ can not be dropped. Consider the Blaschke product in $\mathbb{C}^2$
\begin{equation*}
   f(z_1,z_2)=\prod_{j=1}^\infty\frac{z_1-a_j}{1-\overline{a}_jz_1},
\end{equation*}
where $a_j=1-\frac{1}{j^2}$.
Then $f^{-1}(0)=\cup_{j=1}^\infty\{ (z_1,z_2)\in U: z_1=a_j\}  $. $\{ (z_1,z_2)\in U: |f (z_1,z_2)| = \varepsilon\}$, $\{ (z_1,z_2)\in U: |f (z_1,z_2)|< \varepsilon\}$ could not be determined effectively.
\end{remark}

\section{Counterexamples to Calder\'on-Zygmund theory}
In this section, we construct universal counterexamples to Calder\'on-Zygmund theory, which is the main goal of the paper.
\subsection[Elementary computations for $u$ and $v$]{Elementary computations for $u$ and $v$}
In order to apply our results above to construct counter-examples to Calder\'on-Zygmund theory, we are in needs of elementary computations as follows.

Given a domain $U$ in $\C^n$, $n\geq 1$, denote by $W^{k,p}(U) $ ( resp.  $W_{{\rm loc}}^{k,p}(U) $ ) the Sobolev space of all functions on  $U  $ whose weak derivatives of order $\leq k$ exist and belong to $L^p(U)$ (  resp. $L_{{\rm loc}}^p(U))$, for some $k\in \mathbb{Z}^+,  p\geq 1.$
  
Let $f$ be a holomorphic function in a neighborhood of $U$  and $Z(f)= \{z \in U:  f(z) =0 \}$ be the zero set of $f$. Denote $ \partial _{ z _j } f=   \frac{\partial f}{ \partial   {z _j } }$,  $ \partial _{ \overline{z} _j } f=   \frac{\partial f}{ \partial   {\overline{z} _j } }$. We assume $Z(f) \neq \emptyset$, since our interest is around zeros of $f$.

\smallskip
Let $u=\log ( - \log |f|^2)$. If $n\geq 2$, for $1 \leq i \leq n, 1\leq j \leq n , $ we have when $f\neq 0$
\begin{align*} 
\begin{split}
\frac{\partial u }{\partial z_i }  &= \frac{\partial_{z_i} f}{ f} \frac{1}{ \log (|f| ^2 )  },  \ \ \  \frac{\partial u }{\partial  \overline{z}_i }  = \frac{ \overline{\partial_{z_i} f}}{ \overline{f} }  \frac{1}{ \log (|f| ^2 )  } . \\
 \frac{\partial^2 u  }{\partial z_i \partial z_j  }  &= \frac{\partial_{ z_j}( \partial_{z_i} f)}{ f  \log  |f| ^2    } - \frac{\partial_{z_i}  f  \partial_{z_j} f }{f ^2    \log  |f| ^2  }  -  \frac{\partial_{z_i}  f  \partial_{z_j} f }{f ^2  (  \log  |f| ^2 )^2 }  ;\\
\frac{\partial ^2 u}{ \partial z_i \partial \overline{z}_j}  &= -   \frac{   \partial_{z_i}  f   \overline{\partial_{z_j}  f }}{|f|^2} \frac{1}{(\log|f(z)|^2)^2}; \ \ \  \ \frac{\partial^2 u  }{\partial \overline{z}_i \partial \overline{z}_j  }  =\overline{\frac{\partial^2 u  }{\partial z_i \partial z_j  } }.\\
 \Delta u &=4  \sum_{i=1}^n \frac{\partial^2 u  }{\partial z_i \partial  \overline{z} _i  }  = -4 \frac{\sum_{i=1}^n |\frac{\partial f}{\partial z_i}|^2}{|f|^2( \log |f|^2)^2} = -4\Big|\frac{\nabla |f|}{f}\Big|^2\frac{1}{(\log |f|^2)^2}.
 \end{split}
\end{align*}

Let 
\[
v(z)=
\begin{cases}
\displaystyle \frac{1} {\log ( - \log |f|^2)}, &  z \notin Z(f);   \\  \nonumber 
  0, &  z \in Z(f) .
\end{cases}
\]

Then we have
 
\begin{align*}
\frac{\partial v}{\partial z_i}
    &= -\frac{\partial _{z_i} f}{f} \frac{1}{ \log |f|^2  (\log ( - \log |f|^2))^2}, \\
\frac{\partial v}{\partial \overline{z}_i}
    &= -\frac{\overline{\partial _{ z_i } f}}{\overline{f}} \frac{1}{ \log |f|^2  (\log ( - \log |f|^2))^2};
\end{align*}
and
\begin{align*}
\frac{\partial ^2 v}{\partial z_i \partial z_j}
    &= - \frac{\partial_{ z_j} (\partial_{ z_i} f)}  {f \log (|f| ^2 )}  \frac{1}{  (\log (-\log |f|^2))^2} \\
    &\quad + \frac{\partial _{z_i} f   \partial _{z_j}f}  {f^2} \frac{1}{ \log (|f| ^2 )  (\log (-\log |f|^2))^2} \\
    &\quad +  \frac{\partial _{z_i} f  \partial _{z_j}f}  { f ^2}  \frac{1}{( \log (|f| ^2 ))^2 (\log (-\log |f|^2))^2} \\
    &\quad +2 \frac{\partial _{z_i} f   \partial _{z_j}f}  { f ^2}  \frac{1}{( \log (|f| ^2 ))^2  (\log (-\log |f|^2))^3}; \\[6pt]
\frac{\partial ^2 v}{\partial z_i \partial \overline{z}_j}
    &=  \frac{\partial_{ z_i}f     \overline{ \partial_{ z_j} f} }  {|f|^2 (\log   |f| ^2  )^2}  \frac{1}{  (\log (-\log |f|^2))^2} \\
    &\quad + 2 \frac{\partial_{ z_i} f   \overline{ \partial_{ z_j} f} }  {|f|^2 ( \log (|f| ^2 ))^2  (\log (-\log |f|^2))^3}; \\[6pt]
\frac{\partial ^2 v}{\partial \overline{z}_i \partial \overline{z}_j}
    &= \overline{\frac{\partial ^2 v}{\partial z_i \partial z_j}}; \\[6pt]
\Delta v
    &=   4  \frac{ |\nabla |f| |^2}{|f|^2 (\log   |f| ^2  )^2}  \frac{1}{  (\log (-\log |f|^2))^2}
       + 8 \frac{ |\nabla |f| |^2}{|f|^2 ( \log (|f| ^2 ))^2  (\log (-\log |f|^2))^3}.
\end{align*}

\subsection{Proof of Theorem \ref{t01}}

In this subsection, we are in a position to prove Theorem \ref{t01}. Theorem \ref{t02} and Theorem \ref{tC} will play crucial roles, without which our arguments would fail. 

Before embarking on the rigorous analytic proof, it is highly instructive to provide a brief heuristic observation regarding the one-dimensional case ($n=1$). When the spatial dimension is one, any nonconstant holomorphic function $f$ can be locally reduced to a monomial $w^m$ (where $m \ge 1$ is the multiplicity) near its zero via a suitable local biholomorphic coordinate transformation. Under this local chart, the potential function $u = \log(-\log|f|^2)$ takes the form $\log(-m\log|w|^2) = \log(-\log|w|^2) + \log m$. This implies that, up to a smooth additive constant, the one-dimensional case naturally degenerates to the classical radial counterexample discussed in the introduction. This observation not only confirms the geometric naturality of our construction but also profoundly reveals the fundamental difficulty in the higher-dimensional setting ($n \ge 2$): the intricate geometric and topological structures of analytic singular sets strictly obstruct such a simple local reduction. This is precisely why we must rely heavily on Theorem \ref{t02} and Theorem \ref{tC} to resolve the singularities and establish the refined geometric estimates in higher dimensions.
Since weak derivative problem is local, and to avoid the nonsmooth boundary of a polydisc, we instead take  $U = \{ z\in\mathbb{C}^n : |z|<1 \}$, the unit ball in $\mathbb{C}^n$ now.  Let $f$ be holomorphic in a neighborhood of $\overline{U}$, and assume without loss of generality that $f(0)=0$. 
Let $Z(f) = \{z \in U : f(z) = 0\}$. Note by assumption $0 \in Z(f)$ (i.e., $f(0)=0$).
Since $u = \log(-\log|f(z)|^2)$ and $g \in C^{\infty}(U \setminus Z(f))$, so on $U\setminus Z(f)$,
\[
\Delta u=g
\]
holds obviously. Our goal is to extend the   formulation to be over across $Z(f)$: namely, we want to show  $\Delta u = g$ holds in the sense of distributions on the whole $U$.
 
It is worth mentioning that a well-known Harvey-Polking Lemma \cite{HP70} (also \cite{SZ25}) does not apply here, which would demand $u$ to be continuous on $U$, that is obviously not the case for us. Indeed, $u$ is unbounded due to the singularity (despite rather mild ) at $Z(f)$.
Therefore, the only way we can do is via the verification of the definition of distributions.

The goal is to show for every $\varphi \in C_0^\infty(U)$, we want to verify
\[
\int_U \Delta \varphi\, u = \int_U \varphi\, g.
\]
In order to apply Green's theorem, we excise the singularity set $Z(f)$ by defining the domain.
\begin{equation} \label{eq:dfue}
U_\varepsilon=\{z\in U:|f|>\varepsilon\},\qquad\text{for }\varepsilon < \varepsilon_0,
\end{equation}
where $\varepsilon_0$ is given by Lemma \ref{l2201}.

We need some technical lemmas on $U_{\varepsilon}$ which are pertinent to applying Green's theorem and integration by parts throughout properly.

\begin{lemma}\label{l04}
The level set $\{|f|=\varepsilon\}$ is tangent to $\partial U$ at a point $p\in\partial U$  
if and only if $p$ is a critical point of the restricted function $u\big|_{\partial U}$.
\end{lemma}
\begin{proof}
By the method of Lagrange multipliers, $p$ is a critical point of $u(z)$  
subject to the constraint $|z|^2=1$  if and only if  there exist $\lambda \in \R$ such that
\[
\nabla u(p) = \lambda\, \nabla(|z|^2),
\]
which is exactly the tangency equation defined above.
\end{proof}
\begin{lemma}\label{l04finit}
For a real analytic function on a compact real analytic manifold,  
the set of critical values $V_{\mathrm{crit}}$ is a finite set.
\end{lemma}
\begin{proof}
Let $M$ be any real analytic manifold.  
Let $u : M \to \mathbb{R}$ be a real-analytic function.  
The set of critical points is defined by
\[
\mathcal C = \{\, z\in M : \nabla u(z)=0 \,\}.
\]

Since $\nabla u$ is real analytic, $\mathcal C$  is a real analytic subvariety of $M$. 
We claim that $u$ is constant on connected components of $\mathcal C$.

Let $Z\subset \mathcal C$ be any connected component of the critical set.  
Let $\gamma(t)$ be any smooth path lying entirely inside $Z$.  
Compute the change of $u$ along the path:
\[
\frac{d}{dt} u(\gamma(t)) 
= \nabla u(\gamma(t))\cdot \gamma'(t).
\]
Since $\gamma(t)$ is in the critical set, $\nabla u(\gamma(t))=0$,  by definition
\[
\frac{d}{dt} u(\gamma(t)) = 0.
\]
Therefore, $u$ is constant on every connected component of the critical set. Finally, we have to use a deep geometric step. A fundamental property of real analytic varieties on a compact manifold is that they have a finite number of connected components due to \L{}ojasiewicz \cite{Lo}.

Since $\mathcal C$ has finitely many connected components,  
and $u$ takes a single constant value on each component,  
the set of critical values  
\[
V_{\mathrm{crit}} = u(\mathcal C)
\]
is a finite set.
\end{proof}
Note that without compactness, it is easy to have examples showing otherwise, e.g. $u=\sin x$ on $\mathbb{R}$.

\begin{lemma} [Transversality Lemma]\label{lm:Trans}
   There exists $\varepsilon_0>0$, such that for  all $t \in (0, \varepsilon_0)$, the level set $S_t = \{ z \in U : |f(z)| = t \}$ is transversal to $\partial U$, the unit sphere in $\mathbb{C}^n$. Consequently, the intersection $\Gamma_t = S_t \cap \partial U$ is a smooth submanifold of real dimension $2n-2$, which has measure zero with respect to the $(2n-1)$-dimensional surface measure on $S_t$.
\end{lemma}

\begin{proof} [Proof of Transversality Lemma \ref{lm:Trans}] 
Let $\rho = |z|^2 - 1$ be the defining function of the unit ball $U$.
The level set $S_\varepsilon = \{ z \in U : |f(z)| = \varepsilon \}$ is defined by $\phi(z) = |f|^2 - \varepsilon^2$.
The hypersurfaces $\partial U$ and $S_\varepsilon$ are tangent at a point $p$ if and only if their real normals are parallel at $p$. Note that we need to consider $S_\varepsilon$ extended slightly outside of $\overline{U}$ a little bit, as shown in the picture below, since $f$ is holomorphic in a neighborhood of $\overline{U}$.

\begin{figure}[H]
    \centering
    \includegraphics[width=0.5\textwidth]{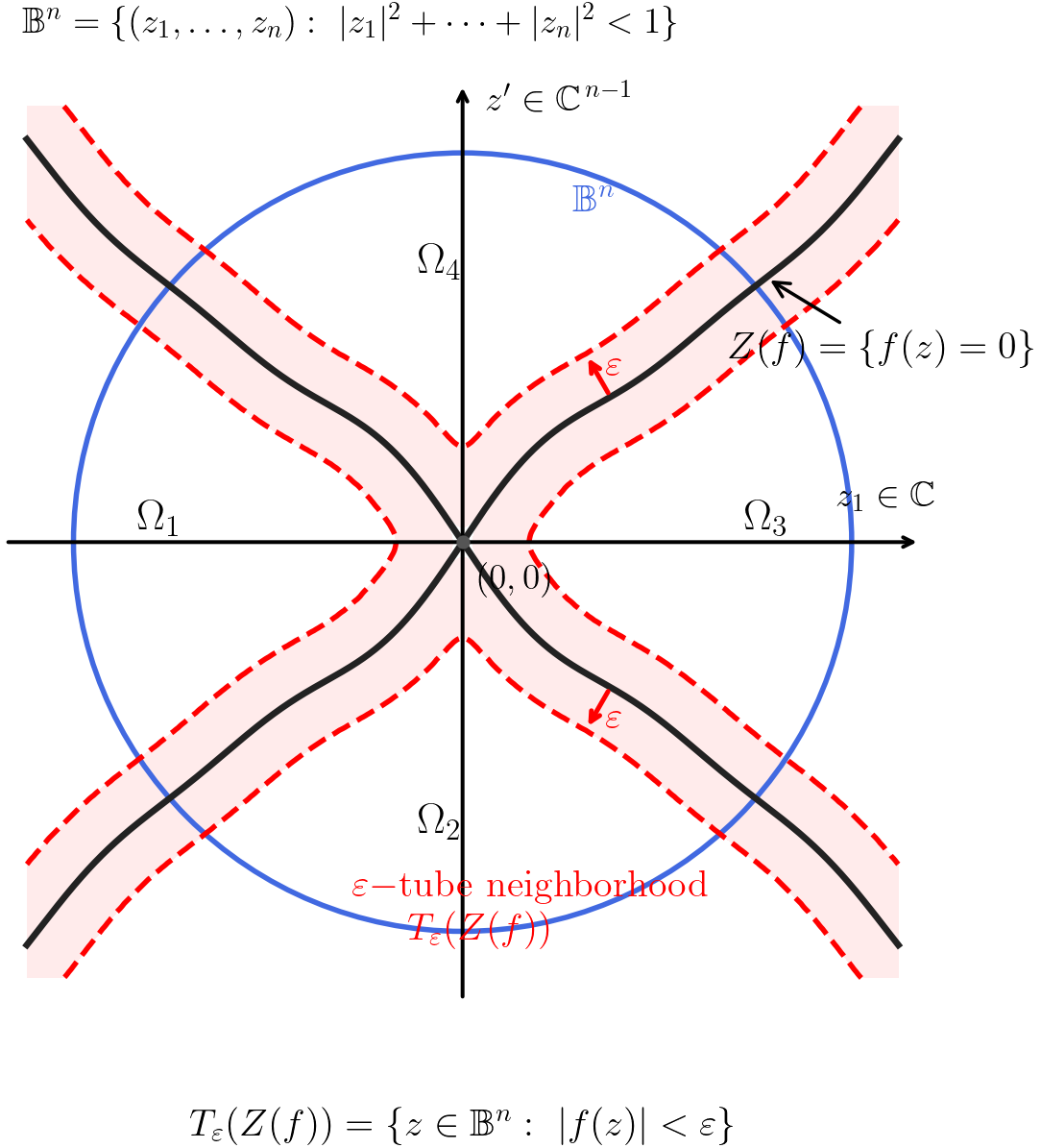} 
    \caption{Tubular neighborhood in the unit ball $\mathbb{B}^n$ }
    \label{test_ball}
\end{figure}


Let us look at the tangency condition. The normal vector field to $\partial U$ is
$ \nu_U = z $  (identifying tangent space with ambient space).
The normal vector field to $S_\varepsilon$ is
\[
\nu_f = \nabla(|f|^2)= f(z)\, \overline{\partial f(z)}.
\]
So $S_\varepsilon$ is tangent to $\partial U$ at $p$ if there exists a real scalar $\lambda \in \mathbb{R}$ such that
\[
\overline{f(p)}\, \frac{\partial f}{\partial z_j}(p) = \lambda\, \overline{p_j},
\qquad \text{for all } j=1,\cdots,n.
\]

Consider the restriction of the function $u(z)=|f(z)|^2$  
to the boundary $\partial U$.
\[
u\big|_{\partial U} : S^{2n-1} \to \mathbb{R}_{\geq0}.\]
Since $f$ is holomorphic in a larger ball, $u$ is real-analytic function on the compact manifold $\partial U (=S^{2n-1})$. 

In order to conclude the transversality, we need to ensure that $\varepsilon^2$ is not a critical value of $u\big|_{\partial U}$. Let $\mathcal C_{\mathrm{crit}}$ be the set of critical points of $u\big|_{\partial U}$. Let $V_{\mathrm{crit}} = u(\mathcal C_{\mathrm{crit}}) \subset \mathbb{R}$ be the set of critical values, which is finite by Lemma \ref{l04finit}. So we set $V_{\mathrm{crit}} = \{c_1, c_2, \dots, c_k\}$.
Hence the bad values of $\varepsilon$ are exactly those where $\varepsilon^2 \in V_{\mathrm{crit}}$.
Since $V_{\mathrm{crit}}$ is finite, there exists a smallest non-zero critical value. Let $\delta = \min\{\, c \in V_{\mathrm{crit}} : c>0 \,\}.$ If no non-zero critical value exists, set $\delta=\infty$.

Define $\varepsilon_0 = \sqrt{\delta}$. For any $\varepsilon$ such that $0<\varepsilon<\varepsilon_0$,
$\varepsilon^2$ is not a critical value of $u\big|_{\partial U}$. Therefore, $u\big|_{\partial U}$ has no critical points on the level set $u=\varepsilon^2$.
This implies the gradient $\nabla|f|^2$ and $\nabla|z|^2$ are never parallel on the intersection due to Lemma \ref{l04}, which means the transversality holds.
\end{proof}


\begin{lemma}\label{l311}
The domain $U_\varepsilon$ consists of finitely many components, each of which has Lipschitz boundary.
Moreover, the boundary of each component must have a part of the boundary of the unit ball.
\end{lemma}

\begin{proof}
Let $C$ be a connected component of $U_{\varepsilon}$. So, by definition, $C$ is an open and connected set. We first want to show that the boundary of $C$, denoted $\partial C$, must intersect the boundary of the unit ball $\partial U$. We will prove this by contradiction. Assume that the boundary of $C$ is entirely contained within $U$. This means that the closure of $C$ is a compact subset of $U$.

For any point $z$ inside the component $C$, we have $|f(z)| > \varepsilon$. By continuity of $f$, for any point $w$ on the boundary $\partial C$, we must have $|f(w)| \le \varepsilon$. Moreover, because $C$ is a connected component of the set where $|f(z)| > \varepsilon$, its boundary within $U$ must be precisely where $|f(z)| = \varepsilon$. Since we assumed $\partial C $ is entirely within $U$, it follows that $|f(z)| = \varepsilon$ for all $z \in \partial C$. By the maximum principle, this is a contradiction, since $f$ would be a constant, and  $f(0)=0$, we would have $f \equiv 0$.

Second, we want to show that $U_{\varepsilon}$ has at most a finite number of connected components. On $U_{\varepsilon}$, define $g(z) = |f(z)|^2 - \varepsilon^2 > 0$. Since $f$ is holomorphic in a neighborhood of $U$, $|f|^2$ is a real-analytic function in $\overline{U}$. Therefore, $U_\varepsilon$ is a subanalytic set and $\overline{U_\varepsilon}$ can be seen as a relative compact subset of a neighborhood of $\overline{U}$. A fundamental theorem in real-analytic geometry states that any relatively compact subanalytic set has a finite number of connected components. Indeed, we can prove it as follows.

Assume that $U_{\varepsilon}$ has an infinite number of connected components, let us  call them $\{C_j\}_{j=1}^{\infty}$.
A key property of subanalytic sets is that they are "locally" finite. This means that for any point $p \in U$, there exists a neighborhood $N$ of $p$ such that $N \cap U_{\varepsilon}$ has only a finite number of connected components.

Let us choose a point $z_j$ from each distinct component $C_j$. This gives us an infinite sequence $\{z_j\}$ within the bounded set $U$. Hence, this sequence would have a convergent subsequence.  Without loss of generality, we assume that $z_j $ converges to a point $ z_0$, which must lie in the closure of $U$, $\overline{U}$.

Case 1. The limit point $z_0$ is inside $U$. If $z_0$ is in $U$, then by the local finiteness property mentioned above, there is a neighborhood of $z_0$ that intersects only a finite number of components $C_j$. But since the sequence $\{z_j\}$ converges to $z_0$, this neighborhood must contain points from infinitely many different $C_j$. This is a contradiction.

Case 2. The limit point $z_0$ is on the boundary $\partial U$. Since $f$ is holomorphic in a neighborhood of $\overline{U}$, therefore this case reduces to Case 1  arriving a contradiction again.

Finally, we conclude that the boundary of $U_\varepsilon$ consists of two parts: 
 \begin{align*}
  \partial U_\varepsilon = \{z\in U:  |f|=\varepsilon \}  \cup  ( \partial U  \backslash  (\partial U \cap \{ |f|\leq \varepsilon \})), 
 \end{align*}
 as shown in the picture.
Since there are only finitely many connected components in $U_\varepsilon$, whose boundaries are Lipschitz due to the fact that $\partial U\cap \{z\in U: |f|=\varepsilon\}$ is transversal by Lemma \ref{lm:Trans}. The proof is completed.
\end{proof}

We can apply the Green's theorem on each component, see Figure $1$, and sum them up to get the conclusion. Note that the boundaries $\partial \Omega_j$ are never smooth, as they contain parts of the zero set of a holomorphic function. However, they are at least Lipschitz, which is sufficient for the theorem.
The condition $f$ being holomorphic near a neighborhood of $U$ is also crucial here, otherwise, it is possible to have infinitely many components.

Recall that the coarea formula for bounded domains with Lipschitz boundary states that, given $\phi\in L^1 (U)$ and a real-valued Lipschitz function $f$ on $U$,
\begin{align}
    \int_U \phi(x)|\nabla f(x)|dV_x=\int_{-\infty} ^{\infty} \int_{\{x\in U:f(x)=t \} }\phi(x) dS_x dt.
\end{align}
or 
\begin{align} \label{eq:coarea-ti}
    \int_U \phi(x)dV_x=\int_{-\infty} ^{\infty} \int_{\{x\in U:f(x) =t \}} \frac{\phi(x)}{|\nabla f(x)|} dS_x dt.
\end{align}
Here given $t\in \R$, $S_x$ is the $(n-1)$-dimensional Hausdorff measure of the level set $\{x\in U: f(x) =t\}$.

\begin{lemma}\label{l08}
Let $f$ be a nonconstant holomorphic function in a neighborhood of the unit polydisc $U$ such that $|f|<\frac{1}{2}$ on $\overline{U}$. If $f(0)=0$, then $\log(-\log|f|^2)\in L^p$ for every $p\geq 1$.
\end{lemma}
\begin{proof}

Applying the coarea formula,
\[
\begin{aligned}
\int_{U} \bigl|\log(-\log|f|^2)\bigr|^{p}\, dV_{2n}
&\approx \int_{0}^{1/2} \Bigl( \int_{\{z\in U:|f|=t\}} 
\frac{|\log(-\log|f|^2)|^{p}}{|\partial f|}\, dS \Bigr) dt \\
&\lesssim \int_{0}^{1/2} \Bigl( \int_{\{z\in U:|f|=t\}} 
\frac{|\log(-\log|f|^2)|^{p}}{|f|^{\alpha}}\, dS \Bigr) dt, \qquad  0<\alpha<1  \\
&\lesssim \int_{0}^{1/2} \frac{|\log(-\log t)|^{p}}{t^{\alpha}}\,
\bigl|\{z\in U: |f|=t \}\bigr|\, dt \\
&\lesssim \int_{0}^{1/2} \frac{|\log(-\log t)|^{p}}{t^{\alpha}}\, t^{1-\alpha}\, dt,
\qquad \text{ by Theorem \ref{th:levelsurface},} \\
&\lesssim \int_{0}^{1/2} \frac{|\log(-\log t)|^{p}}{t^{2\alpha-1}}\, dt .
\end{aligned}
\]
Since \( 0<\alpha<1 \), we have \( 2\alpha-1<1 \). Consequently the last integral is finite, which proves the lemma. 
\end{proof}

Before proceeding to the main theorem of this section, we shall establish a crucial geometric lemma. To analyze the precise integrability of singular weights near the analytic variety $Z(f)$, it is highly advantageous to perform a local reduction. The following lemma guarantees that, locally around a generic point of the singular locus, any holomorphic function can be trivialized into a pure monomial via a biholomorphic change of coordinates.

\begin{lemma}[Local monomialization at regular points]\label{lem:local_monomial}
Let $f \in \mathcal{O}_{\mathbb{C}^n,0}$ be a holomorphic germ at the origin with $f(0) = 0$. For any sufficiently small ball $B_r(0) \subset \mathbb{C}^n$, there exists a regular point $p \in Z(f) \cap B_r(0)$ and a local biholomorphism $\Phi: V \to W$, where $V \subset B_r(0)$ is an open neighborhood of $p$ and $W \subset \mathbb{C}^n$ is an open neighborhood of $0$, satisfying the following properties:
\begin{enumerate}
\item $\Phi(p) = 0$;
\item In the new holomorphic coordinates $w = (w_1, \dots, w_n) = \Phi(z)$, the analytic subset $Z(f) \cap V$ is given by the hyperplane $\{w \in W \mid w_1 = 0\}$;
\item The function $f$ takes the local monomial form $f \circ \Phi^{-1}(w) = w_1^m$, where the integer $m \ge 1$ is the algebraic multiplicity of $f$ along the chosen irreducible branch.  
\end{enumerate}
\end{lemma}
\begin{proof}
Since the local ring of holomorphic germs $\mathcal{O}_{\mathbb{C}^n,0}$ is a unique factorization domain, $f$ admits a factorization 
$$
f(z) = u(z) f_1(z)^{m_1} \cdots f_k(z)^{m_k},
$$
where $u(z)$ is a non-vanishing holomorphic unit, the $f_j$ are distinct irreducible holomorphic germs, and $m_j \ge 1$ are their respective multiplicities. Geometrically, the zero locus decomposes into irreducible branches as $Z(f) = \bigcup_{j=1}^k Z(f_j)$. The singular locus of $Z(f)$, which necessarily contains the intersection of distinct branches $\bigcup_{i \neq j} (Z(f_i) \cap Z(f_j))$, has complex codimension at least $1$ in $Z(f)$. Consequently, the regular points are dense in $Z(f)$. This allows us to choose a regular point $p \in Z(f_1) \cap B_r(0)$ arbitrarily close to the origin such that $p \notin Z(f_j)$ for all $j \ge 2$. Because $p$ is strictly bounded away from the other branches, we can choose a sufficiently small, simply connected open neighborhood $V \subset B_r(0)$ of $p$ such that the function $G(z) = u(z) \prod_{j=2}^k f_j(z)^{m_j}$ is nowhere zero on $V$. Setting $m = m_1$, we can rewrite the original function on $V$ as $f(z) = f_1(z)^m G(z)$. Since $V$ is simply connected and $G(z)$ is non-vanishing, there exists a single-valued holomorphic $m$-th root of $G(z)$ on $V$, denoted by $H(z)$, satisfying $H(z)^m = G(z)$ and $H(p) \neq 0$.We now define a new holomorphic function on $V$ by $w_1(z) = f_1(z) H(z)$, which immediately gives $f(z) = w_1(z)^m$. Evaluating the complex gradient of $w_1$ at $p$ yields 
$$
\nabla w_1(p) = H(p)\nabla f_1(p) + f_1(p) \nabla H(p).
$$

Since $p \in Z(f_1)$ is a regular point, we have $f_1(p) = 0$ and $\nabla f_1(p) \neq 0$. Combined with $H(p) \neq 0$, it follows that $\nabla w_1(p) = H(p)\nabla f_1(p) \neq 0$.This non-vanishing gradient guarantees that $\frac{\partial w_1}{\partial z_j}(p) \neq 0$ for at least one coordinate index $j \in \{1, \dots, n\}$. Up to a permutation of the original coordinates $(z_1, \dots, z_n)$, we may assume without loss of generality that $\frac{\partial w_1}{\partial z_1}(p) \neq 0$. We can then explicitly construct a holomorphic mapping $\Phi: V \to \mathbb{C}^n$ defined by$$\Phi(z) = \begin{pmatrix} 
w_1(z) \\ 
z_2 - p_2 \\ 
\vdots \\ 
z_n - p_n 
\end{pmatrix}.$$By construction, $\Phi(p) = 0$. The Jacobian matrix of $\Phi$ evaluated at $p$ is$$J_\Phi(p) = \begin{pmatrix}
\frac{\partial w_1}{\partial z_1}(p) & \frac{\partial w_1}{\partial z_2}(p) & \cdots & \frac{\partial w_1}{\partial z_n}(p) \\
0 & 1 & \cdots & 0 \\
\vdots & \vdots & \ddots & \vdots \\
0 & 0 & \cdots & 1
\end{pmatrix}.$$The determinant of this matrix is $\det J_\Phi(p) = \frac{\partial w_1}{\partial z_1}(p) \neq 0$. By the holomorphic inverse function theorem, upon shrinking $V$ if necessary, $\Phi$ is a biholomorphism from $V$ onto an open neighborhood $W$ of the origin. In this newly constructed coordinate system $w = \Phi(z)$, the analytic set $Z(f) \cap V$ is flattened to the hyperplane $\{w_1 = 0\}$, and the function takes the monomial form $f \circ \Phi^{-1}(w) = w_1^m$.
\end{proof}

With this geometric reduction at our disposal, we are now positioned to restate and rigorously prove the first assertion of Theorem \ref{t01}, characterizing the exact Sobolev regularity of the logarithmic potential.

\begin{theorem}\label{l10} 
Given a nonconstant holomorphic function $f$ defined in a neighborhood of the unit polydisc $U$ such that $|f|<\frac{1}{2}$ on $\overline{U}$. Set
\[
u = \log\!\bigl(-\log |f|^2\bigr), \qquad 
Z(f) = \{ z\in U : f(z)=0 \} \neq \emptyset,
\]
and define
\[
 h(z)  =  \label{df:h1}
 \begin{cases}
\displaystyle \nabla u= \frac{2\nabla |f|}{|f| \log |f|^2}, & z \notin Z(f), \\[4pt] 
  0, &  z \in Z(f).
\end{cases}
\]
Then \( h \in L^2_{\mathrm{loc}}(U) \), but for any \( p>2 \) we have \( h \notin L^p_{\mathrm{loc}}(U) \) near \( Z(f) \).  
Moreover, the weak derivative \( \nabla u \) exists and equals \( h \) on \( U \). Consequently,
\[
u \in W^{1,2}_{\mathrm{loc}}(U), \quad \text{but} \quad u \notin W^{1,p}_{\mathrm{loc}}(U)
\]
for any \( p>2 \), near \( Z(f) \).
\end{theorem}

\begin{proof}
    First, we show that $h \in L^2_{\rm{loc}} (U)$. Since $f$ is holomorphic in a neighborhood of $U$, whose boundary is Lipschitz, we can extend $|f|$ to be globally Lipschitz in $\mathbb{C}^n$ by the Kirszbraun theorem. Now as in $(\ref{eq:dfue})$ we let 
    \begin{align*}
       U_\varepsilon=\{z\in U:|f|>\varepsilon\},\qquad\text{for }\varepsilon < \varepsilon_0.
    \end{align*}
    Making use of the coarea formula (\ref{eq:coarea-ti}) applied to $|h|^2$ on $U_\varepsilon$, one gets 
    \begin{align*}
    \int _{U_\varepsilon}|h|^2 & \approx  \int _\varepsilon ^{\frac{1}{2}} \left ( \int_{ \{ z\in U: |f|=t  \}}  \frac{ |h|^2 }{  |\nabla |f||}dS\right) dt \\\nonumber 
    &\approx \int _\varepsilon ^{\frac{1}{2}}  \left ( \int_{ \{ z\in U: |f|=t  \}}  \frac{ |\nabla |f||^2}{ |f|^2 (\log |f|^2)^2 } \cdot \frac{1} {|\nabla |f||}dS \right)   dt \\\nonumber 
    & \approx\int _\varepsilon^{\frac{1}{2}}  \frac{1}{t^2 (\log t^2)^2} \left( \int_{ \{ z\in U: |f|=t  \}}  |\nabla |f|| dS\right) dt  \\\nonumber 
    & \lesssim \int _\varepsilon ^{\frac{1}{2}} \frac{1}{t^2 (\log t^2)^2} \left( \int_{ \{ z\in U: |f|=t  \}}  |\partial  f |  dS \right) dt\\\nonumber 
     &\lesssim \int _{\varepsilon} ^{\frac{1}{2}} \frac{1}{t(\log t)^2} dt, \qquad  \ \text{by $(1)$ of Theorem \ref{t02}}, \\ \nonumber 
     & <\infty.
    \end{align*}
    Let \( \chi_{U_\varepsilon} \) be the characteristic function of \( U_\varepsilon \), which is $1$ on $U_\varepsilon$ and $0$ outside of $U_\varepsilon$. Letting $\varepsilon \rightarrow 0$, we obtain by the monotone convergence theorem
    \begin{align*}
    \int _{U}|h|^2 &  = \lim_{\varepsilon \rightarrow 0}     \int _{U} \chi_{U_\varepsilon} |h|^2 \\\nonumber 
     & \lesssim  \lim_{\varepsilon \rightarrow 0}    \int _{\varepsilon} ^{\frac{1}{2}} \frac{1}{t(\log t)^2} dt  \\  \nonumber 
     & \lesssim \int _{0} ^{\frac{1}{2}} \frac{1}{t(\log t)^2} dt  \\  \nonumber 
     & <\infty.
    \end{align*}
    Thus, $h\in L^2_{\rm{loc}} (U)$. 

    Next, we prove that \( h \notin L^{p}_{\mathrm{loc}}(U) \) for any \( p>2 \) near $Z(f)$. It suffices to assume without loss of generality that $0 \in Z(f)$ and to show that $\int_{B_r(0)} |h|^p \, d\lambda_z = \infty$ for any sufficiently small ball $B_r(0) \subset U$.

    By Lemma \ref{lem:local_monomial}, there exists a regular point $p \in Z(f) \cap B_r(0)$ and a local biholomorphism $\Phi: V \to W$, where $V \subset B_r(0)$ is an open neighborhood of $p$ and $W \subset \mathbb{C}^n$ is an open neighborhood of $0$, such that in the new holomorphic coordinates $w = (w_1, \dots, w_n) = \Phi(z)$, the function takes the simplest possible form
    \[
    f \circ \Phi^{-1}(w) = w_1^m, \quad \text{where } m \ge 1.
    \]

    We want to test if the weak derivative \( h = \nabla_z u \) is in \( L^p(V) \). Under the biholomorphism $\Phi$, the length of the gradient and the volume measure are distorted by bounded amounts on compact subsets. Therefore, on the sufficiently small neighborhood $V$, we have the equivalence
    \[
    |\nabla_z u| \approx |\nabla_w u| \quad \text{and} \quad d\lambda_z \approx d\lambda_w.
    \]
    Let us compute \( |\nabla_w u| \) directly using the monomial form \( f(w) = w_1^m \),
    \[
    u(w) = \log\bigl(-\log|f|^2\bigr) = \log\bigl(-\log|w_1|^{2m}\bigr) = \log\bigl(-2m \log|w_1|\bigr).
    \]
    Taking the gradient with respect to $w$, this implies 
    \[
    |\nabla_w u| \approx \frac{1}{|w_1| \, \bigl|\log |w_1|\bigr|}.
    \]

    Now we check the \( L^p \) integrability over the neighborhood \( V \). By the equivalence of norms and measures, we obtain
    \[
    \int_V |h|^p \, d\lambda_z = \int_V |\nabla_z u|^p \, d\lambda_z \approx \int_{\Phi(V)} \frac{1}{|w_1|^p \, \bigl|\log |w_1|\bigr|^p} \, d\lambda_w.
    \]
    By Fubini's theorem, we can separate the integral over the first coordinate \( w_1 \) from the remaining variables \( w_2, \dots, w_n \). The integral over \( w_1 \) is performed over a small disk of radius \( \varepsilon \). Switching to polar coordinates \( w_1 = r e^{i\theta} \), we have
    \[
    \int_{|w_1| < \varepsilon} \frac{1}{|w_1|^p \, \bigl|\log |w_1|\bigr|^p} \, d\lambda_{w_1} = \int_{0}^{\varepsilon} \int_{0}^{2\pi} \frac{1}{r^p \, |\log r|^p} \, r \, d\theta \, dr = 2\pi \int_{0}^{\varepsilon} \frac{dr}{r^{p-1} \, |\log r|^p}.
    \]
    For any \( p > 2 \), the exponent in the denominator satisfies \( p-1 > 1 \). Consequently, the integral strongly diverges at $r=0$.

    Since $\int_V |h|^p \, d\lambda_z = \infty$ and $V \subset B_r(0)$, it immediately follows that the integral over the entire ball $B_r(0)$ diverges as well. This concludes the proof that \( h \notin L^p_{\mathrm{loc}}(U) \) for any \( p>2 \) near $Z(f)$, and thus \( u \notin W^{1,p}_{\mathrm{loc}}(U) \).

We now verify that \( u \) is a weak solution of \( \nabla u = h \); i.e., for every test function \( \varphi \in C^{\infty}_{0}(U) \),
\begin{eqnarray}\label{eq:theo1-3h}
- \int_U u \nabla \varphi   = \int_U h \varphi.
\end{eqnarray}
Equivalently, for each component \( j \) we must check  
\begin{eqnarray}\label{eq:theo1-3hdengjia}
- \int_U u \frac{\partial  \varphi }{\partial \overline{z}_j}  = \int_U \widetilde{h}_j \varphi.
\end{eqnarray}
where \( \tilde h = (\tilde h_{1},\dots ,\tilde h_{n}) \) is defined by  
\[
\tilde h_{j}(z) = 
\begin{cases}
\displaystyle \frac{1}{\bar f(z) \log |f(z)|^{2}} \, \frac{\partial \bar f}{\partial \bar z_{j}}(z), & z \notin Z(f), \\[4pt]
0, & z \in Z(f) .
\end{cases}
\]
Again set \( U_{\varepsilon} = \{ z \in U : |f(z)| > \varepsilon \} \). By Lemma \ref{l311}, \( \partial U_{\varepsilon} \) is Lipschitz. Applying the divergence theorem to each connected component of \( U_{\varepsilon} \) and summing up, we obtain for fixed \( j \)
\begin{eqnarray}\label{eq:theo1-3huepsilonr}
  \int_{U_\varepsilon}  \frac{\partial u} {\partial \overline{z}_j} \varphi  d\lambda = -\int_{U_\varepsilon} u \frac{\partial \varphi } {\partial \overline{z}_j}   d\lambda + \int_{\{ z\in U: |f| = \varepsilon \}} u \varphi \frac{\partial \rho }{\partial \overline{z} _j} \frac{dS}{| \partial \rho|},
\end{eqnarray}
where \( \rho = \varepsilon^{2} - |f|^{2} \) is a defining function of \( U_{\varepsilon} \).  The outer boundary contribution vanishes because \( \varphi \) has compact support inside \( U \).   
 On \( U_{\varepsilon} \) the classical derivative satisfies \( \frac{\partial u}{\partial \bar z_{j}} = \tilde h_{j} \), substituting this into (\ref{eq:theo1-3huepsilonr}) gives  
 \begin{eqnarray*}\label{eq:theo1-3huepsilonrh}
  \int_{U_\varepsilon} \widetilde{h}_j \varphi  d\lambda = -\int_{U_\varepsilon} u \frac{\partial \varphi } {\partial \overline{z}_j}   d\lambda + \int_{\{ z\in U: |f| = \varepsilon \}} u \varphi \frac{\partial \rho }{\partial \overline{z} _j} \frac{dS}{| \partial \rho|}.
\end{eqnarray*}
Rearranging,  
\begin{eqnarray}\label{eq:theo1-3huepsilonrhe}
 \int_{U_\varepsilon} u \frac{\partial \varphi } {\partial \overline{z}_j}   d\lambda= - \int_{U_\varepsilon} \widetilde{h}_j \varphi  d\lambda  + \int_{\{ z\in U: |f| = \varepsilon \}} u \varphi \frac{\partial \rho }{\partial \overline{z} _j} \frac{dS}{| \partial \rho|}.
\end{eqnarray}
We claim that the boundary integral tends to zero as \( \varepsilon \to 0 \):
\begin{align} \label{eq:yaoz1}
   \lim _{\varepsilon \rightarrow 0}\int_{\{ z\in U: |f| = \varepsilon \}} u \varphi \frac{\partial \rho }{\partial \overline{z} _j} \frac{dS}{| \partial \rho|}= 0.
\end{align}
If so, then passing $\varepsilon \rightarrow 0$ in (\ref{eq:theo1-3huepsilonrhe}), we obtain the desired equality (\ref{eq:theo1-3hdengjia}) as a consequence of Lebesgue's dominated convergence theorem. Therefore, \( \nabla u = h \) weakly on \( U \).

To prove (\ref{eq:yaoz1}),  
\begin{align*} \label{eq:guji1rho}
\left|\int_{ \{ z \in U: |f|=\varepsilon \} } u \varphi \frac{\partial \rho }{\partial \overline{z} _j} \frac{dS}{| \partial \rho|}  \right| &  \lesssim   \int _{ \{ z \in U: |f|=\varepsilon \} }   |u| \\
& \lesssim \log (-\log \varepsilon) | \{ z\in U: |f| =  \varepsilon  \}| \\
& \lesssim  \varepsilon ^\gamma  \log (-\log \varepsilon), \qquad  \text{by Theorem \ref{th:levelsurface} for some $0< \gamma\leq 1$},\\
& \rightarrow 0,  (\varepsilon \rightarrow 0).
\end{align*}

Hence, combining this with Lemma \ref{l08}, we conclude that $u$ belongs to $W^{1,2}_{\mathrm{loc}}(U)$ but fails to be in $W^{1,p}_{\mathrm{loc}}(U)$ for any $p > 2$.

\end{proof}

We restate $2$ of Theorem \ref{t01} as the following.
\begin{theorem}\label{lm:lapla}
Let \( f \) be a nonconstant holomorphic function defined in a neighborhood of the unit polydisc \( U \) with \( |f|<\frac12 \) on \( \overline{U} \). Set  
\[
u = \log\!\bigl(-\log|f|^2\bigr), \qquad 
Z(f)\neq \emptyset,
\]
and define
\[
  g(z) = 
\begin{cases}
\Delta u =-4\,\dfrac{|\nabla|f||^{2}}{|f|^{2}\,\bigl(\log|f|^{2}\bigr)^{2}}, &   z\notin Z(f), \\[4pt]
0, &   z\in Z(f). 
\end{cases}
\]
Then \( g\in L^{1}_{\mathrm{loc}}(U) \), but for every \( p>1 \) we have \( g\notin L^{p}_{\mathrm{loc}}(U) \) near \( Z(f) \).
\end{theorem}

\begin{proof}
We first show that \( g\in L^{1}_{\mathrm{loc}}(U) \). Using the same reasoning as in the proof of Theorem~\ref{l10},
\[
\begin{split}
\int_{U}|g| 
&\approx \int_{0}^{1/2}\Bigl(\int_{\{z\in U:|f|=t\}} 
\frac{|\nabla|f||^{2}}{|f|^{2}(\log|f|^{2})^{2}}\cdot\frac{1}{|\nabla|f||}\,dS\Bigr)dt \\
&\approx\int_{0}^{1/2}\frac{1}{t^{2}(\log t^{2})^{2}}
\Bigl(\int_{\{z\in U:|f|=t\}}|\nabla|f||\,dS\Bigr)dt \\
&\approx\int_{0}^{1/2}\frac{1}{t^{2}(\log t^{2})^{2}}
\Bigl(\int_{\{z\in U:|f|=t\}}|\partial f|\,dS\Bigr)dt \\
&\lesssim  \int_{0}^{1/2}\frac{1}{t(\log t)^{2}}\,dt,\qquad 
\text{by part (1) of Theorem~\ref{t02}}, \\
&< \infty .
\end{split}
\]
Hence \( g\in L^{1}_{\mathrm{loc}}(U) \).

To prove that \( g = \Delta u \notin L^p_{\mathrm{loc}}(U) \) for any \( p > 1 \), we again employ the geometric localization strategy provided by Lemma \ref{lem:local_monomial}. 

Without loss of generality, we may assume $0 \in Z(f)$. By Lemma \ref{lem:local_monomial}, there exists a regular point $p \in Z(f)$ arbitrarily close to the origin and a small open neighborhood $V \subset U$ of $p$, along with a biholomorphism $\Phi: V \to W \subset \mathbb{C}^n$. In the new holomorphic coordinates $w = (w_1, \dots, w_n) = \Phi(z)$, the analytic set $Z(f) \cap V$ corresponds to the hyperplane $\{w_1 = 0\}$, and the function takes the pure monomial form $f \circ \Phi^{-1}(w) = w_1^m$ for some integer $m \ge 1$.

In these local coordinates, the function $u$ can be rewritten as
\[
u(w) = \log\bigl(-\log|f|^2\bigr) = \log\bigl(-m\log|w_1|^2\bigr).
\]
Under the biholomorphism $\Phi$, the standard Laplacian $\Delta_z$ transforms into an elliptic operator with smooth coefficients. For the purpose of determining the $L^p$ growth rate on a sufficiently small, relatively compact neighborhood $V$, the volume measure and the norm of the operators are distorted by strictly positive, bounded smooth factors. Thus, we have the asymptotic equivalence $|g(z)| = |\Delta_z u| \approx |\Delta_w u|$ and $d\lambda_z \approx d\lambda_w$. 

A direct computation of the Laplacian with respect to the first coordinate $w_1$ yields
\[
|\Delta_w u| = \left| 4 \frac{\partial^2 u}{\partial w_1 \partial \bar{w}_1} \right| = \frac{1}{|w_1|^2 \, \bigl(\log |w_1|\bigr)^2}.
\]
Consequently, on the neighborhood $V$, the function $g$ satisfies the point-wise estimate 
$$
|g| \approx |w_1|^{-2} |\log |w_1||^{-2}.
$$ 

To test the $L^p$ integrability, we evaluate the integral over $V$. Applying Fubini's theorem, we can separate the integral over the first coordinate $w_1$ from the remaining smooth variables $w_2, \dots, w_n$. This gives
\begin{align*}
\int_{U} |g|^p \, d\lambda_z \ge \int_{V} |g|^p \, d\lambda_z &\approx \int_{\Phi(V)} \frac{1}{|w_1|^{2p} \, \bigl|\log |w_1|\bigr|^{2p}} \, d\lambda_w \\
&\approx \int_{|w_1| < r_0} \frac{1}{|w_1|^{2p} \, \bigl|\log |w_1|\bigr|^{2p}} \, d\lambda_{w_1}
\end{align*}
for some small radius $r_0 > 0$. Switching to polar coordinates $w_1 = s e^{i\theta}$ on the disk, we obtain
\[
\int_{0}^{r_0} \int_{0}^{2\pi} \frac{1}{s^{2p} \, |\log s|^{2p}} \, s \, d\theta \, ds = 2\pi \int_{0}^{r_0} \frac{1}{s^{2p-1} \, |\log s|^{2p}} \, ds.
\]
For any \( p > 1 \), the exponent in the denominator satisfies \( 2p - 1 > 1 \). Consequently, the integral strongly diverges at $s=0$. This establishes that the integral of $|g|^p$ over $V$ is infinite, and therefore \( g \notin L^p_{\mathrm{loc}}(U) \) for all \( p > 1 \).
\end{proof}

\begin{proof}[Proof of  $3$ of Theorem \ref{t01}]

Let
\begin{align}\label{eq:tidu}
\vec{n}:=\frac{\nabla\rho}{|\nabla\rho|}
 =-\frac{\nabla|f|^{2}}{|\nabla|f|^{2}|}
 =-\frac{\nabla|f|}{|\partial f|}
\end{align}
be the outward unit normal vector, where \( \rho \) is a defining function of  
\( U_{\varepsilon}:=\{z\in U:|f|^{2}>\varepsilon^{2}\} \).

Applying Green’s theorem to each of the finitely many components of \( U_{\varepsilon} \), whose boundaries are Lipschitz by Lemma~\ref{l311} and summing up, we obtain
\[
\begin{aligned}
\int_{U_{\varepsilon}}\bigl(u\Delta\varphi-\varphi\Delta u\bigr)
&= \int_{\partial U_{\varepsilon}}(u\nabla\varphi-\varphi\nabla u)\cdot\vec{n}\,dS \\
&= \int_{\{z \in U: |f| = \varepsilon\}}(u\nabla\varphi-\varphi\nabla u)\cdot\vec{n}\,dS \\
&\qquad +\int_{\partial U\setminus(\partial U\cap\{|f|\le\varepsilon\})}
(u\nabla\varphi-\varphi\nabla u)\cdot\vec{n}\,dS \\
&= \int_{\{z \in U: |f| = \varepsilon\}}(u\nabla\varphi-\varphi\nabla u)\cdot\vec{n}\,dS,
\end{aligned}
\]
the last equality follows because \( \varphi\in C_{0}^{\infty}(U) \) vanishes near \( \partial U \).  
Denote
\[
I_{\varepsilon}= \int_{\{z \in U: |f| = \varepsilon\}} u\,\nabla\varphi\cdot\vec{n}\,dS, \qquad
\Pi_{\varepsilon}= -\int_{\{z \in U: |f| = \varepsilon\}}\varphi\,\nabla u\cdot\vec{n}\,dS .
\]

To conclude, we must verify that both \( I_{\varepsilon} \) and \( \Pi_{\varepsilon} \) tend to zero as \( \varepsilon\to 0 \), the dominated convergence theorem will then yield the desired identity.

 From Theorem~\ref{th:levelsurface} we know that \( |\{z\in U:|f(z)|=\varepsilon\}| = O(\varepsilon^{\gamma}) \) for some \( 0<\gamma\le1 \). Hence,
\[
|I_{\varepsilon}|
\le \max_{U}|\nabla\varphi\cdot\vec{n}|
\int_{\{z \in U: |f| = \varepsilon\}}|u|\,dS
\lesssim \log\!\bigl(-\log\varepsilon\bigr)\,|{\{z \in U: |f| = \varepsilon\}}|
\to 0 \quad (\varepsilon\to 0).
\]

For \( \Pi_{\varepsilon} \) we have
\[
|\Pi_{\varepsilon}|
\le \max_{U}|\varphi|
\int_{\{z \in U: |f| = \varepsilon\}}|\nabla u\cdot\vec{n}|\,dS .
\]
Observe that
\[
|\nabla u\cdot\vec{n}|
\lesssim \frac{|\nabla|f||}{|f|\,|\log|f||}.
\]
Consequently,
\[
\begin{aligned}
|\Pi_{\varepsilon}|
&\lesssim \frac{1}{\varepsilon|\log\varepsilon|}
\int_{\{z \in U: |f| = \varepsilon\}}|\nabla|f||\,dS \\
&\lesssim \frac{1}{\varepsilon|\log\varepsilon|}
\int_{\{z \in U: |f| = \varepsilon\}}|\partial f|\,dS \\
&\lesssim \frac{1}{|\log\varepsilon|}\to 0\qquad(\varepsilon\to 0),
\end{aligned}
\]
where the last estimate uses part (1) of Theorem~\ref{t02}.

\end{proof}

\begin{proof}[Proof of  $4$ of Theorem \ref{t01}]

We argue by contradiction. Assume that $u \in W^{2,1}_{\mathrm{loc}}(U)$. By the definition of Sobolev spaces, this implies that all real second-order weak derivatives of $u$ belong to $L^1_{\mathrm{loc}}(U)$. Since any pure complex second derivative is a finite linear combination of real second derivatives, our assumption necessitates that $\frac{\partial^2 u}{\partial z_i \partial z_j} \in L^1_{\mathrm{loc}}(U)$ for all $1 \le i, j \le n$.  From the calculation in the beginning of this Section, the pure complex second-order derivative for $z \notin Z(f)$ takes the explicit form 
$$
\frac{\partial^2 u}{\partial z_i \partial z_j} = \frac{\partial_{z_i z_j} f}{f \log|f|^2} - \frac{\partial_{z_i} f \, \partial_{z_j} f}{f^2 \log|f|^2} - \frac{\partial_{z_i} f \, \partial_{z_j} f}{f^2 \bigl(\log|f|^2 \bigr)^2}.
$$
As $f \to 0$, the most singular term determining the integrability is precisely $\frac{\partial_{z_i}f\,\partial_{z_j}f}{f^{2}\log|f|^{2}}$. Without loss of generality, we may assume $0 \in Z(f)$. By  Lemma \ref{lem:local_monomial}, we can choose a sufficiently small neighborhood $V \subset U$ intersecting the regular part of $Z(f)$, and perform a biholomorphic coordinate transformation $w = \Phi(z)$. In these new holomorphic coordinates, the analytic set $Z(f) \cap V$ corresponds to the hyperplane $\{w_1 = 0\}$, and locally $f(z(w)) = w_1^m$ for some integer $m \ge 1$. Since the local Sobolev space $W^{2,1}$ is invariant under biholomorphic mappings, the function in the new coordinates, $\tilde{u}(w) := u(z(w))$, must satisfy $\frac{\partial^2 \tilde{u}}{\partial w_1^2} \in L^1_{\mathrm{loc}}(\Phi(V))$. Using $f = w_1^m$, a direct computation of this pure second derivative yields 
$$
\frac{\partial^2 \tilde{u}}{\partial w_1^2} = - \frac{1}{w_1^2 \log|w_1|^2} - \frac{1}{w_1^2 \bigl(\log|w_1|^2\bigr)^2}.
$$
For sufficiently small $|w_1|$, the term containing $|\log|w_1|^2|^{-1}$ strictly dominates. Thus, there exists a constant $C > 0$ such that we have the point-wise lower bound $$\left| \frac{\partial^2 \tilde{u}}{\partial w_1^2} \right| \ge \frac{C}{|w_1|^2 \, \bigl| \log|w_1|^2 \bigr|}$$To test its integrability, we evaluate the integral over a small polydisc $\Delta^n = \Delta_1 \times \Delta' \subset \Phi(V)$, where $\Delta_1 = \{w_1 \in \mathbb{C} \mid |w_1| < 1/2\}$. Applying Fubini's theorem to separate the variables, we obtain $$
\int_{\Phi(V)} \left| \frac{\partial^2 \tilde{u}}{\partial w_1^2} \right| \, dV_w \ge \int_{\Delta'} \left( \int_{\Delta_1} \frac{C}{|w_1|^2 \, \bigl| \log|w_1|^2 \bigr|} \, dV_{w_1} \right) dV_{w'}.
$$
Passing to polar coordinates $w_1 = t e^{i\theta}$ on the disk $\Delta_1$, the inner integral can be evaluated as follows 
$$
\int_{\Delta_1} \frac{1}{|w_1|^2 \, \bigl| \log|w_1|^2 \bigr|} \, dV_{w_1} = 2\pi \int_0^{1/2} \frac{1}{t^2 \, \bigl| 2\log t \bigr|} t \, dt \gtrsim \int_0^{1/2} \frac{1}{t \, |\log t|} \, dt.
$$
The final radial integral yields $[\log(|\log t|)]_0^{1/2}$, which strictly diverges to $\infty$ as $t \to 0$. Consequently, $\int_{\Phi(V)} \left| \frac{\partial^2 \tilde{u}}{\partial w_1^2} \right| \, dV_w = \infty$. This implies that the assumed weak second derivatives cannot all belong to $L^1_{\mathrm{loc}}(U)$, generating a direct contradiction. Hence, $u \notin W^{2,1}_{\mathrm{loc}}(U)$. 
\end{proof}

\subsection{Proof of Theorem \ref{th:maintheorem2}}
We restate  $1$ of Theorem \ref{th:maintheorem2} as the following theorem.
\begin{theorem}\label{lm:lapla}

Let \( f \) be a nonconstant holomorphic function defined in a neighborhood of the unit polydisc \( U \), with \( |f|<\frac12 \) on \( \overline{U} \).  
Set  
\[
v = \frac{1}{\log\!\bigl(-\log|f|^2\bigr)}, \qquad 
Z(f) = \{ z\in U : f(z)=0 \} \neq \emptyset,
\]
and define
\[
 l(z): =
 \begin{cases}
\displaystyle 
\nabla v = \frac{2\,\nabla|f|}{|f|\,\log|f|^{2}\,\bigl(\log(-\log|f|^{2})\bigr)^{2}}, &   z\notin Z(f),  \\[4pt]
0, &   z\in Z(f). 
\end{cases}
\]
Then \( l\in L^{2}_{\mathrm{loc}}(U) \), but for every \( p>2 \) we have \( l\notin L^{p}_{\mathrm{loc}}(U) \) near \( Z(f) \).  
Moreover, the weak derivative \( \nabla v \) exists and satisfies \( \nabla v = l \) on \( U \). Consequently,
\[
v \in W^{1,2}_{\mathrm{loc}}(U), \qquad \text{but} \qquad 
v \notin W^{1,p}_{\mathrm{loc}}(U) \quad \text{for any } p>2,
\]
 near \( Z(f) \).
\end{theorem}

\begin{proof} 

Since \( |l(z)| \le |h(z)| \) on \( U\setminus Z(f) \) and \( h(z)\in L^{2}_{\mathrm{loc}}(U) \) (see Definition (\ref{df:h1})), we immediately obtain \( l\in L^{2}_{\mathrm{loc}}(U) \).

To see that \( l\notin L^{p}_{\mathrm{loc}}(U) \) for any \( p>2 \) near \( Z(f) \), performing the same holomorphic coordinate transformation $w = \Phi(z)$ as Lemma \ref{lem:local_monomial},  we can choose a sufficiently small neighborhood $V \subset U$ intersecting the regular part of $Z(f)$, such that in these new holomorphic coordinates, the analytic set $Z(f) \cap V$ corresponds to the hyperplane $\{w_1 = 0\}$, and locally $f(z(w)) = w_1^m$ for some integer $m \ge 1$.  Thus, 
\[
\begin{aligned}
\int_{V} |l|^{p}
&\gtrsim \int_{\Phi(V)} 
\frac{1}{|w_{1}|^{p}\,|\log|w_{1}||^{p}\,\bigl|\log(-\log|w_{1}|^{2})\bigr|^{2p}} 
\, dV_{w_1}, \qquad \text{ by Fubini's theorem,} \\
&\gtrsim \int_{0}^{r_{0}} 
\frac{1}{s^{p-1}\,|\log s|^{p}\,\bigl(\log(-\log s^{2})\bigr)^{2p}} \, ds
\end{aligned}
\]
for some \( r_{0}>0 \). For \( p>2 \) the last integral diverges; therefore \( l(z) \notin L^{1,p}_{\mathrm{loc}}(U) \) for any \( p>2 \) near $Z(f)$.

Next we verify that \( v \) is a weak solution of \( \nabla v = l \); i.e., for every test function \( \varphi\in C^{\infty}_{0}(U) \),
\begin{eqnarray}\label{eq:theo1-3hv}
- \int_U v \nabla \varphi   = \int_U l(z) \varphi .
\end{eqnarray}
Proceeding as in the proof of Theorem \ref{l10}, it suffices to show that the boundary term vanishes in the limit:
\begin{align} \label{eq:yaoz}
   \lim _{\varepsilon \rightarrow 0}\int_{\{ z\in U: |f| = \varepsilon \}} v \varphi \frac{\partial \rho }{\partial \overline{z} _j} \frac{dS}{| \partial \rho|}= 0.
\end{align}
where \( \rho = \varepsilon^{2}-|f|^{2} \) is a defining function of \( U_{\varepsilon}=\{  |f|>\varepsilon \} \).

We estimate the left hand side of (\ref{eq:yaoz}) using Theorem \ref{th:levelsurface}, which gives \( |\{ z\in U:|f|=\varepsilon \}| = O(\varepsilon^{\gamma}) \) for some \( 0<\gamma\le1 \):
\[
\begin{aligned}
\Bigl|\int_{\{z \in U: |f| = \varepsilon\}}v\,\varphi\,
\frac{\partial\rho}{\partial\bar z_{j}}\frac{dS}{|\partial\rho|}\Bigr|
&\lesssim \int_{\{z \in U: |f| = \varepsilon\}} |v|\,dS \\
&\lesssim \frac{1}{\log(-\log\varepsilon)}\,
\bigl|\{ z\in U:|f|=\varepsilon \}\bigr| \\
&\lesssim \frac{\varepsilon^{\gamma}}{\log(-\log\varepsilon)} \;\to\;0 \qquad (\varepsilon\to0).
\end{aligned}
\]

Hence (\ref{eq:yaoz}) holds, and by the dominated convergence theorem we obtain $(\ref{eq:theo1-3hv})$.  
Thus, \( \nabla v = l \) weakly, which together with the integrability properties established above and continuous of $v$ yields
\[
v \in W^{1,2}_{\mathrm{loc}}(U), \qquad 
v \notin W^{1,p}_{\mathrm{loc}}(U) \quad \text{for any } p>2,
\]
near \( Z(f) \).
 
\end{proof}

We restate $2$ of Theorem \ref{th:maintheorem2} as the following.
\begin{theorem}\label{lm:lapla}

Let \( f \) be a nonconstant holomorphic function defined in a neighborhood of the unit polydisc \( U \), with \( |f|<\frac12 \) on \( \overline{U} \).  
Set  
\[
v = \frac{1}{\log\!\bigl(-\log|f|^2\bigr)}, \qquad 
Z(f) = \{ z\in U : f(z)=0 \} \neq \emptyset,
\]
and define
\[
 t(z): = 
\begin{cases}
\displaystyle 
\Delta v = \, \frac{4|\nabla|f||^{2}}
{|f|^{2}\,\bigl(\log|f|^{2}\bigr)^{2}\,
\bigl(\log(-\log|f|^{2})\bigr)^{2}}
\Bigl[1+\frac{2}{\log(-\log|f|^{2})}\Bigr], &   z\notin Z(f),  \\[4pt]
0, &   z\in Z(f).  \\  \nonumber 
\end{cases}
\]
Then \( t\in L^{1}_{\mathrm{loc}}(U) \), but for every \( p>1 \) we have \( t\notin L^{p}_{\mathrm{loc}}(U) \) near \( Z(f) \).

\end{theorem}

\begin{proof}
   First we show that $t(z)\in L^1_{\rm{loc}} (U)$.   Note that
   \begin{align}  
  |t(z)| \lesssim 
\frac{|\nabla|f|(z)|^{2}}
{|f(z)|^{2}\,\bigl(\log|f(z)|^{2}\bigr)^{2}\,
\bigl(\log(-\log|f(z)|^{2})\bigr)^{2}} .
   \end{align}
Applying the coarea formula (\ref{eq:coarea-ti}), we obtain
\[
\begin{aligned}
\int_{U}|t|\, dV
&\lesssim \int_{U} 
\frac{|\nabla|f||^{2}}
{|f|^{2}\,(\log|f|^{2})^{2}\,
(\log(-\log|f|^{2}))^{2}}\, dV \\
&\lesssim\int_{0}^{1/2}\Bigl(
\int_{\{z\in U:|f|=t\}}
\frac{|\nabla|f||^{2}}
{|f|^{2}(\log|f|^{2})^{2}(\log(-\log|f|^{2}))^{2}}
\cdot\frac{1}{|\nabla|f||}\, dS\Bigr) dt \\
&\lesssim\int_{0}^{1/2}\frac{1}{t^{2}(\log t^{2})^{2}}
\Bigl(\int_{\{z\in U:|f|=t\}}|\nabla|f||\, dS\Bigr) dt \\
&\lesssim\int_{0}^{1/2}\frac{1}{t^{2}(\log t^{2})^{2}\,
(\log(-\log|t|^{2}))^{2}}
\Bigl(\int_{\{z\in U:|f|=t\}}|\partial f|\, dS\Bigr) dt \\
&\lesssim\int_{0}^{1/2}
\frac{1}{t(\log t)^{2}\,
(\log(-\log|t|^{2}))^{2}}\, dt,
\qquad \text{by part (1) of Theorem \ref{t02}}, \\
&< \infty .
\end{aligned}
\]

Thus \( t\in L^{1}_{\mathrm{loc}}(U) \).

To prove that \( t \notin L^{p}_{\mathrm{loc}}(U) \) for any \( p > 1 \) near \( Z(f) \), we decompose \( t = t_{1} + t_{2} \), where
\[
\begin{aligned}
t_{1}(z) &:=
\frac{4|\nabla|f|(z)|^{2}}
{|f(z)|^{2}\,\bigl(\log|f(z)|^{2}\bigr)^{2}\,
\bigl(\log(-\log|f(z)|^{2})\bigr)^{2}},\\[2mm]
t_{2}(z) &:=
\frac{8\,|\nabla|f|(z)|^{2}}
{|f(z)|^{2}\,\bigl(\log|f(z)|^{2}\bigr)^{2}\,
\bigl(\log(-\log|f(z)|^{2})\bigr)^{3}}.
\end{aligned}
\]
Since both \( t_{1} \) and \( t_{2} \) are strictly positive everywhere outside \( Z(f) \), we have the pointwise lower bound \( t^p \ge t_1^p \). Therefore, it suffices to demonstrate that \( t_1 \notin L^p_{\mathrm{loc}}(U) \).

Without loss of generality, we may assume \( 0 \in Z(f) \). Utilizing Lemma \ref{lem:local_monomial} we can select a sufficiently small neighborhood \( V \subset U \) intersecting the regular part of \( Z(f) \) and perform a biholomorphic change of coordinates \( w = \Phi(z) \) such that locally \( f(z(w)) = w_1^m \) for some integer \( m \ge 1 \). 

Under this smooth diffeomorphism, the measure \( dV_z \) is equivalent to \( dV_w \). Furthermore, the critical singular term transforms asymptotically as \( \frac{|\nabla |f||^2}{|f|^2} \approx \frac{|w_1|^{2m-2}}{|w_1|^{2m}} = \frac{1}{|w_1|^2} \)  up to bounded, strictly positive smooth factors. Thus, in the local coordinates on \( \Phi(V) \), \( t_1 \) obeys the asymptotic lower bound 
\[
t_1(w) \gtrsim \frac{1}{|w_1|^2 \, \bigl|\log|w_1|\bigr|^2 \, \bigl|\log(-\log|w_1|^2)\bigr|^2}.
\]

To test its \( L^p \) integrability, we evaluate the integral over a small polydisc \( \Delta^n = \Delta_1 \times \Delta' \subset \Phi(V) \), where \( \Delta_1 = \{w_1 \in \mathbb{C} \mid |w_1| < r_0\} \). Applying Fubini's theorem to integrate out the tangential variables \( w' \), we obtain 
\[
\begin{aligned}
\int_{V}|t_{1}|^{p}\, dV_z 
&\approx \int_{\Phi(V)} |t_1|^p \, dV_w \\
&\gtrsim \text{Vol}(\Delta') \int_{\Delta_1}
\frac{1}{|w_{1}|^{2p}\,\bigl|\log|w_{1}|\bigr|^{2p}\,
\bigl|\log(-\log|w_{1}|^{2})\bigr|^{2p}}\, dV_{w_1}.
\end{aligned}
\]
Switching to polar coordinates \( w_1 = s e^{i\theta} \) on the disk \( \Delta_1 \), the inner integral evaluates to 
\[
\int_{0}^{2\pi} \int_{0}^{r_{0}}
\frac{1}{s^{2p}\,\bigl|\log s\bigr|^{2p}\,
\bigl|\log(-\log s^{2})\bigr|^{2p}} \, s \, ds \, d\theta 
\gtrsim \int_{0}^{r_{0}}
\frac{1}{s^{2p-1}\,\bigl|\log s\bigr|^{2p}\,
\bigl|\log(-\log s^{2})\bigr|^{2p}}\, ds.
\]
For any \( p > 1 \), the exponent of \( s \) in the denominator satisfies \( 2p - 1 > 1 \). This power-law singularity is strictly stronger than \( s^{-1} \), meaning the logarithmic damping factors in the denominator are entirely insufficient to overcome the singularity at \( s=0 \). Consequently, the integral strongly diverges. 

This establishes that \( \int_V |t_1|^p \, dV_z = \infty \), meaning \( t_1 \notin L^p_{\mathrm{loc}}(U) \). Since \( t \ge t_1 > 0 \), it directly follows that \( t \notin L^{p}_{\mathrm{loc}}(U) \) for any \( p > 1 \) near \( Z(f) \).

\end{proof}

\begin{proof} [Proof of $3$ of Theorem \ref{th:maintheorem2}]

Let \( \vec n \) be the outward unit normal vector as defined in (\ref{eq:tidu}).
Applying Green's theorem to each connected component of \( U_{\varepsilon} = \{ z\in U : |f(z)|>\varepsilon \} \) and summing up, we obtain
\[
\begin{aligned}
\int_{U_{\varepsilon}} \bigl( v\Delta\varphi - \varphi\Delta v \bigr)\, dV
&= \int_{\partial U_{\varepsilon}} ( v\nabla\varphi - \varphi\nabla v )\cdot \vec n \, dS \\
&= \int_{\{z \in U: |f| = \varepsilon\}} ( v\nabla\varphi - \varphi\nabla v )\cdot \vec n \, dS \\
&\quad + \int_{\partial U \setminus (\partial U \cap \{ |f|\le\varepsilon \})}
( v\nabla\varphi - \varphi\nabla v )\cdot \vec n \, dS \\
&= \int_{\{z \in U: |f| = \varepsilon\}} ( v\nabla\varphi - \varphi\nabla v )\cdot \vec n \, dS ,
\end{aligned}
\]
the last equality holding because \( \varphi\in C_{0}^{\infty}(U) \) vanishes near \( \partial U \).  
Denote
\[
\Lambda_{\varepsilon} = \int_{\{z \in U: |f| = \varepsilon\}} v\, \nabla\varphi\cdot\vec n \, dS , \qquad
\Xi_{\varepsilon} = -\int_{\{z \in U: |f| = \varepsilon\}} \varphi\, \nabla v\cdot\vec n \, dS .
\]

To complete the proof we must verify that both \( \Lambda_{\varepsilon} \) and \( \Xi_{\varepsilon} \) tend to zero as \( \varepsilon\to 0 \); the dominated convergence theorem will then yield the desired identity.
From Theorem \ref{th:levelsurface} we have \( |\{ z\in U : |f|=\varepsilon \}| = O(\varepsilon^{\alpha}) \) for some \( \alpha>0 \). Hence
\[
\begin{aligned}
|\Lambda_{\varepsilon}|
&\le \max_{U} |\nabla\varphi\cdot\vec n| \int_{\{z \in U: |f| = \varepsilon\}} |v|\, dS \\
&\lesssim \frac{1}{\log(-\log\varepsilon)}\,
\bigl|{\{z \in U: |f| = \varepsilon\}}\bigr| \\
&\to 0 \qquad (\varepsilon\to 0).
\end{aligned}
\]

For \( \Xi_{\varepsilon} \), note that
\[
|\Xi_{\varepsilon}|
\le \max_{U}|\varphi| \int_{\{z \in U: |f| = \varepsilon\}}|\nabla v\cdot\vec n|\, dS .
\]

Observe that
\[
|\nabla v\cdot\vec n|
\lesssim \frac{|\nabla|f||}{|f|\,\log|f|^{2}\,\bigl|\log(-\log|f|^{2})\bigr|^{2}} .
\]

Consequently,
\[
\begin{aligned}
|\Xi_{\varepsilon}|
&\lesssim \frac{1}{\varepsilon\,|\log\varepsilon|\,\bigl|\log(-\log\varepsilon)\bigr|^{2}}
\int_{\{z \in U: |f| = \varepsilon\}} |\nabla|f||\, dS \\
&\lesssim \frac{1}{\varepsilon\,|\log\varepsilon|\,\bigl|\log(-\log\varepsilon)\bigr|^{2}}
\int_{\{z \in U: |f| = \varepsilon\}} |\partial f|\, dS \\
&\lesssim \frac{1}{|\log\varepsilon|\,\bigl|\log(-\log\varepsilon)\bigr|^{2}}
\to 0 \qquad (\varepsilon\to 0),
\end{aligned}
\]
where the last estimate uses part (1) of Theorem \ref{t02}.

Thus \( \Lambda_{\varepsilon}\to0 \) and \( \Xi_{\varepsilon}\to0 \), which together with the dominated convergence theorem establishes the required weak formulation.

\end{proof}

We restate the 4 of Theorem \ref{th:maintheorem2}.
\begin{theorem}[Second-order Sobolev regularity] \label{thm:W21_regularity}
Let $U \subset \mathbb{C}^n$ be a domain. Suppose $f$ is a non-constant holomorphic function defined in a neighborhood of $\overline{U}$ such that \( |f |<\frac12 \) on  $  U$. Define the real-valued function $v: U \to \mathbb{R}$ by
\begin{equation} \label{eq:v_def}
    v(z) := 
    \begin{cases} 
    \displaystyle \frac{1}{\log(-\log|f(z)|^2)}, & \text{if } z \in U \setminus Z(f), \\[10pt]
    0, & \text{if } z \in Z(f),
    \end{cases}
\end{equation}
where $Z(f) = \{z \in U : f(z) = 0\}$ is the zero set of $f$. Then $v$ belongs to the local Sobolev space $W^{2,1}_{\mathrm{loc}}(U)$. That is, the second-order weak derivatives of $v$ exist across the singular set $Z(f)$ and are locally absolutely integrable.
\end{theorem}

To prove Theorem \ref{thm:W21_regularity}, we must show that the classical second-order derivatives of $v$ are locally integrable on $U \setminus Z(f)$, and that no Dirac mass is generated across the singular set $Z(f)$ in the weak sense. Direct calculation shows that the most singular term in the complex Hessian of $v$ is dominated by the bounding function $\mathcal{A}_1 := \frac{|\nabla^2 f|}{|f| L(f)}$, where $L(f) := \big| \log|f|^2 \big| \bigl(\log(-\log|f|^2)\bigr)^2$.

\begin{remark} 
To establish the local integrability of the second-order derivative terms, one might naturally attempt to employ the coarea formula and the \L{}ojasiewicz gradient inequality---the very approach that successfully handles the lower-order terms. However, this naive strategy fundamentally fails for the complex Hessian term $\mathcal{A}_1$. If we naively bounded the complex Hessian in the numerator by a constant---thereby ignoring the vanishing behavior of $\nabla^2 f$---and applied the coarea formula alongside the \L{}ojasiewicz inequality $|\partial f| \ge C|f|^\alpha$ for some $0 \le \alpha < 1$, we would obtain
\begin{align*}
\int_{U}|A_{1}|\,dV
&\lesssim\int_{U}
\frac{1}{|f|\,|\log|f||\,\bigl(\log(-\log|f|^{2})\bigr)^{2}}\,dV \\
&\lesssim\int_{0}^{1/10}\Bigl(
\int_{\{z\in U:|f|=t\}}
\frac{1}
{|f|\,|\log|f||\,(\log(-\log|f|^{2}))^{2}}
\cdot\frac{1}{|\nabla|f||}\,dS\Bigr)dt \\
&\lesssim\int_{0}^{1/10}\Bigl(
\int_{\{z\in U:|f|=t\}}
\frac{1}
{|f|\,|\log|f||\,(\log(-\log|f|^{2}))^{2}}
\cdot\frac{1}{|\partial f|}\,dS\Bigr)dt \\
&\lesssim\int_{0}^{1/10}\Bigl(
\int_{\{z\in U:|f|=t\}}
\frac{1}
{|f|\,|\log|f||\,(\log(-\log|f|^{2}))^{2}}
\cdot\frac{1}{t^\alpha}\,dS\Bigr)dt, \qquad  \text{since} \ |\partial f| \ge C |f|
^\alpha,  \\ 
&\lesssim\int_{0}^{1/10}
\frac{1}{t^{1+\alpha}\,|\log t|\,(\log(-\log t^{2}))^{2}}
|\{z\in U:|f|=t\}|dt \\
&\lesssim\int_{0}^{1/10}
\frac{t^{1-\alpha}}{t^{1+\alpha}\,|\log t|\,(\log(-\log t^{2}))^{2}}
dt, \qquad  \text{by Theorem \ref{th:levelsurface},}  \\
&\lesssim\int_{0}^{1/10}
\frac{1}{t^{2\alpha}\,|\log t|\,(\log(-\log t^{2}))^{2}}\,dt.
\end{align*}
This final one-dimensional integral would converge only if $0\leq \alpha < 1/2$, and would strictly diverge for $\alpha \ge 1/2$. Because the \L{}ojasiewicz exponent $\alpha$ can easily exceed $1/2$ for holomorphic functions with higher-order singularities (for example, the cusp singularity $f(z_1, z_2) = z_1^2 - z_2^3$ yields $\alpha = 2/3$), this crude upper bound is unworkable. 

The fundamental flaw in this approach lies in replacing $|\nabla^2 f|$ with a constant, which completely misses the internal vanishing order of the Hessian near the singular set $Z(f)$. To rigorously establish integrability, we must capture the precise algebraic cancellations between the vanishing numerator and the singular denominator. This necessitates the use of Hironaka's Resolution of Singularities.
\end{remark}

Following this resolution strategy, we first establish integrability for the pure monomial case, where the exact algebraic cancellations can be computed explicitly.

\begin{lemma}[Monomial case]\label{lem:monomial_A1}
Let $U = \{z \in \mathbb{C}^n : |z_m| < r_0\}$ be a polydisc centered at the origin, where $r_0 \in (0, 1/e)$ is chosen sufficiently small. Suppose $f(z) = \prod_{i=1}^n z_i^{k_i}$ is a non-constant monomial with non-negative integers $k_i$. Define the logarithmic damping function 
\begin{equation} \label{eq:lf}
L(f) := \big| \log|f|^2 \big| \bigl(\log(-\log|f|^2)\bigr)^2,
\end{equation}
and define the function $\mathcal{A}_1$ associated with the real Hessian norm of $f$ on $U \setminus Z(f)$ by
\begin{equation}\label{eq:A1}
\mathcal{A}_1 := \frac{|\nabla^2 f|}{|f| L(f)}.
\end{equation}
where $|\nabla^2 f|$ is the Frobenius norm of the real Hessian matrix of $f$. Then $\mathcal{A}_1$ is locally integrable on $U$, namely, $\mathcal{A}_1 \in L^1_{\rm{loc}}(U)$.
\end{lemma}

\begin{proof}
By our notation convention, the pointwise norm of the real Hessian is bounded, up to a dimensional constant, by the sum of the absolute values of the complex second-order derivatives. For the monomial $f(z) = \prod_{i=1}^n z_i^{k_i}$, the first-order complex derivatives are $\frac{\partial f}{\partial z_i} = \frac{k_i}{z_i} f(z)$. We evaluate the second-order complex derivatives in two distinct cases:

\noindent  Case I: Off-diagonal terms  ($i \neq j$). 
We have
\begin{align*}
\frac{\partial^2 f}{\partial z_i \partial z_j} &= \frac{\partial}{\partial z_i} \left( \frac{k_j}{z_j} f(z) \right) = \frac{k_i k_j}{z_i z_j} f(z).
\end{align*}
Taking the absolute value yields
\begin{align*}
\left| \frac{\partial^2 f}{\partial z_i \partial z_j} \right| &= \frac{k_i k_j}{|z_i| |z_j|} |f(z)|.
\end{align*}

\noindent  Case II: Diagonal terms  ($i = j$). 
We have
\begin{align*}
\frac{\partial^2 f}{\partial z_i^2} &= \frac{\partial}{\partial z_i} \left( \frac{k_i}{z_i} f(z) \right) = \frac{k_i(k_i - 1)}{z_i^2} f(z).
\end{align*}
Taking the absolute value gives the bound
\begin{align*}
\left| \frac{\partial^2 f}{\partial z_i^2} \right| &\le \frac{k_i^2}{|z_i|^2} |f(z)|.
\end{align*}

Summing these bounds, the total real Hessian norm satisfies
$$|\nabla^2 f| \le C |f(z)| \sum_{i=1}^n \sum_{j=1}^n \frac{k_i k_j}{|z_i| |z_j|},$$
for some dimensional constant $C > 0$. Dividing by $|f|L(f)$, we obtain the pointwise estimate for $\mathcal{A}_1$:
$$\mathcal{A}_1 \le C \sum_{i=1}^n \sum_{j=1}^n \frac{k_i k_j}{|z_i| |z_j| L(f)}.$$

To establish that $\mathcal{A}_1 \in L^1_{\rm{loc}}(U)$, it suffices to prove that each individual integral 
$$I_{i,j} := \int_U \frac{1}{|z_i| |z_j| L(f)} \, dV$$ 
is finite. Passing to polar coordinates $z_m = r_m e^{i\theta_m}$, the Euclidean volume element becomes $dV = \prod_{m=1}^n r_m \, dr_m \, d\theta_m$. Note that $w := -\log(|f|^2) = -2\sum_{m=1}^n k_m \log r_m$. Since $r_m < r_0 < 1/e$, we have $w > 0$ strictly, which yields $L(f) = w(\log w)^2$. Integrating out the angular variables contributes a constant factor $(2\pi)^n$, reducing the problem to bounding the radial integral:
$$I_{i,j} \lesssim \int_0^{r_0} \cdots \int_0^{r_0} \frac{1}{r_i r_j w(\log w)^2} \left( \prod_{m=1}^n r_m \right) dr_1 \cdots dr_n.$$ 

We analyze this integral according to the two subcases.

  Subcase I: Off-diagonal terms  ($i \neq j$). 
The singular factors $r_i$ and $r_j$ in the denominator are exactly cancelled by the corresponding radial factors in the volume element. The integrand simplifies to
$$\frac{1}{w(\log w)^2} \prod_{m \neq i,j} r_m.$$
Since $r_m \le r_0 < 1/e$, the logarithmic term satisfies $w \ge -2 \log r_0 > 2$. The function $w \mapsto (w \log^2 w)^{-1}$ is strictly decreasing and bounded uniformly away from its asymptote on the interval $[2, \infty)$. Because the domain of integration is bounded and the integrand has no singularities, the integral $I_{i,j}$ converges trivially.

   Subcase II: Diagonal terms  ($i = j$). 
In this scenario, the volume element provides only one factor of $r_i$, leaving a simple pole in the radial variable:
$$I_{i,i} \lesssim \int_0^{r_0} \cdots \int_0^{r_0} \frac{1}{r_i w(\log w)^2} \left( \prod_{m \neq i} r_m \right) dr_1 \cdots dr_n.$$
By Fubini's theorem, we isolate the integration with respect to $r_i$, holding the other variables $r_m$ fixed. Define $C_i := -2 \sum_{m \neq i} k_m \log r_m \ge 0$, so that $w = -2k_i \log r_i + C_i$. Applying the substitution $u = -2k_i \log r_i + C_i$, we have $du = -2k_i \frac{dr_i}{r_i}$. As $r_i \to 0^+$, $u \to +\infty$, and when $r_i = r_0$, the lower bound becomes $u_0 := -2k_i \log r_0 + C_i > 0$. The inner one-dimensional integral transforms into:
$$\int_0^{r_0} \frac{dr_i}{r_i w(\log w)^2} = \frac{1}{2k_i} \int_{u_0}^{\infty} \frac{du}{u(\log u)^2}.$$
Since the antiderivative of $(u \log^2 u)^{-1}$ is $-(\log u)^{-1}$, this improper integral converges strictly to $\frac{1}{2k_i \log(u_0)} < \infty$. Integrating the remaining bounded variables $r_m$ over $[0, r_0]$ ensures that $I_{i,i}$ is finite.

Because both the off-diagonal and diagonal integrals $I_{i,j}$ are bounded, their finite sum converges. We conclude that $\int_U \mathcal{A}_1 \, dV < \infty$, which proves $\mathcal{A}_1 \in L^1_{\rm{loc}}(U)$.
\end{proof}

With the monomial case established, we can now prove the integrability for arbitrary holomorphic functions by employing Hironaka's Theorem. The crux of the proof is tracking the transformation of the complex Hessian under the proper birational pullback, revealing a crucial geometric cancellation.

\begin{proposition}[Integrability of the Hessian term]\label{prop:general_A1}
Let $U \subset \mathbb{C}^n$ be a domain and $f$ be a non-constant holomorphic function defined in a neighborhood of $\overline{U}$ such that \( |f |<\frac12 \) on  $U$. If $L(f)$ and $\mathcal{A}_1$ are defined on $U \setminus Z(f)$ as in \eqref{eq:lf} and \eqref{eq:A1} respectively, then $\mathcal{A}_1 \in L^1_{\rm{loc}}(U)$.
\end{proposition}

\begin{proof}
To establish the local integrability $\int_K \mathcal{A}_1 dV_z < \infty$ over any compact subset $K \subset U$, we invoke Hironaka's Resolution of Singularities theorem. There exists a smooth complex manifold $\tilde{X}$ and a proper, birational holomorphic map $\pi: \tilde{X} \rightarrow U$ such that the pullback $\tilde{f} = f \circ \pi$ has only normal crossings. Crucially, $\pi$ restricts to a biholomorphism from $\tilde{X} \setminus \pi^{-1}(Z(f))$ onto $U \setminus Z(f)$. 

This means that locally on $\tilde{X}$, in coordinates $w = (w_1, \dots, w_n)$, the function behaves exactly like a pure monomial $\tilde{f}(w) = w^K$ (absorbing any non-vanishing holomorphic units into the coordinates via a local biholomorphism).

We pull the integral back to the resolution space $\tilde{X}$ using $\pi$. The volume form transforms as $dV_z = |J_\pi|^2 dV_w$, where $J_\pi = \det(D\pi)$ is the complex Jacobian determinant of $\pi$. By the chain rule, the gradients satisfy $\nabla_w \tilde{f} = (D\pi)^T (\nabla_z f \circ \pi)$, which yields:
\begin{equation}\label{eq:gradient_pullback}
\nabla_z f \circ \pi = (D\pi)^{-T} \nabla_w \tilde{f} \qquad \text{on} \qquad \tilde{X} \setminus \pi^{-1}(Z(f)).
\end{equation}

Let $M = D\pi$ be the complex Jacobian matrix of the resolution map. The inverse matrix is given by Cramer's rule:
\begin{align*}
(D\pi)^{-1} = \frac{1}{\det(D\pi)} \text{adj}(D\pi) \qquad \text{on} \qquad \tilde{X} \setminus \pi^{-1}(Z(f)), 
\end{align*}
where $\text{adj}(D\pi)$ denotes the adjugate matrix of $D\pi$. Because the map $\pi$ is a smooth holomorphic map on the resolution manifold $\tilde{X}$, every entry in $D\pi$ is a smooth, well-defined holomorphic function. Consequently, on the compact preimage $\pi^{-1}(K)$, the norm of the adjugate matrix is uniformly bounded, i.e., $\|\text{adj}(D\pi)\| \le C$. This gives:
\begin{align*} 
\|(D\pi)^{-1}\| \le \frac{C}{|J_\pi|} \qquad \text{on} \qquad \tilde{X} \setminus \pi^{-1}(Z(f)).
\end{align*}

To bound the pulled-back Hessian $|\nabla_z^2 f \circ \pi|$, we differentiate the gradient relation \eqref{eq:gradient_pullback} with respect to the coordinates $w_i$, obtaining  the known relation 
for the (holomorphic) Hessian matrices:
\begin{equation}\label{eq:hessian_relation}
H_w(\tilde{f}) = (D\pi)^T H_z(f) (D\pi) + \sum_{k=1}^n (\partial_{z_k} f \circ \pi) H_w(\pi_k).
\end{equation}
To isolate $H_z(f)$, we multiply from the left by $(D\pi)^{-T}$ and from the right by $(D\pi)^{-1}$:
\begin{eqnarray*}
H_z(f) = (D\pi)^{-T} \left[ H_w(\tilde{f}) - \sum_{k=1}^n (\partial_{z_k} f \circ \pi) H_w(\pi_k) \right] (D\pi)^{-1}  \qquad  \text{on} \qquad   \tilde{X} \setminus \pi^{-1}(Z(f)).
\end{eqnarray*}
Taking the matrix norm and applying the inverse bound $\|(D\pi)^{-1}\| \le C/|J_\pi|$, we pick up a factor of $1/|J_\pi|^2$:
\begin{eqnarray*}
|\nabla_z^2 f \circ \pi| \le \frac{C}{|J_\pi|^2} \left( |\nabla_w^2 \tilde{f}| + \sum_{k=1}^n |\partial_{z_k} f \circ \pi| |\nabla_w^2 \pi_k| \right) \qquad  \text{on} \qquad   \tilde{X} \setminus \pi^{-1}(Z(f)).
\end{eqnarray*}

Substituting this bound into our integral reveals a crucial geometric cancellation: the singular factor $1/|J_\pi|^2$ generated by the inverse Jacobian matrices perfectly cancels the $|J_\pi|^2$ factor originating from the volume element $dV_z$. On the other hand, since $Z(f)$ is of complex dimension at most $n-1$, its real dimension is at most $2n-2$, thus it is a null set with respect to the Lebesgue measure $dV_z$. Therefore, the integral is controlled by two terms:
\begin{align*}
\int_K \frac{|\nabla_z^2 f|}{|f| L(f)} dV_z &= \int_{K\setminus Z(f)} \frac{|\nabla_z^2 f|}{|f| L(f)} dV_z\\
& = \int_{\pi^{-1}(K\setminus Z(f))} \frac{|\nabla_z^2 f \circ \pi|}{|\tilde{f}| L(\tilde{f})} |J_\pi|^2 dV_w \\ 
& = \int_{\pi^{-1}(K)} \frac{|\nabla_z^2 f \circ \pi|}{|\tilde{f}| L(\tilde{f})} |J_\pi|^2 dV_w\\ 
&\le C \int_{\pi^{-1}(K)} \frac{|\nabla_w^2 \tilde{f}|}{|\tilde{f}| L(\tilde{f})} dV_w + C \int_{\pi^{-1}(K)} \frac{\sum_{k=1}^n |\partial_{z_k} f \circ \pi| |\nabla_w^2 \pi_k|}{|\tilde{f}| L(\tilde{f})} dV_w.
\end{align*}

The first term involves only the Hessian of the pullback $\tilde{f}$. Since $\tilde{f}$ is locally a monomial, this term reduces identically to the pure monomial case established in Lemma \ref{lem:monomial_A1}.

For the second term, we bound the derivatives of the coordinate components $\pi_k = z_k \circ \pi$. Since $\pi$ is constructed locally via a finite sequence of blow-ups with smooth centers, its coordinate components take the form of pure monomials in the standard affine charts: $\pi_k(w) = w_1^{m_1} \cdots w_n^{m_n}$. Direct differentiation yields:
$$ |\nabla_w^2 \pi_k| \le C \sum_{i,j} \frac{|\pi_k(w)|}{|w_i| |w_j|}. $$
Applying this to the second integral, the integrand is bounded by:
$$ \frac{|\partial_{z_k} f \circ \pi| |\nabla_w^2 \pi_k|}{|\tilde{f}| L(\tilde{f})} \le \sum_{i,j} \frac{1}{|w_i| |w_j| L(\tilde{f})} \left( \frac{|\partial_{z_k} f \circ \pi| |\pi_k(w)|}{|\tilde{f}|} \right). $$

It remains to bound the ratio inside the parenthesis. Define $g_k(z) = z_k \frac{\partial f}{\partial z_k}$. Expanding $f(z)$ as a Taylor series $f(z) = \sum c_A z^A$, we observe that $g_k(z) = \sum c_A a_k z^A$. The function $g_k(z)$ contains a subset of the exact same monomial terms as $f(z)$, scaled only by the exponent integers $a_k$. 

By the principalization of ideals, the pullback $f(\pi(w)) = w^K u(w)$, where $w^K$ represents the minimum vanishing order along the exceptional divisors for all terms in the infinite sum $\sum c_A (\pi(w))^A$. Because every term in the pullback $g_k \circ \pi$ is proportional to a term in $f \circ \pi$, the exact same monomial $w^K$ must perfectly factor out of $g_k(\pi(w))$. Thus, we can write $g_k(\pi(w)) = w^K h(w)$ for some holomorphic function $h(w)$.

Evaluating the ratio on the resolution space yields:
$$ \frac{|\partial_{z_k} f \circ \pi| |\pi_k|}{|f \circ \pi|} = \frac{|g_k \circ \pi|}{|f \circ \pi|} = \left| \frac{w^K h(w)}{w^K u(w)} \right| = \left| \frac{h(w)}{u(w)} \right| \le C', $$
where the uniform bound holds because $u(w)$ is a non-vanishing unit. 

Consequently, the entire integrand of the second term is controlled by $C'' \sum_{i,j} \frac{1}{|w_i| |w_j| L(\tilde{f})}$, which possesses the exact same integrable pole structure as the pure monomial Hessian. Both terms are therefore locally integrable on $\tilde{X}$, and we conclude that $\mathcal{A}_1 \in L^1_{\rm{loc}}(U)$.
\end{proof}

We are now fully equipped to prove the main regularity theorem.

\begin{proof}[Proof of Theorem \ref{thm:W21_regularity}]
In light of part 3 of Theorem \ref{th:maintheorem2}, it remains to show that the second-order weak derivatives
\[
\frac{\partial^{2}v}{\partial z_{i}\partial z_{j}},\quad 
\frac{\partial^{2}v}{\partial \bar z_{i}\partial \bar z_{j}},\quad 
\frac{\partial^{2}v}{\partial z_{i}\partial \bar z_{j}}
\]
exist on \(U\) and belong to \(L^{1}_{\mathrm{loc}}(U)\) near the zero set \(Z(f)\).  
Because these derivatives are mutually conjugate, it is sufficient to verify the claim for \(\frac{\partial^{2}v}{\partial z_{i}\partial z_{j}}\) and $\frac{\partial^{2}v}{\partial z_{i}\partial \bar z_{j}}$.

First, we define the candidate weak derivatives by
\[
v_{i,j}(z) :=
\begin{cases} 
\displaystyle\frac{\partial^{2}v}{\partial z_{i}\partial z_{j}}, &  z\notin Z(f),  \\[4pt]
0, & z\in Z(f); 
\end{cases}
\]
analogous definitions are made for \(v_{i,\bar j}\) and \(v_{\bar i,\bar j}\). We compute the classical second-order derivative on \(U\setminus Z(f)\) explicitly:
\begin{align} \label{eq:A_1}
\frac{\partial^{2}v}{\partial z_{i}\partial z_{j}}
&= -\frac{\partial_{z_{j}}(\partial_{z_{i}}f)}{f\,\log(|f|^{2})\,\bigl(\log(-\log|f|^{2})\bigr)^{2}} \nonumber \\ 
&\quad + \frac{\partial_{z_{i}}f\,\partial_{z_{j}}f}{f^{2}}\, \frac{1}{\log(|f|^{2})\,\bigl(\log(-\log|f|^{2})\bigr)^{2}} \nonumber \\ 
&\quad +\frac{\partial_{z_{i}}f\,\partial_{z_{j}}f}{ f ^{2}}\, \frac{1}{(\log(|f|^{2}))^{2}\,\bigl(\log(-\log|f|^{2})\bigr)^{2}} \nonumber \\ 
&\quad +2\frac{\partial_{z_{i}}f\,\partial_{z_{j}}f}{ f ^{2}}\, \frac{1}{(\log(|f|^{2}))^{2}\,\bigl(\log(-\log|f|^{2})\bigr)^{3}} \nonumber \\ 
& =: A_{1}+A_{2}+A_{3}+A_{4}.
\end{align}

Since \(v\) is continuous on \(U\), it belongs to \(L^{\infty}_{\mathrm{loc}}(U)\). Consequently, to prove that \(v_{i,j}\in L^{1}_{\mathrm{loc}}(U)\), it is enough to show that each term \(A_{k}\) (\(k=1,2,3,4\)) lies in \(L^{1}_{\mathrm{loc}}(U)\).

For the most singular term \(A_{1}\), Proposition \ref{prop:general_A1} rigorously guarantees that \(A_{1} \in L^{1}_{\mathrm{loc}}(U)\). Note that the terms $A_3$ and $A_4$ are pointwise bounded by:
\[
|A_{3}|, |A_{4}| \le C\left|\frac{\partial f}{f\log|f|^{2}}\right|^{2}.
\]
By the proof of Theorem \ref{l10}, we have already established that \(A_{3}, A_{4} \in L^{1}_{\mathrm{loc}}(U)\).

For \(A_{2}\), we apply the coarea formula together with the energy estimate of the level sets:
\begin{align*}
\int_{U}|A_{2}|\,dV
&\lesssim\int_{U} \frac{|\nabla|f||^{2}}{|f|^{2}\,|\log|f||\,\bigl(\log(-\log|f|^{2})\bigr)^{2}}\,dV \\
&\lesssim\int_{0}^{1/10}\Bigl( \int_{\{z\in U:|f|=t\}} \frac{|\nabla|f||^{2}} {|f|^{2}\,|\log|f||\,(\log(-\log|f|^{2}))^{2}} \cdot\frac{1}{|\nabla|f||}\,dS\Bigr)dt \\
&\lesssim\int_{0}^{1/10} \frac{1}{t^{2}\,|\log t|\,(\log(-\log t^{2}))^{2}} \Bigl(\int_{\{z\in U:|f|=t\}}|\nabla|f||\,dS\Bigr)dt \\
&\lesssim\int_{0}^{1/10} \frac{1}{t^{2}\,|\log t|\,(\log(-\log t^{2}))^{2}} \Bigl(\int_{\{z\in U:|f|=t\}}|\partial f|\,dS\Bigr)dt \\
&\lesssim\int_{0}^{1/10} \frac{1}{t\,|\log t|\,(\log(-\log t^{2}))^{2}}\,dt, \qquad \text{  by part (1) of Theorem \ref{t02}},
\end{align*}
which is a finite integral since its antiderivative is $-(\log(-\log t^2))^{-1}$. Hence \(A_{2}\in L^{1}_{\mathrm{loc}}(U)\). 

Having established that $A_1, A_2, A_3, A_4 \in L^1_{\mathrm{loc}}(U)$, we conclude that the candidate derivative $v_{i,j}$ is locally integrable. The same arguments show that \(v_{i,\bar j}\) and \(v_{\bar i,\bar j}\) also belong to \(L^{1}_{\mathrm{loc}}(U)\).

To complete the proof of $v \in W^{2,1}_{\mathrm{loc}}(U)$, it remains to verify that $v_{i,j}$ satisfies the definition of a weak derivative; that is, for any test function $\varphi \in C_0^\infty(U)$,  
\[
\int_{U} v \frac{\partial^2 \varphi}{\partial z_i \partial z_j } dV= \int_{U} v_{i,j} \varphi dV.
\]

Set \(U_{\varepsilon}:=\{z\in U:|f(z)|>\varepsilon\}\). On \(U_{\varepsilon}\) the function \(v\) is smooth, so by integrating by parts we have
\begin{align} \label{eq:t42}
\int_{U_{\varepsilon}}\varphi\,\frac{\partial^{2}v}{\partial z_{i}\partial z_{j}}\,dV
&= \int_{U_{\varepsilon}}  
\frac{\partial}{\partial z_{i}}\!\Bigl(\frac{\partial v}{\partial z_{j}}\Bigr)\,\varphi\,dV \nonumber \\ 
&= \int_{\{z \in U: |f| = \varepsilon\}}
\frac{\partial v}{\partial z_{j}}\,\varphi\,
\frac{\partial\rho}{\partial z_{i}}\frac{dS}{|\partial\rho|}
- \int_{U_{\varepsilon}}
\frac{\partial v}{\partial z_{j}}\,
\frac{\partial\varphi}{\partial z_{i}}\,dV,  
\end{align}
where \(\rho=\varepsilon^{2}-|f|^{2}\) is a defining function of \(U_{\varepsilon}\).
Similarly, integrating the second term by parts again yields
\begin{align} \label{eq:t41}
\int_{U_{\varepsilon}}\frac{\partial v}{\partial z_{j}}\,
\frac{\partial\varphi}{\partial z_{i}}\,dV
&= \int_{\{z \in U: |f| = \varepsilon\}} v\,
\frac{\partial\varphi}{\partial z_{i}}\,
\frac{\partial\rho}{\partial z_{j}}\frac{dS}{|\partial\rho|}
- \int_{U_{\varepsilon}} v\,
\frac{\partial^{2}\varphi}{\partial z_{i}\partial z_{j}}\,dV.  
\end{align}
Inserting \eqref{eq:t41} into \eqref{eq:t42}, we obtain
\begin{align}\label{eq:weak3}
\int_{U_\varepsilon}\varphi \frac{\partial^2 v}{\partial z_i \partial z_j} dV
&= \int_{\{z \in U: |f| = \varepsilon\}}\varphi \frac{\partial v}{\partial z_j} \frac{\partial \rho }{ \partial z_i} \frac{dS}{|\partial \rho|} 
   - \int_{\{z \in U: |f| = \varepsilon\}} v \frac{\partial \varphi}{\partial z_i} \frac{\partial \rho }{ \partial z_j} \frac{dS}{|\partial \rho|} \nonumber \\
   &\quad + \int_{U_\varepsilon} v \frac{\partial^2 \varphi}{\partial z_i \partial z_j} dV. 
\end{align}
Thus, it remains to show that the boundary integrals vanish as $\varepsilon \to 0$:
\begin{align}\label{eq:weak4}
\lim_{\varepsilon \to 0} \int_{\{z \in U: |f| = \varepsilon\}} \varphi \frac{\partial v}{\partial z_j} \frac{\partial \rho }{ \partial z_i} \frac{dS}{|\partial \rho|} = 0, 
\end{align}
\begin{align}\label{eq:weak5}
\lim_{\varepsilon \to 0} \int_{\{z \in U: |f| = \varepsilon\}} v \frac{\partial \varphi}{\partial z_i} \frac{\partial \rho }{ \partial z_j} \frac{dS}{|\partial \rho|} = 0. 
\end{align}
Letting $\varepsilon \to 0$ in \eqref{eq:weak3} will then yield the desired weak derivative property:
\begin{align*}\label{eq:weak6}
\int_{U} v_{i,j} \varphi dV = \int_{U} v \frac{\partial^2 \varphi}{\partial z_i \partial z_j } dV.
\end{align*}

For \eqref{eq:weak4}, we estimate:
\begin{eqnarray*}\label{eq:weak8}
\left| \int_{\{z\in U: |f|=\varepsilon\}}\varphi \frac{\partial v}{\partial z_j} \frac{\partial \rho }{ \partial z_i} \frac{dS}{|\partial \rho|} \right|
&\lesssim \int_{\{z\in U: |f|=\varepsilon\}} \frac{|\partial f|}{|f| |\log |f| |  (\log(-\log |f|^2))^2} dS \nonumber \\ 
& \lesssim  \frac{1}{\varepsilon |\log \varepsilon| \, (\log(-\log \varepsilon^2))^2} \int_{\{z\in U: |f|=\varepsilon\}}  |\partial f| dS \nonumber \\ 
& \lesssim  \frac{O(\varepsilon)}{\varepsilon |\log \varepsilon| \, (\log(-\log \varepsilon^2))^2},
\end{eqnarray*}
where the last line uses part $(1)$ of Theorem \ref{t02}. This clearly tends to $0$ as $\varepsilon \to 0$.

For \eqref{eq:weak5}, we use Theorem \ref{th:levelsurface}, which gives:
\begin{eqnarray*}\label{eq:weak7}
\left| \int_{\{z\in U: |f|=\varepsilon\}} v \frac{\partial \varphi}{\partial z_i} \frac{\partial \rho }{ \partial z_j} \frac{dS}{|\partial \rho|} \right|   
& \lesssim  \frac{1}{\log(-\log \varepsilon^2)} \left| \{z\in U: |f|=\varepsilon\} \right| \nonumber \\ 
& \lesssim \frac{\varepsilon^{\gamma}}{\log(-\log \varepsilon^2)}, \quad 0 < \gamma \leq 1,
\end{eqnarray*}
which also tends to \(0\) as \(\varepsilon \rightarrow 0\). The same reasoning applies to the derivatives \(\frac{\partial^{2}v}{\partial\bar z_{i}\partial\bar z_{j}}\) and \(\frac{\partial^{2}v}{\partial z_{i}\partial\bar z_{j}}\). Consequently, $v \in W^{2,1}_{\mathrm{loc}}(U)$ near \(Z(f)\).

Finally, we show that \(v\notin W^{2,p}_{\mathrm{loc}}(U)\) for any \(p>1\) near \(Z(f)\).  We proceed by contradiction. Assume that there exists a \( p > 1 \) such that \( v \in W^{2,p}_{\mathrm{loc}}(U) \). By the definition of Sobolev spaces, this implies that all real second-order weak partial derivatives of \( v \) are locally in \( L^p(U) \). Consequently, any pure complex second-order derivative \( \frac{\partial^2 v}{\partial z_i \partial z_j} \), which is a finite linear combination of real second derivatives, must also belong to \( L^p_{\mathrm{loc}}(U) \).

Outside the analytic set \( Z(f) \), we have the exact decomposition as in (\ref{eq:A_1}). To determine the integrability of this expression, we must identify its leading singular term as \( z \) approaches \( Z(f) \). Since the regular points of \( Z(f) \) form a dense open subset of the variety, we can invoke the localization lemma to select a regular point \( p \in Z(f) \) and a sufficiently small neighborhood \( V \subset U \) around it. There exists a biholomorphic coordinate transformation \( w = \Phi(z) \) such that in these local coordinates, the zero locus \( Z(f) \cap V \) is rectified to the hyperplane \( \{w_1 = 0\} \), and the function can be expressed as \( f(z(w)) = w_1 E(w) \), where \( E(w) \) is a non-vanishing holomorphic function (i.e., \( E(0) \neq 0 \)).

Since local Sobolev regularity is invariant under biholomorphic diffeomorphisms, the function in the new coordinates, \( \tilde{v}(w) := v(z(w)) \), must satisfy \( \frac{\partial^2 \tilde{v}}{\partial w_1^2} \in L^p_{\mathrm{loc}}(\Phi(V)) \). 

We now analyze the asymptotic singularity of the terms \( A_1 \) and \( A_2 \) with respect to the normal derivative \( \partial_{w_1} \). Evaluating the derivatives of \( f \), we have \( \partial_{w_1} f = E(w) + w_1 \partial_{w_1} E(w) = O(1) \) (which is strictly non-zero near the origin since \( E(0) \neq 0 \)), and \( \partial_{w_1}^2 f = 2\partial_{w_1} E(w) + w_1 \partial_{w_1}^2 E(w) = O(1) \). Substituting these into the expressions for \( A_1 \) and \( A_2 \), we observe the following behavior as \( w_1 \to 0 \):
\[
|A_1| \approx \left| \frac{O(1)}{w_1 E(w) \log|w_1|^2 \dots} \right| = O\left( \frac{1}{|w_1| \, \bigl|\log|w_1|\bigr|} \right),
\]
\[
|A_2| \approx \left| \frac{(O(1))^2}{(w_1 E(w))^2 \log|w_1|^2 \dots} \right| = O\left( \frac{1}{|w_1|^2 \, \bigl|\log|w_1|\bigr|} \right).
\]
The term \( A_2 \) exhibits a strong pole of order \( O(|w_1|^{-2}) \), whereas \( A_1 \) only possesses a weak pole of order \( O(|w_1|^{-1}) \). Furthermore, \( A_3 \) and \( A_4 \) contain higher powers of the logarithmic damping factor in their denominators, making them strictly subdominant to \( A_2 \). Thus, \( A_2 \) is the absolute leading singular term. There exists a constant \( C > 0 \) such that, within a sufficiently small polydisc \( \Delta^n = \Delta_1 \times \Delta' \subset \Phi(V) \) where \( \Delta_1 = \{w_1 \in \mathbb{C} \mid |w_1| < r_0 < 1/2\} \), the \( p \)-th power of the second derivative satisfies the strict pointwise lower bound 
\[
\left| \frac{\partial^2 \tilde{v}}{\partial w_1^2} \right|^p \ge C \frac{1}{|w_1|^{2p} \, \bigl|\log|w_1|\bigr|^p \, \bigl|\log(-\log|w_1|^2)\bigr|^{2p}}.
\]

To test the condition \( \frac{\partial^2 \tilde{v}}{\partial w_1^2} \notin L^p_{\mathrm{loc}}(\Phi(V)) \), we integrate this lower bound over the polydisc. Applying Fubini's theorem to separate the tangential variables \( w' \), we obtain 
\[
\int_{\Phi(V)} \left| \frac{\partial^2 \tilde{v}}{\partial w_1^2} \right|^p \, dV_w \ge \text{Vol}(\Delta') \int_{\Delta_1} \frac{C}{|w_1|^{2p} \, \bigl|\log|w_1|\bigr|^p \, \bigl|\log(-\log|w_1|^2)\bigr|^{2p}} \, dV_{w_1}.
\]
Passing to polar coordinates \( w_1 = s e^{i\theta} \) on the disk \( \Delta_1 \), the inner integral yields 
\[
\int_0^{2\pi} \int_0^{r_0} \frac{C}{s^{2p} \, |\log s|^p \, \bigl|\log(-\log s^2)\bigr|^{2p}} \, s \, ds \, d\theta \approx \int_0^{r_0} \frac{1}{s^{2p-1} \, |\log s|^p \, \bigl|\log(-\log s^2)\bigr|^{2p}} \, ds.
\]

For any \( p > 1 \), the exponent of \( s \) in the denominator satisfies \( 2p - 1 > 1 \). This constitutes a strong power-law singularity at \( s=0 \). The logarithmic factors in the denominator grow far too slowly to compensate for the strict polynomial divergence of \( s^{1-2p} \). Consequently, the radial integral strictly diverges to \( \infty \).

This establishes that \( \int_{\Phi(V)} \left| \frac{\partial^2 \tilde{v}}{\partial w_1^2} \right|^p \, dV_w = \infty \), directly contradicting our initial assumption that all second-order weak derivatives are locally in \( L^p \). We conclude that \( v \notin W^{2,p}_{\mathrm{loc}}(U) \) for any \( p > 1 \).
 
\end{proof}

It is instructive that we explicitly verify the integrability of $\mathcal{A}_1$ for the cusp singularity $f(z_1, z_2) = z_1^2 - z_2^3$, illustrating the ``magic'' cancellation between the resolution Jacobian and the logarithmic damping factor $L(f)$. In particular, for this case, the Hessian $|\nabla^2 f|$ does not vanish at the origin, causing that the integrability of $\mathcal{A}_1$ fully depends on the the denominator.
\begin{example}[Integrability of $ \mathcal{A}_1$ for the cusp singularity]
Let $U$ be a neighborhood of the origin and define $f(z_1, z_2) = z_1^2 - z_2^3$. Let $Z(f) = \{(z_1,z_2)\in \C^2:f(z_1,z_2)=0\}$. Define for $z \in U \setminus Z(f)$:
\begin{equation*}
    \mathcal{A}_1(z) = \frac{|\nabla^2 f|}{|f(z)| \cdot L(f(z))}
\end{equation*}
where 
\begin{equation*}
    L(f(z)) = \left| \log |f(z)|^2 \right| \cdot \left( \log(-\log |f(z)|^2) \right)^2.
\end{equation*}
Then we have  $\mathcal{A}_1 \in L_{\text{\rm{loc}}}^1(U)$.
\end{example}

\begin{proof}
The proof proceeds via Hironaka's resolution of singularities. We blow up the origin in $\mathbb{C}^2$. The resolution space is covered by two affine charts. By symmetry, it suffices to analyze Chart 1, where we set:
\begin{equation*}
    z_1 = z, \quad z_2 = z_1 t, \quad (z_1,t)\in \C^2.
\end{equation*}
The blow-down map $\pi(z_1, t)=(z_1, z_1t).$  Its Jacobian determinant is

$$J_{\pi}=\det\left(\begin{matrix} 1 & 0\\ t&z_{1}\end{matrix}\right)=z_{1}.$$ 
The pull back of $f$ is
\begin{eqnarray*}
 F(z_{1},t): =f\circ\pi={z_{1}}^{2}-(z_{1}t)^{3}=z_{1}^{2}(1-z_{1}t^{3})
=:{z_{1}}^{2}\cdot\eta(z_{1},t)   
\end{eqnarray*}
where $\eta(z_{1}, t)=1-z_{1}t^{3}$ is a nowhere vanishing holomorphic function near $ z_{1}=0 $. The exceptional divisor is $ E=\{z_{1}=0\} $.
  
A direct calculation gives the first and second derivatives of $F$:
$$\frac{\partial F}{\partial z_1} = 2z_1 - 3z_1^2 t^3, \quad \frac{\partial F}{\partial t} = -3z_1^3 t^2$$
$$\frac{\partial^2 F}{\partial z_1^2} = 2 - 6z_1 t^3, \quad \frac{\partial^2 F}{\partial t^2} = -6z_1^3 t, \quad \frac{\partial^2 F}{\partial z_1 \partial t} = -9z_1^2 t^2$$

As $z_1 \to 0$, we have $|\nabla_w^2 F|$ tends to $2$, i.e., the Hessian of the pullback does not vanish on the exceptional divisor. Using the chain rule  as in Proposition \ref{prop:general_A1}, there exists a constant $C > 0$ such that
\begin{eqnarray}\label{eq:tiduxin}
  |\nabla_z^2 f \circ \pi| \le \frac{C}{|J_\pi|^2} \left( |\nabla_w^2 F| + \sum_{k=1}^2 |\partial_{z_k} f \circ \pi| |\nabla_w^2 \pi_k| \right)  
\end{eqnarray}
 
For the cusp, we have 
\begin{enumerate}
    \item $\partial_{z_1} f = 2z_1$, so $|\partial_{z_1} f \circ \pi| = 2|z_1|$. Moreover $\pi_1(z_1, t) = z_1$, so $|\nabla_w^2 \pi_1| = 0$.
    \item $\partial_{z_2} f = -3z_2^2$, so $|\partial_{z_2} f \circ \pi| = 3|z_1 t|^2 = 3|z_1|^2|t|^2$. The map $\pi_2(z_1, t) = z_1 t$ satisfies $|\nabla_w^2 \pi_2| \le C_0$.
\end{enumerate}
Thus the second term of the right hand side of (\ref{eq:tiduxin}) is bounded by $\mathcal{O}(|z_1|^2)$ and is negligible near $z_1 = 0$. Consequently, for sufficiently small $|z_1|$,
$$|\nabla_z^2 f \circ \pi| \le \frac{C}{|z_1|^2}(2 + \mathcal{O}(|z_1|^2)) \le \frac{C_1}{|z_1|^2}.$$

Now we look at the pointwise bound for the integrand. We have $|F| = |z_1|^2 |\eta|$ with $|\eta|$ bounded above and below near $z_1 = 0$. Hence,
$$
\frac{|\nabla_z^2 f \circ \pi|}{|F|} \le \frac{C_1}{|z_1|^2} \frac{1}{|z_1|^2 |\eta|} \le \frac{C_2}{|z_1|^4}.
$$
Multiplying by the Jacobian factor $|J_\pi|^2 = |z_1|^2$ (which comes from the change of variables $dV_z = |J_\pi|^2 dV_w$), we obtain:
$$\frac{|\nabla_z^2 f \circ \pi|}{|F|} |J_\pi|^2 \le \frac{C_2}{|z_1|^2}$$

Next we look at the asymptotics of the logarithmic damping factor $L(F)$. Since $|F| = |z_1|^2 |1 - z_1t^3|$, near the critical locus where $z=0$, we have:
$$
|F| \sim |z_1|^2, \quad \log|F|^2 = 4\log|z_1| + \log|\eta|^2 \sim 4\log|z_1|\quad \text{as}  \quad  z _1 \rightarrow 0.  
$$
So, $\big|\log|F|^2\big| \sim 4\big|\log|z_1|\big|$ as  $z _1 \rightarrow 0$. 
\begin{equation*}
-\log|F|^2 \sim -4\log|z_1| \quad  \text{as}  \quad  z _1 \rightarrow 0. 
\end{equation*}
Thus,
\begin{equation*}
\log(-\log|F|^2) \sim \log(-4\log|z_1|) \sim \log 4 + \log(-\log|z_1|).
\end{equation*}
Therefore,
\begin{equation*}
L(F) = \big|\log|F|^2\big| \cdot (\log(-\log|F|^2))^2 \sim 4\big|\log|z_1|\big| \cdot (\log|\log|z_1||)^2
\end{equation*}
Now we are ready to integrate in the resolved coordinates. In Chart 1, write $z_1 = r e^{i\theta}$ with $r = |z_1|$.
The Euclidean volume element on $X$ (in the $w$-coordinates) is
$$dV_w = r dr d\theta dV_t$$
where $dV_t$ is the Lebesgue measure in the complex $t$-plane.
Then for any compact $K \subset U$, pulling back from $U \setminus Z(f)$ via the biholomorphism (which is a diffeomorphism away from $Z(f)$) gives 
\begin{eqnarray*}
  \int_K \mathcal{A}_1(z) dV_z & = & \int_{\pi^{-1}(K)} \frac{|\nabla_z^2 f \circ \pi|}{|F| L(F)} |J_\pi|^2 dV_w  \\ \nonumber 
  & \leq &\iint \frac{C_2}{|z_1|^2} \frac{1}{4\big|\log|z_1|\big| (\log|\log|z_1||)^2} r dr d\theta dV_t\\ \nonumber 
  & \approx & \int_0^\delta \frac{C_3}{r |\log r| (\log|\log r|)^2} dr \\ \nonumber 
  &< &  \infty,  \quad \text{for } \quad  \delta < \frac{1}{2},
\end{eqnarray*}
 proving that $\mathcal{A}_1 \in L_{\text{loc}}^1(U)$.
\end{proof}

\section{Weighted singular exponent of holomorphic functions}
In this section we explore some weighted singular exponents of holomorphic functions, especially the important invariant complex singularity exponents, related to our main theorems.

\subsection{Sharp examples for \L{}ojasiewicz exponents}
In this subsection, we compute explicitly the \L{}ojasiewicz exponent $\alpha$, 
the level-set area exponent $\gamma$, and the sublevel-set volume exponent $\tau$ 
for two simple but instructive families of holomorphic functions. 
The calculations illustrate that the exponents $\gamma$ and $\tau$ 
appearing in Theorems~\ref{th:levelsurface} and~\ref{th:volume_sublevel} 
can indeed attain a wide range of values within their theoretical bounds 
$0<\gamma\le1$ and $0<\tau\le2$.

\label{subsec:monomial-example}
Let $k_1,k_2$ be positive integers. Consider the monomials  $f(z_1,z_2)=z_1^{k_1}z_2^{k_2}$ defined in a neighborhood of origin in $\mathbb C^2$. Its gradient is  
	\begin{align*}
		\partial f = \bigl(k_1 z_1^{k_1-1}z_2^{k_2},\; k_2 z_1^{k_1}z_2^{k_2-1}\bigr).
	\end{align*}
	We determine the optimal \L{}ojasiewicz exponent $\alpha$; that is, 
	the smallest $\alpha<1$ for which there exist constants $C>0$ and a neighborhood $V$ of $0$ such that  \begin{align*}
		|\partial f(z)|\ge C|f(z)|^\alpha,  \ \  \text{for any} \  z\in V.
	\end{align*}
	
	To analyse the behavior near the origin, we examine curves of the form  
	\begin{align*}
		z_1(t)=t^{a_1},\quad z_2(t)=t^{a_2}\qquad (t\to0^+),
	\end{align*}
	where $a_1,a_2\ge0$ are not both zero. Along such a curve we have  
	\begin{align*}
		|f(z(t))|=t^{\,k_1a_1+k_2a_2}\bigl(1+o(1)\bigr),
	\end{align*}
	and  
	\begin{align*}
		|\partial f(z(t))|
		= \Bigl(|k_1z_1^{k_1-1}z_2^{k_2}|^2+|k_2z_1^{k_1}z_2^{k_2-1}|^2\Bigr)^{1/2}
		\sim c\,t^{\,m_\nabla(a_1,a_2)},
	\end{align*}
	with  
	\begin{align*}
		m_\nabla(a_1,a_2) : = 
		k_1a_1+k_2a_2-\max\{ a_1, a_2 \}
	\end{align*}
	where $c>0$ is a constant depending on $a_1,a_2$.  
	For the \L{}ojasiewicz inequality to hold uniformly near $0$, we must have  
	\begin{align*}\label{eq:examplein1}
		m_\nabla(a_1,a_2)\leq\alpha\,(k_1a_1+k_2a_2)\qquad\text{for all }a_1,a_2\ge0.
	\end{align*}
	then we get
	\begin{align*}
		\alpha\,(k_1a_1+k_2a_2) \geq k_1a_1+k_2a_2-\max \{ a_1, a_2\}
	\end{align*}
	so 
	\begin{align*}
		\alpha \geq 1-\frac{\max\{ a_1, a_2\}}{k_1a_1+k_2a_2} \qquad\text{for all }a_1,a_2\ge0.
	\end{align*}
	We need to get the maximum of the right of above inequality.
	Denote $M=\max\{ a_1, a_2 \} $, then $k_1a_1+k_2a_2 \leq M(k_1+k_2)$, so
	\begin{align*}
		1-\frac{M}{k_1a_1+k_2a_2} \leq 1-\frac{M}{M(k_1+k_2)}=1-\frac{1}{k_1+k_2}. 
	\end{align*}
	On the other hand taking $a_1=1, a_2=1$, we get 
	\begin{align*}
		1-\frac{M}{k_1a_1+k_2a_2} = 1-\frac{1}{k_1+k_2}. 
	\end{align*}
	So $$\alpha \geq 1-\frac{1}{k_1+k_2}$$
	
	To see that this value is actually attainable, we examine the ``worst'' direction: Take $a_1=1$, $a_2=1$. Then $|f|\sim t^{k_1+k_2}$ and $|\partial f|\sim t^{k_1+k_2-1}$, so $|\partial f|\sim |f|^{1-\frac{1}{(k_1+k_2)}}$. Consequently, the optimal \L{}ojasiewicz exponent for the monomial is  
	
	\begin{align*} 
		\alpha=1-\frac{1}{k_1+k_2}. 
	\end{align*}
    
	\begin{corollary}\label{cor:monomial-exponents}
		For $f(z_1,z_2)=z_1^{k_1}z_2^{k_2}$, the exponents $\gamma$ and $\tau$ 
		defined in Theorem ~\ref{th:levelsurface} and~ Theorem \ref{th:volume_sublevel} are  
		\begin{align*} 
			\gamma=1-\alpha=\frac1{k_1+k_2},\qquad
			\tau=2-2\alpha=\frac2{k_1+k_2} .
		\end{align*}
		In particular, by choosing $k_1$ and $k_2$ appropriately we obtain examples where 
		$\gamma$ takes any value of the form $1,\frac12,\frac13,\dots$ 
		(tending to $0$ as $\min\{k_1,k_2\}\to\infty$), 
		and $\tau$ takes any value of the form $2,1,\frac23,\dots$ (tending to $0$ as well).
	\end{corollary}


 Let $p,q$ be positive integers with $2 \leq p \leq q$. Consider the holomorphic function  $f(z_1,z_2)=z_1^{p}-z_2^{q}$ defined in a neighborhood of origin in $\mathbb C^2$. Its gradient is  
	\begin{align*}
		\partial f=\bigl(p z_1^{p-1},\, -q z_2^{q-1}\bigr).
	\end{align*}
First we analyse the lower bound of $\alpha$. We examine the behaviour of $f$ along smooth curves of the form$$z_1=t^{a_1},\qquad z_2=t^{a_2}.$$
	where $a_1,a_2 > 0, a_1p \neq a_2q .$ Along such a curve we have 
	\begin{equation*}
		|f(z(t))|
		=\bigl|t^{a_1p}-t^{a_2q}\bigr|
		\sim t^{\min(a_1p,\;a_2q)},
	\end{equation*} 
	and  
	\begin{align*}
		|\partial f(z(t))|
		=\sqrt{p^2\,t^{2a_1(p-1)}+q^2\,t^{2a_2(q-1)}}
		\sim t^{m_\nabla(a_1,a_2)}.
	\end{align*}
	with  
	\begin{align*}
		m_\nabla(a_1,a_2) : = 
		\min\bigl\{a_1(p-1),\;a_2(q-1)\bigr\}.
	\end{align*} 
	
	For the \L{}ojasiewicz inequality to hold uniformly near $0$, it is necessary that  
	
	\begin{align*}\label{eq:examplein1}
		m_\nabla(a_1,a_2)\leq\alpha\,\min\{a_1p,\;a_2q\}\qquad\text{for all }a_1,a_2 > 0, 
	\end{align*}
We now examine two cases.\\

	\noindent {Case 1.}  If $q a_2 < p a_1$, then $\min\{p a_1,q a_2\}=q a_2$. Since
	\begin{equation*}
		m_{\nabla}(a_1,a_2):=\min\big\{a_1(p-1),\,a_2(q-1)\big\}\le a_2(q-1),
	\end{equation*}
	we obtain
	\begin{equation*}
		\frac{m_{\nabla}(a_1,a_2)}{\min\{p a_1,q a_2\}}
		\le
		\frac{(q-1)a_2}{q a_2}
		=
		1-\frac{1}{q}.
	\end{equation*}
	From the necessary condition
	\begin{equation*}
		m_{\nabla}(a_1,a_2)\le \alpha\,\min\{p a_1,q a_2\},
	\end{equation*}
	it follows that $\alpha \ge m_{\nabla}(a_1,a_2)/\min\{p a_1,q a_2\},$ for all $a_1,a_2 > 0$. Hence, in this case it suffices to require
	\begin{equation*}
		\alpha \ge 1-\frac{1}{q}.
	\end{equation*}
	
	\medskip

\noindent {Case 2.}  If $p a_1 > q a_2$, then $\min\{p a_1,q a_2\}=p a_1$. Similarly,
	\begin{equation*}
		m_{\nabla}(a_1,a_2)\le a_1(p-1),
	\end{equation*}
	and therefore
	\begin{equation*}
		\frac{m_{\nabla}(a_1,a_2)}{\min\{p a_1,q a_2\}}
		\le
		\frac{(p-1)a_1}{p a_1}
		=
		1-\frac{1}{p},
	\end{equation*}
	Hence, Hence, in this case it suffices to require
	\begin{equation*}
		\alpha \ge 1-\frac{1}{p}.
	\end{equation*}
	
	By assumption $p \leq q$, we know that $1 - \frac{1}{p} \leq 1 - \frac{1}{q}.$ 
	Hence, it suffices to require $\alpha \geq 1 - \frac{1}{q}.$

	According to the above calculation, the optimal \L{}ojasiewicz exponent for  $f(z_1,z_2)=z_1^{p}- z_2^{q}$ is  
	 $\alpha = 1-\frac1q\ .$
\begin{corollary}\label{cor:binomial-exponents}
For $f(z_1,z_2)=z_1^{\,p}-z_2^{\,q}$ with $p\le q$,  
\[
\gamma=1-\alpha=\frac1q,\qquad \tau=2-2\alpha=\frac2q .
\]
Thus, by choosing $q$ large we obtain examples where $\gamma$ and $\tau$ are arbitrarily close to $0$, 
while for $q=1$ (and necessarily $p=1$) we recover the smooth case $\gamma=1$, $\tau=2$.
\end{corollary}


The following table collects the results obtained above.
\[
\begin{array}{c|c|c|c}
\text{Function }f(z_1,z_2) & \alpha\;(\text{\L{}ojasiewicz}) & \gamma=1-\alpha & \tau=2-2\alpha\\ \hline
z_1^{k_1}z_2^{k_2} &1-\frac{1}{k_1+k_2} & \frac{1}{k_1+k_2}&  \frac{2}{k_1+k_2} \\[8pt]
z_1^{\,p}-z_2^{\,q},\;p\le q & 1-\dfrac1q & \dfrac1q & \dfrac2q
\end{array}
\]
Both families demonstrate that the exponents $\gamma$ and $\tau$ appearing in the general theorems 
can indeed attain a continuum of values within the predicted ranges $(0,1]$ and $(0,2]$, respectively. 
In particular, they can be arbitrarily close to $0$ when the zero set of $f$ is sufficiently 
singular at the origin.

\subsection {Integrability and complex singularity exponents}

Let $U$ be a domain in $\C^n$ and $\varphi$ be a plurisubharmonic function on $U$. We say $\varphi$ has singular at $z$ if $\varphi(z) =-\infty$.  A fundamental holomorphic invariant is its \textbf{complex singularity exponent} (also called the \emph{log canonical threshold} in algebraic geometry), defined at a point $z$ as
\[
c_z(\varphi):=\sup\bigl\{\,c>0:e^{-2c\varphi} \ \ \text{is} \ \  L^1_{\mathrm{loc}}\ \ \text{near} \ \   z  \,\bigr\}.
\]
This exponent quantifies the ``strength'' of the singularity of $\varphi$ at $z$.  
For the special but ubiquitous case $\varphi=\log|f|$, where $f$ is holomorphic, this analytic invariant has a precise geometric interpretation.  
Indeed, the condition $e^{-2c\varphi}=|f|^{-2c} $ is $L^1_{\mathrm{loc}}$ near the zero set of $f$, is exactly the local integrability of an inverse power of $|f|$.  
It therefore connects directly to the \L{}ojasiewicz--type exponent $\alpha$ introduced earlier, which satisfies $|\partial f|\gtrsim|f|^{\alpha}$, $0\leq \alpha <1$, near the zero set of $f$.  

Our previous geometric results---Theorem~\ref{tC}, area bounds for the level sets $\{|f|=t\}$  and, volume estimates for the sublevel sets $\{|f|\le t\}$, provide the quantitative link that makes this connection explicit.  
In particular, the volume estimate
\[
\bigl|\{\,z\in U:|f|\le t\,\}\bigr|=O(t^{2-2\alpha})
\]
is, in essence, a geometric characterization of the complex singularity exponent $c_0(\log|f|)$.

Applying the coarea formula together with these geometric estimates, we obtain sharp integrability thresholds for several differential forms naturally attached to $f$.  
The following proposition unifies four interrelated statements; it shows how the exponent $\alpha$ governs the local integrability of the logarithmic derivative $\partial f/f$, of the weighted gradient $|\partial f|^2/|f|^\delta$, and of the pure inverse power $1/|f|^\delta$.

\begin{proposition}\label{th:maintheorem22}
Let $U$ be a domain in $\mathbb{C}^n$ and let $f$ be a nonconstant holomorphic function on $U$ with $f(0)=0$.  
Denote by $\alpha\;(0\le\alpha<1)$ the optimal exponent satisfying $|\partial f|\gtrsim|f|^{\alpha}$ near~$0$.  
Then the following hold locally near the origin:
\begin{enumerate}
    \item $\displaystyle\frac{\partial f}{f}\in L^p_{\mathrm{loc}}$ for every $p<2$, but 
          $\displaystyle\frac{\partial f}{f}\notin L^2_{\mathrm{loc}}$;
    \item $\displaystyle\frac{|\partial f|^2}{|f|^{\delta}}\in L^1_{\mathrm{loc}}$ for every $\delta<2$;
    \item $\displaystyle\frac{1}{|f|^{\delta}}\in L^1_{\mathrm{loc}}$ for every $\delta<2(1-\alpha)$;
    \item  \[
\int_{\{z \in U : |f| < \varepsilon\}} |\partial f(z)|^p \, dV(z) = O(\varepsilon^{\gamma(p)}),
\]
where $\gamma(p)=2(1-\alpha)+\alpha p > 0$ is a constant depending on $p$ and the \L{}ojasiewicz exponent $\alpha$.  In particular, for $\varphi=\log|f|$ we have the inequality
\[
c_0(\log|f|)\geq 1-\alpha.
\]
\end{enumerate}
\end{proposition}

\begin{proof}
First we prove $(1)$ of Proposition  \ref{th:maintheorem22}.
We begin by proving the local integrability for $p<2$. First, consider the case $1 < p < 2$. Let $V$ be a sufficiently small neighborhood of $0$. Applying the coarea formula to the function $\left| \frac{\partial f}{f} \right|^p$, we can find some small sufficient $\varepsilon_0$ such that
\begin{align}\label{eq:coarea12}
    \int_V \left| \frac{\partial f}{f} \right|^p dV &= \int_0^{\varepsilon_0} \left( \int_{\{z \in V : |f| = t\} } \left| \frac{\partial f}{f} \right|^p \frac{dS}{|\nabla |f||} \right) dt \\ \nonumber 
    & = \int_0^{\varepsilon_0} \left( \int_{\{z \in V : |f| = t\}} \left| \frac{\partial f}{f} \right|^p \frac{dS}{|\partial f|} \right) dt \\ \nonumber 
\end{align}
Since $|\partial f| \gtrsim |f|^{\alpha}$ near the zero set for some $0 \le \alpha < 1$, in view of Lemma \ref{l2204}, the inner integral of the last term in \eqref{eq:coarea12} can be bounded by
\[
\int_{\{ |f| = t\}} \left| \frac{\partial f}{f} \right|^p \frac{1}{ |f|^{\alpha} } \, dS \lesssim t^{-p-\alpha} \cdot \mathcal{H}^{2n-1}(\{ |f| = t \} \cap V).
\]
Moreover, by Theorem \ref{tC} we have $\mathcal{H}^{2n-1}(\{ |f| = t \} \cap V) = O(1)$ as $t \to 0^+$. Consequently,
\[
\int_V \left| \frac{\partial f}{f} \right|^p \, dV \lesssim \int_0^{\varepsilon_0} t^{-(p+\alpha)} \, dt < \infty,
\]
because $1< p + \alpha < 2 + \alpha < 3$. For the case $0 < p \le 1$, the result follows from the case $1 < p < 2$ by applying Hölder's inequality. Hence, $\frac{\partial f}{f} \in L^p_{\mathrm{loc}}$ near $0$ for every $p < 2$.\\
 
To prove the failure of local $L^2$-integrability, we use the similar idea as Theorem \ref{l10}, and get
\begin{align*}
    \int_{V} \left| \frac{\partial f}{f} \right|^2 \, dV_z 
    &\gtrsim \int_{f(V)} \frac{1}{|w_1|^2} \, dV_{w_1} \\
    &\gtrsim \int_0^{r_0} \frac{1}{s^2} \cdot s \, ds, \quad \text{ by integration in polar coordinates in the $w_n$-variable,}\\
    &\gtrsim  \int_0^{r_0} \frac{1}{s} \, ds = \infty.
\end{align*}
Therefore, $\frac{\partial f}{f} \notin L^2_{\mathrm{loc}}$ near $Z(f)$.\\
\bigskip

Now we establish the local integrability of $\frac{|\partial f|^2}{|f|^{\delta}}$ for $\delta < 2$. Decompose a sufficiently small neighborhood of $0$ into annular regions:
\begin{align*}
    \int_{U} \frac{|\partial f|^2}{|f|^{\delta}} \, dV 
    &= \sum_{j=1}^{\infty} \int_{\frac{1}{2^j} \le |f| \le \frac{1}{2^{j-1}}} \frac{|\partial f|^2}{|f|^{\delta}} \, dV \\
    &\le \sum_{j=1}^{\infty} 2^{j\delta} \int_{|f| \le \frac{1}{2^{j-1}}} |\partial f|^2 \, dV.
\end{align*}
By applying statement (2) of Theorem~\ref{t02}, which provides an estimate of the form 
$$
\int_{\{ z\in U: |f| \le t\} } |\partial f|^2 \, dV = O(t^2),  \text{ as} \ \   t \to 0,
$$
we obtain
\[
\int_{U} \frac{|\partial f|^2}{|f|^{\delta}} \, dV \lesssim \sum_{j=1}^{\infty} 2^{j\delta} \cdot 2^{-2(j-1)}\lesssim \frac{1}{4} \sum_{j=1}^{\infty} 2^{-j(2-\delta)}.
\]
The series converges if and only if $\delta < 2$, which completes the proof.\\

Next,  we show $(3)$ of Proposition  \ref{th:maintheorem22}.
By Theorem  \ref{tC}, the \L{}ojasiewicz inequality together with coarea formula implies the following volume estimate for the sublevel sets of $|f|$:
\begin{equation}\label{eq:volume_est}
\bigl|\{ z \in U : |f| < \varepsilon \}\bigr| = O(\varepsilon^{\tau}), \qquad \tau = 2-2\alpha,
\end{equation}
where $U$ is a sufficiently small polydisc centered at the origin.
Let $K \subset U$ be any compact neighborhood of $0$. It suffices to show that
\[
\int_{K} \frac{1}{|f(z)|^{\delta}}\, dV(z) < \infty.
\]
Without loss of generality, we may assume $K \subset \{ z \in U : |f(z)| \le r_0 \}$ for some small $r_0>0$ such that  \eqref{eq:volume_est} holds on $K$.
We decompose the set $\{ z \in K : 0 < |f(z)| < r_0 \}$ into dyadic annuli:
\[
E_j := \Bigl\{ z \in K : \frac{1}{2^{j}} \le |f(z)| < \frac{1}{2^{j-1}} \Bigr\}, \qquad j = 1, 2, 3, \dots .
\]
Then $\{ z \in K : 0 < |f(z)| < r_0 \} \subset \bigcup_{j=1}^{\infty} E_j$. For $z \in E_j$, we have $|f(z)|^{-\delta} \le 2^{\,j\delta}$. Hence,
\[
\int_{E_j} \frac{1}{|f(z)|^{\delta}}\, dV(z) \le 2^{\,j\delta} \cdot |E_j|,
\]
where $|E_j|$ denotes the Euclidean volume of $E_j$.

Since $E_j \subset \{ z \in U : |f(z)| < 2^{-(j-1)} \}$, the volume estimate \eqref{eq:volume_est} yields
\[
|E_j| \le \bigl|\{ z \in U : |f(z)| < 2^{-(j-1)} \}\bigr| \le C_1 \bigl(2^{-(j-1)}\bigr)^{\tau} = C_2 \cdot 2^{-j\tau},
\]
for some constants $C_1, C_2 > 0$ independent of $j$, and with $\tau = 2-2\alpha$.
Therefore,
\[
\int_{E_j} \frac{1}{|f(z)|^{\delta}}\, dV(z) \le C_2 \cdot 2^{-j(\tau - \delta)}.
\]
Summing over all $j \ge 1$, we obtain
\[
\int_{\bigcup_{j=1}^{\infty} E_j} \frac{1}{|f(z)|^{\delta}}\, dV(z) \le C_2 \sum_{j=1}^{\infty} 2^{-j(\tau - \delta)}.
\]

The geometric series on the right-hand side converges if and only if $\tau - \delta > 0$, i.e., $\delta < \tau = 2(1-\alpha)$. Under this condition, we conclude
\[
\int_{K \cap \{ 0 < |f| < r_0 \}} \frac{1}{|f(z)|^{\delta}}\, dV(z) < \infty.
\]

The zero set of $f$, being an analytic set, has Lebesgue measure zero and does not affect the integrability. Consequently,
\[
\int_{K} \frac{1}{|f(z)|^{\delta}}\, dV(z) < \infty,
\]
which proves $|f|^{-\delta} \in L^1_{\mathrm{loc}}$ near the origin for every $\delta < 2(1-\alpha)$.

We now estimate the integral  
\[
J_p(\varepsilon):=\int_{\{z\in U:|f|<\varepsilon\}}|\partial f(z)|^{p}\,dV(z)
\]
for \(p<2\), where \(U=\{z\in\mathbb{C}^{n}:|z_{j}|<1,\;j=1,\dots ,n\}\) is the unit polydisc.  
Applying Hölder's inequality with exponents \(\frac{1}{s}+\frac{1}{t}=1\) and choosing \(s\) so that \(sp=2\) (hence \(s=\frac{2}{p}\) and \(t=\frac{2}{2-p}\)), we obtain  
\begin{align*}
J_{p}(\varepsilon)
&\le\Bigl(\int_{\{z\in U:|f(z)|<\varepsilon\}}|\partial f(z)|^{ps}\,dV(z)\Bigr)^{\frac1s}
\Bigl(\int_{\{z\in U:|f(z)|<\varepsilon\}}dV(z)\Bigr)^{\frac1t} \nonumber\\[4pt]
&=\Bigl(\int_{\{z\in U:|f(z)|<\varepsilon\}}|\partial f(z)|^{2}\,dV(z)\Bigr)^{\frac{p}{2}}
\;|\{z\in U:|f(z)|<\varepsilon\}|^{\frac{2-p}{2}} \nonumber\\[4pt]
&\lesssim \bigl(\varepsilon^{2}\bigr)^{\frac{p}{2}}
\Bigl(\varepsilon^{2-2\alpha}\Bigr)^{\frac{2-p}{2}} \label{eq:Jp-bound}\\[4pt]
&\lesssim \varepsilon^{\,p+(1-\alpha)(2-p)}. \nonumber
\end{align*}

In the last two lines we have used the a priori estimate
\[
\int_{\{z\in U:|f(z)|<\varepsilon\}}|\partial f(z)|^{2}\,dV(z)\lesssim\varepsilon^{2},
\]
which follows from the  \(L^{2}\) energy estimate of the sublevel set, together with the volume estimate
\[
|\{z\in U:|f|<\varepsilon\}|=O\bigl(\varepsilon^{2-2\alpha}\bigr)
\]
obtained in \eqref{eq:volume-est}. Thus for every \(p<2\) and every \(\alpha\in[0,1)\),
\[
J_{p}(\varepsilon)=O\!\left(\varepsilon^{\,p+(1-\alpha)(2-p)}\right)\qquad(\varepsilon\to0^{+}).
\]
This completes the proof of statement (4).
\end{proof}

\begin{remark}
The results for $(1)$ and $(2)$ are sharp; for example, the function $f(z)=z_1$ gives $\alpha=0$, attaining the bounds in $(1)$ and $(2)$.

\end{remark}

  



\appendix
\section{Volume of the graph of a holomorphic function}\label{C}
In this section the volume of the graph of a holomorphic function is given.
\begin{proposition}\label{p141}
Let \(\psi : D \subset \mathbb{C}^{n-1} \rightarrow \mathbb{C}\) be holomorphic. Then the volume of its graph is given by
\[
\operatorname{Vol}(\operatorname{graph} \psi) = \int_{D} \left(1 + |\nabla_{z'} \psi(z')|^2 \right) dV_{\mathbb{C}^{n-1}}(z').
\]
\end{proposition}
\begin{proof}
Let $z'= (z_1, z_2, \cdots, z_{n-1})$ with $z_j = x_j + i y_j$ and 
w rite \(\psi = u + i v\). The graph is the embedding
\[
\Phi : D \rightarrow \mathbb{C}^n, \quad z' \mapsto (z', \psi(z')).
\]
The Euclidean metric on \(\mathbb{C}^n \cong \mathbb{R}^{2n}\) is
\[
ds^2 = \sum_{j=1}^{n-1} (dx_j^2 + dy_j^2) + du^2 + dv^2.
\]
Because \(\psi\) is holomorphic, its real and imaginary parts satisfy the Cauchy-Riemann equations
\[
\frac{\partial u}{\partial x_j} = \frac{\partial v}{\partial y_j}, \quad \frac{\partial u}{\partial y_j} = -\frac{\partial v}{\partial x_j}, \quad j=1,\dots,n-1.
\]
The Jacobian matrix of \(\Phi\) with respect to the real coordinates \((x_1, y_1, \dots, x_{n-1}, y_{n-1})\) is a \(2n \times (2n-2)\) matrix. The induced metric on the graph is
\[
g_{ij} = \delta_{ij} + \frac{\partial u}{\partial x_i} \frac{\partial u}{\partial x_j} + \frac{\partial v}{\partial x_i} \frac{\partial v}{\partial x_j}, \quad i,j=1,\dots,2n-2.
\]
A direct computation gives
\[
|\nabla u|^2 = |\nabla v|^2, \quad \nabla u \cdot \nabla v = 0,
\]
where \(\nabla\) denotes the gradient with respect to \((x_1, y_1, \dots, x_{n-1}, y_{n-1})\).
Consequently, the $2\times 2$ matrix \(I_{2} + J J^T\) (where \(J\) is the Jacobian of \((u,v)\)) has determinant
\[
\det(I_{2} + J J^T) = (1 + |\nabla u|^2)^2.
\]
Since \(\det g = \det(I_{2n-2} + J^T J)=\det(I_{2} + J J^T)\) by Schur complement identity, we obtain
\[
\sqrt{\det g} = 1 + |\nabla u|^2.
\]
It remains to express \(|\nabla u|^2\) in complex terms. Using again the Cauchy-Riemann equations
\[
|\nabla u|^2 = \sum_{j=1}^{n-1} \left( \left|\frac{\partial \psi}{\partial z_j}\right|^2 +\left |\frac{\partial \psi}{\partial \bar z_j}\right|^2 \right) = \sum_{j=1}^{n-1} \left| \frac{\partial \psi}{\partial z_j} \right|^2.
\]
Thus,
\[
\sqrt{\det g} = 1 + |\nabla_{z'} \psi|^2.
\]
The volume element on the graph in the coordinates \(z'\) is therefore
\[
dV_{\text{graph}} = \left(1 + |\nabla_{z'} \psi|^2 \right) dV_{\mathbb{C}^{n-1}}(z'),
\]
where \(dV_{\mathbb{C}^{n-1}}(z') = \prod_{j=1}^{n-1} dx_j dy_j\) is the standard Lebesgue measure on \(\mathbb{C}^{n-1}\).
The proof is done.
\end{proof}

\section{Hironaka's theorem in resolution of singularities}\label{Hi}

In this section, we discuss the possible technical aspects by applying Hironaka's theorem.

Let $ f: U \to \mathbb{C} $ be a non-zero holomorphic function in a neighborhood of $ U $, the unit polydisc. Let $ Z (f)= \{z \in U: f(z) = 0\} $ be the complex variety defined by $ f $. The Hironaka's theorem guarantees that there exists a complex manifold $ X $ and a holomorphic map $ \pi : X \to U $ satisfying the following properties:

\begin{itemize}
\item The map $ \pi $ is proper. Namely that $ \pi $ is proper means the preimage of any compact set in $ U $ is a compact set in $ X $. This is the most crucial property for analysis. It guarantees that if our integral is confined to compact neighborhood $ K \subset U $ of the singularity, then the entire problem on the resolved space is confined to the compact set $ \pi^{-1}(K) $. This allows us to use a finite cover by coordinate charts, and partition of unity, which is essential for reducing a global problem to a finite number of local ones.

\item The map $ \pi $ is a birational morphism. In the context of manifolds, this means it establishes a biholomorphism between dense open subsets. Specifically, $ \pi $ is a (global) biholomorphism from $ X \setminus \pi^{-1}(Z(f)) $ to $ U \setminus Z(f) $. This ensures that the geometry far away from the singularity is unchanged. Our change of variables is non-trivial only in the neighborhood of the singular set, which is exactly where we need it. For any $ \varepsilon > 0 $, the level set $ \{z \in U : |f(z)| = \varepsilon\} $ lies entirely in $ U \setminus Z(f) $, so the map $\pi$ provides a one-to-one correspondence for the domains of our surface integrals.

\item The total space $X$ is a smooth complex manifold, and the preimage of the singular set, $E = \pi^{-1}(Z(f))$, is now a smooth sub-manifold with normal crossings inside the smooth space $X$. This means that locally, $E$ looks like the union of coordinate hyperplanes, a very simple and well-behaved geometric objects. This is the power of the resolution: it smooths out the object of study itself, allowing us to use the full power of calculus and differential geometry on the manifold $X$ without worrying about the pathological geometry of the original singular set $Z(f)$.

\item The local nature of the pullback function $\pi^*(f)$.  
   For the pullback $F: F = f \circ \pi$, there is a local monomial form. For every point $p \in X$, there exists a local coordinate chart $(V, u)$ centered at $p$ with coordinates $u = (u_1, \ldots, u_n)$. Such that on $V$, $F$ can be written in the form  
   $$F(u) =  \eta(u)u_1^{a_1} u_2^{a_2} \cdots u_n^{a_n}$$
   where $A = (a_1 \cdots a_n)$ is a multi-index of non-negative integers, and $\eta(u)$ is a holomorphic and non-vanishing function on the chart $V$.  
\end{itemize}

The set where at least $a_i > 0$ forms the exceptional divisor in that chart. The condition that the divisor has normal crossings means that locally we can always find coordinates where it is defined by the vanishing of some coordinate functions(i.e. $\{u_i = 0\}$ for some $i$).

As application to our integrals, this is the engine of the simplification . It tells us that, locally, any complicated holomorphic function can be transformed into a simple monomial, which we know how to analyze, time a "benign" invertible factor $ \eta(u) $. The fact that $ \eta(u) $ is non-vanishing is critical, as it allows us to define a further local biholomorphism $ v_i = u_i\eta(u)^{1/a_i} $ to absorb it completely, reducing the problem to the pure monomial case $ v^A = v_1^{a_1} \cdots v_n^{a_n} $.
This local simplification is the key to calculating the asymptotic behavior of the integral. We like to mention that employing Hironaka's theorem to construct Sobolev functions was carried in \cite{PZ} for at least real analytic functions.


\section{Examples of the resolved space}\label{E}
In order to help understand Hironaka's theorem, we compute the resolution space of the theorem. The first example includes the singular case: $f = z_1^2 - z_2^3$ in $\mathbb{C}^2$ restricted to the unit polydisc $U = \{ (z_1,z_2)\in\mathbb{C}^2:|z_1| < 1, |z_2| < 1 \}$.

First we locate the singular locus of the variety
\[
V = \{ (z_1, z_2) \in U : z_1^2 - z_2^3 = 0 \}.
\]
A point is singular if the gradient of the defining function vanishes.
Since $\partial f = (2z_1, -3z_2^2)$, the only point where the gradient is zero is the origin $(0, 0)$. This point also lies on the curve, so the singular locus $S$ is the single point $S = \{ (0, 0) \}$.
Since the singular locus is a smooth submanifold (a point is a $0$-dimensional manifold), Hironaka's theorem tells us to resolve the singularity by blowing up the ambient space $U$ at the point $S$, where $\pi : X \to U$.

Now let's compute the formula for $X$ and $\pi$.
The blow-up of $U \subset \mathbb{C}^2$ at the origin is a submanifold of the product space $U \times \mathbb{P}^1$. Let $(z_1, z_2)$ be the coordinates on $U$ and $[w_1 : w_2]$ be the homogeneous coordinates for $\mathbb{P}^1$. The resolved space $X$, is defined as the set of points $((z_1, z_2), [w_1 : w_2])$ in $U \times \mathbb{P}^1$ that satisfy the relation:
\[
z_1 w_2 = z_2 w_1.
\]
As a formal definition is
\[
X = \{ ((z_1, z_2), [w_1 : w_2]) \in U \times \mathbb{P}^1 : z_1 w_2 = z_2 w_1 \}.
\]
The resolution map $\pi$:
The map $\pi : X \to U$ is the natural projection onto the $U$ factor:
\[
\pi((z_1, z_2), [w_1 : w_2]) = (z_1, z_2).
\]

To see how the singularity is resolved, we examine $X$ in the two standard affine charts of 
$\mathbb{P}^1$.

\textbf{Chart 1:} $w_1 \neq 0$. We can set $w_1 = 1$, making $t = \frac{w_2}{w_1}$ our local coordinate. The defining equation for $X$ becomes $z_1 t = z_2$. We can use $(z_1, t)$ as local coordinates for this part of $X$ and  $\pi(z_1, t) = (z_1, z_1 t)$.

Now, let's lift the original curve's equation $z_1^2 - z_2^3 = 0$ into these coordinates by substituting $z_2 = z_1 t$:
\[
z_1^2 - (z_1 t)^3 = 0 \;\Longleftrightarrow\; z_1^2 (1 - z_1 t^3) = 0.
\]
This equation for the ``total transform'' in $X$ has two parts:
\begin{enumerate}
	\item $z_1^2 = 0$: This corresponds to the Exceptional Divisor $E = \pi^{-1}(0,0)$, which is the line $z_1 = 0$ in this chart. The exponent $2$ indicates that the original curve was singular with multiplicity $2$ at the origin.
	\item $1 - z_1 t^3 = 0$: This is the strict transform $\tilde{V}$, which is properly resolved version of our original curve. This curve is smooth.
\end{enumerate}
\textbf{Chart 2:} $w_2 \neq 0$. Here we set $w_2 = 1$ and use $s = \frac{w_1}{w_2}$ as our coordinate. The defining equation for $X$ is $z_1 = z_2 s$. We use $(z_2, s)$ as local coordinates and  $\pi(z_2, s) = (z_2 s, z_2)$.

Now, lift the curve's equation $z_1^2 - z_2^3 = 0$ using $z_1 = z_2 s$:
\[
(z_2 s)^2 - z_2^3 = 0 \;\Longleftrightarrow\; z_2^2 (s^2 - z_2) = 0.
\]
Again, we have two parts:
\begin{enumerate}
	\item $z_2^2 = 0$: This is the Exceptional Divisor $E$ in these coordinates. $E=\pi^{-1}(0,0)$.
	\item $s^2 - z_2 = 0$: This is the strict transform $\tilde{V}$. The equation $z_2 = s^2$ defines a smooth parabola in the $(s, z_2)$-plane.
\end{enumerate}

After all, the summary of the resolution:
\begin{itemize}
	\item $X$ is the manifold obtained by blowing up the unit polydisc $U \subset \mathbb{C}^2$ at the origin $(0,0)$. It can be explicitly written as
	\[
	X = \{ ((z_1, z_2), [w_1 : w_2]) \in U \times \mathbb{P}^1 : z_1 w_2 = z_2 w_1 \}.
	\]
	\item $\pi$ is the projection map from $X$ back to $U$.
\end{itemize} 	
The process resolves the singularity as follows: 
\begin{enumerate}
	\item The singular point $(0,0)$ in $U$ is replaced in $X$ by a whole projective line $\mathbb{P}^1$, called the exceptional divisor $E$.
	\item The original singular curve $V$ is transformed into a new smooth curve $\tilde{V}$ in $X$ called the strict transform. In local coordinates, its equation is $z_2 = s^2$, which is smooth.
	\item The new smooth curve $\tilde{V}$ intersects the exceptional divisor $E$ at a single point (where $s = 0$ and $z_2 = 0$). This intersection is transverse and the singularity is resolved.
\end{enumerate}

\section*{Acknowledgements}

\thanks{This study was supported by National Key R \& D Program of China,  No. 2024YFA1015200.
Guokuan Shao was supported by National Natural Science Foundation of China No. 12471082, and also acknowledged the support from the Labex CEMPI (ANR-11-LABX-0007-01) \& the project QuaSiDy (ANR-21-CE40-0016) during his visit in Lille. Jianfei Wang was supported by the National Natural Science Foundation of China No. 12571087 and the Natural Science Foundation of Fujian Province No. 2024J01087.  Jujie Wu was supported by the Natural Science Foundation of Guangdong Province, No. 2025A1515011428.}

\bibliographystyle{plain}

\bigskip

\end{document}